\documentclass[11pt, a4paper]{article}
\usepackage[utf8]{inputenc}
\usepackage{amsfonts}
\usepackage{appendix}
\usepackage{amsmath}
\usepackage{dsfont}
\usepackage{amssymb}
\usepackage{graphicx}
\usepackage[colorlinks=true, linkcolor=blue, citecolor=blue, urlcolor=blue]{hyperref}
\usepackage{array}
\usepackage{amsthm}
\usepackage{xcolor}
\usepackage{enumitem}

\newcolumntype{R}[1]{>{\raggedleft\arraybackslash }b{#1}}
\newcolumntype{L}[1]{>{\raggedright\arraybackslash }b{#1}}
\newcolumntype{C}[1]{>{\centering\arraybackslash }b{#1}}
\newtheorem{Def}{Definition}
\newtheorem{thm}{Theorem}
\newtheorem{prop}{Proposition}
\newtheorem{lem}{Lemma}
\newtheorem{cor}{Corollary}
\newtheorem{rem}{Remark}

\newcommand{\R}{\mathbb{R}}
\newcommand{\Z}{\mathbb{Z}}
\newcommand{\LL}{\mathcal{L}}
\newcommand{\N}{\mathbb{N}}

\newcommand{\e}{\eta}

\newcommand{\rmd}{\mathrm{d}}
\newcommand{\ddt}{\frac{\rmd}{\rmd t}}
\newcommand{\exval}[1]{\mbox{$\langle \, {#1}\, \rangle$}}

\newcommand{\ba}{\begin{array}}
\newcommand{\ea}{\end{array}}
\newcommand{\bea}{\begin{eqnarray}}
\newcommand{\eea}{\end{eqnarray}}

\DeclareFontFamily{U}{BOONDOX-calo}{\skewchar\font=45 }
\DeclareFontShape{U}{BOONDOX-calo}{m}{n}{
  <-> s*[1.05] BOONDOX-r-calo}{}
\DeclareFontShape{U}{BOONDOX-calo}{b}{n}{
  <-> s*[1.05] BOONDOX-b-calo}{}
\DeclareMathAlphabet{\mathcalboondox}{U}{BOONDOX-calo}{m}{n}
\SetMathAlphabet{\mathcalboondox}{bold}{U}{BOONDOX-calo}{b}{n}
\DeclareMathAlphabet{\mathbcalboondox}{U}{BOONDOX-calo}{b}{n}
\newcommand{\mcb}[1]{{\mathcalboondox #1}}
\usepackage{authblk}

\numberwithin{equation}{section}

\usepackage{geometry} 
\title{Duality for some models of epidemic spreading}

  \author[1]{C. Franceschini}
    \author[2]{E. Saada}
    \author[3]{G. M. Sch\"utz}
     \author[4]{S. Velasco}
     
    \affil[1]{\small{Dipartimento di Scienze Fisiche, Informatiche e Matematiche, Universit\`a degli Studi di Modena e Reggio Emilia, Via G. Campi 213/b 41125, Italy}}
    \affil[2]{\small{CNRS, UMR 8145, Laboratoire MAP5, Universit\'e Paris Cit\'e, 45 rue des Saints-P\`eres, 75270 Paris Cedex 06, France}}
    \affil[3]{\small{Centro de An\'alise Matem\'atica, Geometria e Sistemas Din\^amicos, Departamento de Matem\'atica, Instituto Superior T\'ecnico, Universidade de Lisboa, Av. Rovisco Pais 1, 1049-001 Lisboa, Portugal}}
    \affil[4]{\small{Laboratoire MAP5, Universit\'e Paris Cit\'e, 45 rue des Saints-P\`eres, 75270 Paris Cedex 06, France}}

\begin{document}

\maketitle

\begin{abstract} 
We examine the role of boundaries and the structure of nontrivial duality functions
for three non-conservative interacting particle systems 
in one dimension that model epidemic spreading: (i) the diffusive contact 
process 
(DCP), (ii) a  model that we introduce and call generalized diffusive 
contact process (GDCP), both
in finite volume in contact with boundary reservoirs, i.e., with open boundaries, 
and (iii) the susceptible-infectious-recovered 
(SIR) model on $\Z$. We establish duality 
relations for each system through an analytical approach.
It turns out that with open boundaries self-duality breaks down and qualitatively different properties 
compared to closed boundaries (i.e., finite volume without reservoirs) arise: 
Both the DCP and GDCP are ergodic but no longer absorbing, while the respective dual processes are absorbing but not ergodic. We provide expressions for the stationary correlation 
functions in terms of the dual absorption probabilities. 
We perform explicit computations 
for a small sized DCP, and for arbitrary size
in a particular setting of the GDCP.
The duality function is factorized for the DCP and GDCP, contrary to the SIR model for which the duality relation is nonlocal and yields an explicit
  expression of the time evolution of some specific correlation functions, 
  describing the time decay of the sizes of clusters of susceptible individuals.
\end{abstract}


\section{Introduction}\label{sec:intro}
In the context of Markov processes, 
duality is a remarkable tool to analyze a
model of interest using another model, its dual, via an observable for both 
models called duality function. For interacting particle systems, there are various
ways of deriving a duality relation 
: using analytical
tools via computations on generators or intensity matrices, 
through pathwise methods via a graphical representation,  or 
by an algebraic approach. 
A nonexhaustive list of books 
or surveys that deal with duality is 
 \cite{IPS, Liggett2, Griff, FluctuationsSchutz, SSV, CGR2024}, 
 see also references therein. 
 
In this article,  we study three different non-conservative 
interacting particle systems defined on a one-dimensional lattice, 
namely \textit{(i)} the 
diffusive contact process (see for instance \cite{Konno1994}) 
on a finite lattice, 
abbreviated below by DCP, 
\textit{(ii)} a generalization thereof with a variable death rate, 
abbreviated below by GDCP which like the DCP is a special 
case of a broad class of reaction diffusion systems surveyed 
with a view on duality in \cite{RG,LloydII,sudbury2000dual},
and \textit{(iii)} the 
susceptible-infectious-recovered (SIR) model on $\Z$
(as in \cite{Schutzetal2008}). 
All these particle systems, described
informally below and defined formally in terms of their generators in Sections 
\ref{dcp section}--\ref{sec:SIR}, are Markovian models for the 
spreading of infectious diseases, but with different microscopic dynamics. 
For each of these models we prove a duality relation 
using analytical tools. 

Informative duality relations are, 
in general, rare and exceptional. They occur, for 
example, in conservative particle systems in the presence of non-Abelian 
symmetries of the generator, as first pointed out in \cite{AbelianSchutz}
and further developed in \cite{Giardina&al2009},
or when the time evolution of $n$-point correlations is fully determined by 
$k$-point  correlations of lower or equal order $k\leq n$ \cite{RG,Fuji97}. 
This is not the case in general for models such as ours,  where the total number of 
particles is not 
conserved by the bulk dynamics, and the evolution of correlations involves 
higher order correlations.

The main novelty of our discussion of duality in the DCP and
the GDCP are open boundaries
that allow for particle exchanges with external reservoirs at the boundary
sites of the
one-dimensional finite lattice, much like in the paradigmatic 
open simple exclusion process \cite{Spohn83,Schutz1993PhaseTI,Derrida1993ExactSO,carinci2013duality,Floreani22}.
In the papers \cite{RG,LloydII} such reservoirs are not 
considered and in \cite{sudbury2000dual} 
they are actually explicitly ruled out as they include spontaneous 
creation of particles
which invalidates the results obtained in that paper.
Reservoirs have been considered in \cite{DianeSchwartz}, but at generic lattice sites
and for the purpose of establishing ergodic theorems, while we focus 
on boundary reservoirs and study their effect on correlations.

Indeed, it turns out that open boundaries lead to qualitatively different properties:
 For closed boundaries the empty configuration is an absorbing state, 
 so the invariant measure is unique and completely explicit. It is 
 the trivial Dirac measure 
 concentrated on the empty configuration. Moreover, both the closed DCP and 
 closed GDCP are
self-dual \cite{Konno1994,sudbury2000dual}. In contrast, when adding 
reservoirs of 
particles, which can remove or insert particles into the system at random times,
three new features arise: 
\textit{(i)} The empty configuration is 
no longer an absorbing state, 
as the reservoirs maintain the process ``alive''. 
\textit{(ii)} Self-duality is lost.
The dual process has absorbing sinks which are additional sites at the
edges of the one-dimensional lattice which absorb particles without 
ever releasing them. 
\textit{(iii)} The dual process is not ergodic. 
As time tends to infinity, the limiting measure depends 
strongly on the initial configuration.

Despite these differences some features of duality persist 
in the presence
of reservoirs. This is the factorization of the duality function, 
meaning that the global duality function is 
a product over local duality functions on each lattice site.
Such a duality 
is long-known for the contact process \cite{Liggett2} both on $\Z$ and
for closed boundaries and appears also 
in the closed DCP and GDCP 
\cite{Konno1994,sudbury2000dual}. Here we prove -- somewhat unexpectedly --
 that apart from extra boundary terms the same factorized duality function can be used also for the open DCP and open GDCP and yields
information about the invariant measure in terms of absorption probabilities
of the dual process. 
In a special case of the GDCP, it also yields explicit results for the time-dependent 
one-point function for arbitrary initial states. 

Duality functions that factorize into local duality functions for each lattice
site appear also in the general discussions of duality for reaction-diffusion systems
of \cite{RG,LloydII,Fuji97,sudbury2000dual} and thus factorization appears to be a natural feature of informative dualities \cite{Redi17}.
For the SIR model on $\Z$, the main novelty is the appearance of a
non-factorized kind of duality, that we call {\it cluster duality}, since the 
duality function relates the particle configuration in a cluster of 
neighboring sites to the dual process which involves the boundaries of the 
cluster. This type of non-local duality function, reminiscent of the so-called 
empty-interval method used to study correlations in reaction-diffusion 
processes \cite{Burschka89,Alim01}, 
appears to be unexplored in the context of duality in
interacting particle systems and may have further applications to other models. 
Also, note that although we treat each model separately, we aim, in future work, 
at a unified treatment of such models, via an algebraic approach. 
\medskip

We now briefly describe the three models studied in this work, 
that are one-dimensional interacting particle systems, 
either on $\Z$ or on $\Lambda_N:= \{1,..,N\}$ with $N \geq 1$ sites, and 
the results we obtain for them through duality. Let us first recall the 
definition of duality between two Markov processes.
\begin{Def}[Definition 3.1 in Chapter II of \cite{IPS}]\label{def:duality}
Let $X = (X_t)_{t \geq 0} $ and $Y = (Y_t)_{t \geq 0} $ be two continuous 
time Markov processes with state spaces $\Omega$ and 
$\Omega^{{\rm dual}}$, respectively.
We say that $X$ and $Y$ are dual with respect to a local duality function 
$D: \Omega\times {\Omega}^{{\rm dual}}\to \mathbb{R}$ if
\begin{equation} \label{processduality}
    \mathbb{E}_{x} D(X_t, y) =    \mathbb{E}_{y} D(x, Y_t)
\end{equation}
for all $(x, y) \in \Omega \times \Omega^{{\rm dual}}$ and $t \geq 0 $.
 In equation \eqref{processduality},   $\mathbb{E}_{x}$ (respectively 
 $  \mathbb{E}_{y}$) is the expectation w.r.t. the law of the process $X$ 
 initialized at $x$ (respectively the process $Y$  initialized at $y$). 
\end{Def}

In the diffusive contact process on $\Lambda_N$ (DCP),  
each lattice site can be occupied by 
at most one particle. Its dynamics is the superposition of the 
(basic) contact process (CP) with birth rate $\lambda>0$, with the 
symmetric simple exclusion process (SSEP), a particle exchange dynamics 
parameterized by a diffusion coefficient $\mathcal{D}\geq 0$. In the CP, 
 a particle on site $x$ on the lattice ``dies''
(i.e., disappears from the lattice) with rate 1 and ``infects'' (i.e. attempts to
create another particle),  with rate $\lambda>0$, on a neighboring site, provided
the latter is empty; otherwise the creation attempt fails. In the SSEP, particles 
attempt to jump to nearest neighbor sites, at rate $\mathcal{D}\geq 0$, provided 
the chosen target site is empty; otherwise the jump attempt fails, 
see \cite{IPS, Liggett2} for a precise definition
of both the CP and the SSEP.

The goal of using duality for the DCP, is to get a better 
understanding of its stationary state, when in contact with boundary reservoirs. 
In Section \ref{dcp section},  we establish in Theorem \ref{Dual} a duality result 
between the DCP with an open boundary, and the 
DCP with a purely absorbing boundary. From this duality relation, 
we provide a new expression for the correlation functions of the invariant measure 
in terms of absorption probabilities, and compute
them for small sized systems in appendix \ref{appN2}.  
We also comment briefly on the fast stirring limit, in which the
process effectively reduces to a birth-death chain on  the integers
$\{0,1,\dots,N\}$. Note that fast stirring limits have been much investigated 
for infinite volume systems in $\Z^d$, to establish bounds on the critical value
of the contact process, see e.g.
 \cite{konno1995asymptotic,berezin2014asymptotic,levit2017improved,mytnik2021general} 
 and also with a rescaling of time to study
hydrodynamic limits, see e.g. \cite{demasietal,durneu}.

The generalized diffusive contact process on $\Lambda_N$
(GDCP) has the same structure
as the DCP, namely the superposition of a SSEP with diffusion coefficient 
$\mathcal{D}$, and a generalization of the contact process, that is, a process with 
infection rate $\lambda>0$, but with death rates depending on the occupation 
variables on the  neighboring sites, instead of being constant equal to $1$. 
In Section \ref{sec:GDCP}, we derive in Subsection \ref{subsec:dualGDCP} 
a duality relation for the GDCP, and focus on a particular case that induces  
a dual process with no birth rate. As an application of the latter dual relation, 
we compute in Subsection \ref{subsec:compcor} the one-point
function in the invariant measure for a general lattice of $N$ sites.
Moreover, we obtain a system of ordinary differential equations
for the time-dependent one-point function, for arbitrary 
initial distributions.
After some elementary transformations, this system of ODE's becomes
-- again somewhat surprisingly -- identical to that
of the SSEP with open boundaries, which has no nonconservative bulk 
processes.\medskip

The one-dimensional susceptible-infectious-recovered (SIR) model, that goes 
back to the partial differential-equation model for epidemics introduced 
in \cite{Kermack&al1927}, describes the evolution on the infinite lattice $\Z$ 
(hence without reservoirs), of three species of individuals: susceptible 
(in state $S$), infectious ($I$),
recovered ($R$). An infectious individual infects a susceptible neighbor
 at rate $\beta>0$, it recovers and becomes immune 
 (that is, it remains forever in state $R$) at rate $\gamma>0$. 
This process has an 
interesting feature in common with the classical problem
of random sequential adsorption \cite{Evans1993RandomAC} which is 
a discrete version of R\'enyi's random space filling problem \cite{Renyi63}: 
an ongoing fragmentation
of the state space into independent ergodic components occurs
as the process evolves in time. This is accompanied
by a strong dependence of the limiting measure on the initial configuration.

As in the previous sections, our goal is to compute correlation functions 
which,
like in random sequential absorption, remains a non-trivial task despite 
the fragmentation of the state space. 
However for the SIR model, this is quite different to the DCP and the GDCP 
where we were interested in the invariant measure. Here, the state space is 
not irreducible and there are infinitely many invariant 
measures (that are Dirac measures on blocked states), so the asymptotic behavior of 
the system is highly dependent on the initial condition. We are rather 
interested in the time evolution of correlation functions according to the 
initial state of the system. In Section \ref{sec:SIR}, we study the 
``$n$-point cluster functions'', named as such in 
\cite{Schutzetal2008}, which are quantities written 
in terms of higher order correlation functions. 
Thanks to these cluster functions, we derive in Subsection \ref{sectioncluster}
 a duality relation between the SIR model and a two dimensional random walk 
 on two layers with a trap. From that, using a probabilistic approach, we 
 provide in Subsection \ref{subsec:appldual} an explicit expression of 
 the average of these cluster functions at any time. This new treatment of 
 the model allows us to extend the results for 
translation invariant initial distributions derived in
 \cite{Schutzetal2008}, to the case of non translation 
invariant initial distributions. \medskip

The paper is organised as follows. In Section \ref{dcp section}, we study the
DCP, adding in Appendix \ref{Konno}  computations 
that do not rely on duality. In Section \ref{sec:GDCP} we study the GDCP, 
and in Section \ref{sec:SIR}, we study the SIR model.
The proofs for all results are provided at the end of their respective sections.

\section{The diffusive contact process (DCP)} \label{dcp section}

In physical interacting particle systems that at their edges 
are in contact with 
external particle reservoirs, the particle exchange with such reservoirs is 
commonly modelled by so-called open boundaries. This means that particles 
at the edge sites of the lattice on which the particle system is defined are created 
and annihilated at certain rates. The reservoirs themselves do not enter the
mathematical description of the process.
In this section, we define the DCP with open 
boundaries, and establish the duality relation with the DCP with purely absorbing
boundaries. These are edge sites from which particles jump into 
auxiliary sites that correspond to sinks from which no particle can escape. In contrast
to the reservoirs envisaged in the open boundary setting, these sinks become
part of the dual process by enlarging the dual state space (compared to the
state space of the DCP with open boundaries) to keep track of
the number of absorbed particles.

A duality relation which turns reservoirs into sinks is quite common in 
interacting particle systems. We refer to \cite{DianeSchwartz} 
for an early result where absorbing sites in dual processes are discussed
with a view on establishing ergodic theorems. 
Such a duality also appears in the SSEP on a discrete segment 
with two open boundaries, see \cite{Spohn83,carinci2013duality,Floreani22}. However,
the analysis of the effect of reservoirs on the stationary state 
of the DCP substantially differs from the case of conservative particle systems 
such as the  SSEP, where 
the system is driven out of equilibrium due to a current of particles 
produced by the effect of reservoirs tuned differently, 
see \cite{Derrida,ExactSchutz, gonccalves2017hydrodynamics}. In our 
non-conservative setting, the system is constantly out of equilibrium.

\subsection{The model}\label{subsec:model-dcp}

For $N\geq 1$, denote by $\Lambda_N = \{1, \ldots,N\}$ the one dimensional 
finite chain of size $N$, which we refer to as the bulk. A site of 
$\Lambda_N$ is either empty (in state $0$), or occupied (in state $1$). 
Therefore, the state of the system is described by an element 
$\e=(\e_1,\ldots,\e_N)\in \Omega_N:=\{0,1\}^{\Lambda_N}$ such that for $ x \in \Lambda_N$,
\begin{equation}
\e_x=
    \begin{cases}
        &1~ \text{if}~x ~\text{is occupied}\\
        &0~ \text{if}~x ~\text{is empty.}
    \end{cases}
\end{equation}
In the bulk, the dynamics considered is the superposition of a contact process 
with parameter $\lambda>0$ (CP), with a symmetric simple exclusion process (SSEP).
Both these dynamics are of nearest-neighbor type and are defined as follows. 
For $x,y\in \Lambda_N$, $x\sim y$ means that $x$ and $y$ are 
neighboring sites in $\Lambda_N$.\medskip

In the CP, a site $x \in \Lambda_N$ becomes 
occupied at rate $\lambda \sum_{y\sim x} \e_y$ if it is empty, 
and becomes empty at rate $1$ 
if it is occupied. The generator of this dynamics acts on 
$f:\Omega_N \rightarrow \R$ as follows, 
for $\eta\in\Omega_N$ \cite{Liggett2}:
\begin{equation}\label{contact generator}
\begin{split}
        \LL^{\rm CP} f(\e) 
        &= \sum_{x=2}^{N-1}\big[\e_x + \lambda\big(1-\e_x\big)\big(\e_{x-1}
         + \e_{x+1}\big) \big]\big[f(\e^x)-f(\e) \big]\\
        &+ \big[\e_1 + \lambda\big(1-\e_1\big)\e_2 \big]\big[f(\e^1)-f(\e) \big]\\
        &+ \big[\e_N + \lambda\big(1-\e_N\big)\e_{N-1} \big]\big[f(\e^{N})-f(\e) \big].
\end{split}
\end{equation}
where  for $x,y\in \Lambda_N$, 
\begin{equation}\label{def:flip} 
   \e^x_y= \left\{
    \begin{array}{ll}
       \e_y ~ ~\text{if}~ y\neq x\\
       1-\e_x~~ ~ ~ ~~
      \text{if}~ y=x
        \end{array}
\right.
\end{equation}
represents the flip of the occupation variable at site $x \in \Lambda_N$. 
Notice that this generator can be re-written as follows:
\begin{equation}\label{def:genCP}
     \LL^{\rm CP} f(\e) = \frac{\eta_1}{2} \big[f(\e^1)-f(\e) \big] + \sum_{x=1}^{N-1}\LL_{x,x+1}^{\rm CP} f(\e) 
     + \frac{\eta_N}{2} \big[f(\e^N)-f(\e) \big] 
\end{equation}
with the bond generator
\begin{equation}\label{def:genCPbond}
\begin{split}
        \LL_{x,x+1}^{\rm CP}f(\e) &= \frac{1}{2}\e_x\big[f(\e^x)-f(\e) \big]
        + \lambda (1-\e_x)\e_{x+1}\big[f(\e^x)-f(\e) \big]\\
        & + \frac{1}{2}\e_{x+1}\big[f(\e^{x+1})-f(\e) \big]
        + \lambda (1-\e_{x+1})\e_{x}\big[f(\e^{x+1})-f(\e) \big].
\end{split}
\end{equation}

In the SSEP, particles jump to one of their neighboring sites, under the 
exclusion rule, namely each site can accommodate at most one particle, 
and the direction of the jump is not biased to the left or right.
The generator of this dynamics acts on $f:\Omega_N \rightarrow \R$ as follows, 
for $\eta\in\Omega_N$ :
\begin{equation}\label{exclusion generator}
    \LL^{\rm SSEP}f(\e) = \sum_{x=1}^{N-1}\LL_{x,x+1}^{\rm SSEP}f(\e),
    \end{equation}
with
\begin{equation}\label{exclusion bond generator}
\LL_{x,x+1}^{\rm SSEP}f(\e) = \e_x(1-\e_{x+1})\big(f(\e^{x,x+1})-f(\e)\big) 
+\e_{x+1}(1-\e_x)\big(f(\e^{x+1,x})-f(\e)\big),
\end{equation}
where for $x,y\in \{1,\ldots ,N\}$, if $\eta_x(1-\eta_y)=1$,
\begin{equation}\label{def:jump}  
   \e^{x,y}_{z}= \left\{
    \begin{array}{ll}
       \e_{z} ~ ~  ~ ~  ~ ~ \text{if}~ z\notin\{x,y\}\\
       \e_{x}-1 ~~ ~ ~ ~~
      \text{if}~ z=x \\
      \e_{y}+1  ~~ ~ ~ ~~
      \text{if}~ z=y.
        \end{array}
\right.
\end{equation} 
and $\eta^{x,y}=\eta$ otherwise.

We study the open boundary version of the DCP, namely we imagine
reservoirs of particles at each edge of the system, which insert
or remove a particle on sites $1$ and $N$. These reservoirs are parameterized 
by four positive parameters $\alpha,\gamma,\beta,\delta$. A particle is 
inserted at site $1$, resp. $N$, at rate $\alpha$, resp. $\delta$, provided 
the site is empty. A particle is removed from the system at site $1$, resp. $N$, 
at rate $\gamma$, resp. $\beta$, provided the site is occupied (see 
Figure \ref{C+Ewithreservoirs}). The associated generators of these boundary 
dynamics act on $f:\Omega_N \rightarrow \R$ as follows, for $\eta\in\Omega_N$:
\begin{eqnarray}\label{leftres}
    \LL_{-} f(\e) &=& \alpha\big(1-\e_1\big)\big[f(\e^1) - f(\e) \big] 
    + \gamma \e_1 \big[f(\e^1) - f(\e) \big]\\
\label{rightres}
\hbox{and} \qquad   \LL_{+} f(\e) &=& \delta\big(1-\e_N\big)\big[f(\e^N) - f(\e) \big] 
    + \beta \e_N \big[f(\e^N) - f(\e) \big].
\end{eqnarray}
Notice that we do not put an index `DCP' in these generators since they are not 
specific to our dynamics.

Finally, the generator for the open DCP is given by
\begin{equation}\label{generator_opendiffusivecontact}
   \mathcal{L}^{\rm DCP}= \LL_{-} + \LL^{\rm CP}+\mathcal{D}\LL^{\rm SSEP} + \LL_{+},
\end{equation} 
where $\mathcal{D}\geq 0$ is the diffusion parameter. 
Superposing the contact dynamics with the diffusion dynamics enables us to study the same model as in \cite{mourragui2023hydrodynamic}, which is more general than the dynamics governed by the contact process only. Setting $\mathcal{D} = 0$ would eliminate the contribution of the diffusion. We are, however, interested in the case when $\mathcal{D} \to + \infty $, the so called fast stirring limit studied in subsection \ref{subsec:faststirring}.
The dynamics of the DCP is irreducible and as the state space $\Omega_N$ 
is finite, the process admits a unique invariant measure $\nu^{\rm DCP}$, 
which depends on $\alpha, \delta, \gamma,\beta, \mathcal{D}$ and $\lambda$. 
When $\alpha=\gamma=\beta=\delta=0$ we speak of closed boundaries.
In this case the empty configuration is an absorbing state which means that
the invariant measure is the trivial Dirac measure concentrated on the 
empty configuration.
As soon as $\alpha\neq 0$ or $\delta\neq 0$, the empty configuration is no 
longer an absorbing state and $\nu^{\rm DCP}$ cannot be found by direct 
computations when $N$ is large.
\begin{figure}\label{C+Ewithreservoirs}
        \centering
        \includegraphics[scale=0.18]{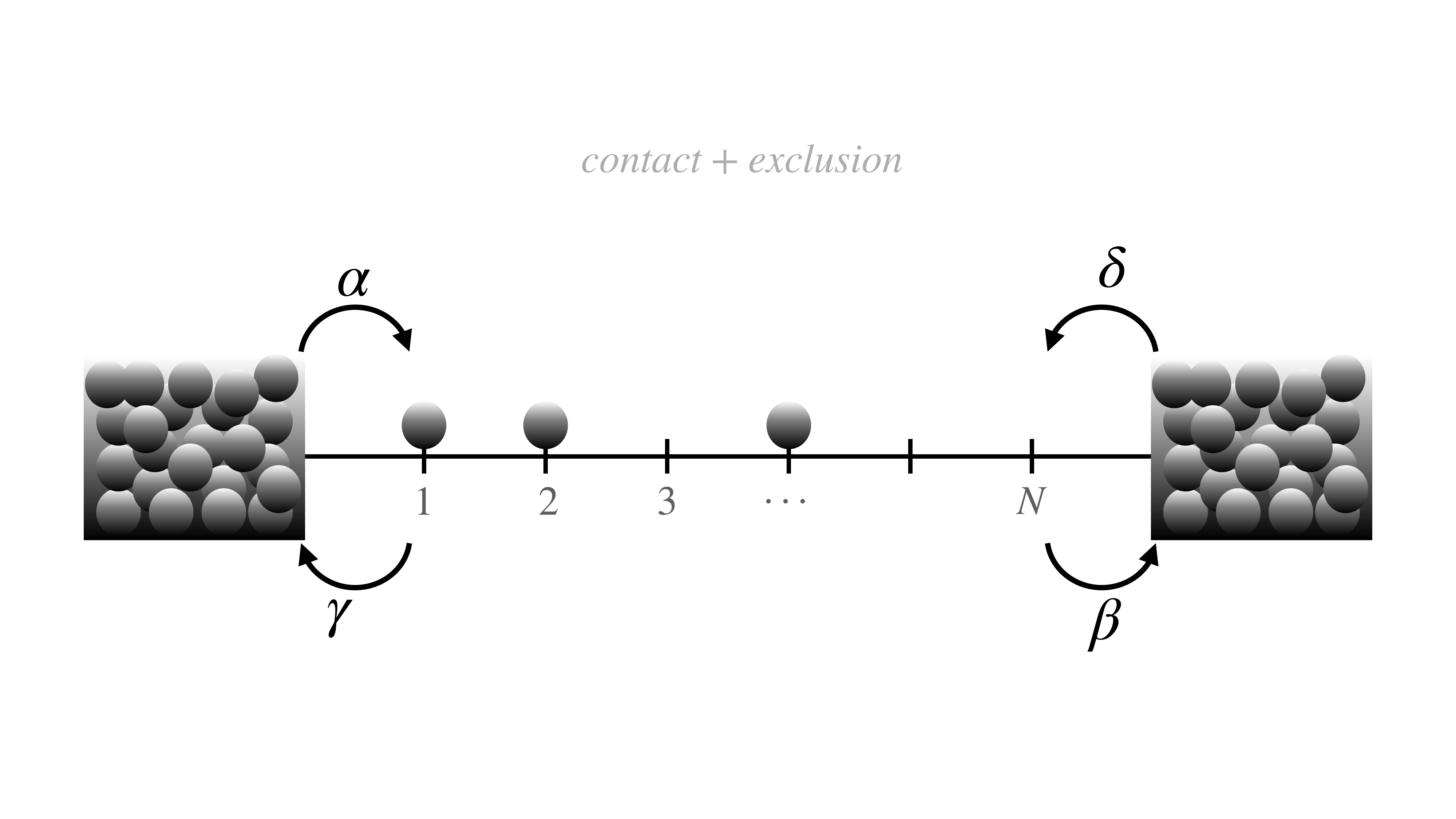}
        \caption{Contact + exclusion process on $\Lambda_N$ with reservoirs.}
    \end{figure}

\subsubsection*{Matrix formulation}
Since the DCP and the GDCP are particle systems on a finite lattice, 
and therefore with a finite state space, 
the duality function can be written as a matrix and the Markov 
generators of the processes can be written as 
intensity matrices \cite{liggett2010continuous}. 
This matrix approach with the choice of basis \eqref{order_i} below
is often a convenient tool to treat interacting particle
systems with countable state 
space \cite{Gwa1992BetheSF,LloydI,ExactSchutz,Giardina&al2009}
and will be used below alongside with the more traditional generator 
treatment used above.

Recall that an intensity matrix $L$ is such that
\begin{equation*}
L(x,y) \geq 0 \qquad \text{for} \qquad x \neq y \qquad 
\text{and} \qquad \sum_{y} L(x,y) = 0  \;.
\end{equation*}
For $x \neq y$, $L(x,y)$ represents the rate to go from state $x$ to state $y$.
\begin{rem}[Countable state space]\label{Matrixapproach}
If in Definition \ref{def:duality} both the original process 
$X = (X_t)_{t \geq 0} $ and the dual process $Y = 
(Y_t)_{t \geq 0} $ are Markov processes with a countable state space, the 
duality relation can be 
written in matrix notation as \cite{LloydII}
\begin{equation} \label{matrixnotation}
    LD = D\left(L^{\rm dual} \right)^T
\end{equation}
where $L$ is the intensity matrix corresponding to the Markov generator ${\cal 
L}$ of $X$, $L^{\rm dual}$ is the intensity matrix 
corresponding to ${\cal L}^{\rm dual}$ and 
the superscript $T$ denotes matrix transposition. Equivalently one could 
choose to adopt the quantum Hamiltonian representation 
\cite{LloydI,ExactSchutz} of the intensity matrix, namely $H=-L^T$, and in this 
case duality reads $H^TD=DH^{\rm dual}$, see \cite{AbelianSchutz}.
\end{rem}

Therefore, as the state space of the DCP is finite, one can encode the dynamics 
in a finite sized matrix. To write these matrices in the canonical basis $(e_i)_{1\leq i\leq 2^N}$, to each vector $e_i$, we associate the 
unique configuration $\e^{i}=(\e_1^{i},...,\e_N^{i})$ such that 
\begin{equation}\label{order_i}
    i= 1 + \sum_{k=1}^N 2^{N-k}\e_k^{i}.
\end{equation}
For example, for $N=2$, the order is
\begin{eqnarray*}
    e_1=(1,0,0,0) ~\text{associated to}~ (0,0),
    && ~e_2=(0,1,0,0) ~\text{associated to}~ (0,1),\\
    ~e_3=(0,0,1,0)~ \text{associated to}~ (1,0)~
    &&\text{and}~~e_4=(0,0,0,1)~ \text{associated to}~ (1,1).
\end{eqnarray*}
With this ordering of configurations, the bulk dynamics matrix writes
\begin{equation}\label{dcp-bulkmatrix}
L^{\rm CP}+ \mathcal{D}L^{\rm SSEP}=\sum_{x=1}^{N-1} \ell_x 
+ \widetilde \ell_1+\widetilde \ell_N,
\end{equation}
where for  $x=1, \ldots, N-1$, denoting by $\otimes$ the Kronecker product 
and $\mathds{1}$ the two-dimensional unit matrix,
\begin{equation}\label{dcp-ellx}
\ell_x := \mathds{1}^{\otimes (x-1)}\otimes \ell \otimes \mathds{1}^{\otimes (N-1-x)},
\end{equation} with $\ell$ 
corresponding to $\LL_{x,x+1}^{\rm CP}+\mathcal{D}\LL_{x,x+1}^{\rm SSEP} $: 
\begin{equation*}
\ell = \begin{pmatrix}
0 &  0 & 0 & 0  \\
1/2  & -( 1/2 + \mathcal{D} + \lambda) & \mathcal{D} & \lambda \\
1/2   & \mathcal{D} & -( 1/2  + \mathcal{D} + \lambda) & \lambda  \\
0 & 1/2  &1/2  & -1
\end{pmatrix}, \;
\end{equation*}
and, for $x=1$ and $x=N$, $\widetilde\ell_1$ and $\widetilde\ell_N$ correspond 
to the first and last terms in the right hand side of \eqref{def:genCP} with,
\begin{equation}\label{dcp-ell1 and ellN}
\widetilde\ell_1 := \widetilde \ell \otimes \mathds{1}^{\otimes (N-1)} 
\qquad \text{ and } \qquad 
\widetilde\ell_N := \mathds{1}^{\otimes (N-1)} \otimes \widetilde \ell
\end{equation}
where 
\begin{equation*}
\widetilde \ell = \begin{pmatrix}
0 &  0   \\
1/2  & -1/2 
\end{pmatrix}. \;
\end{equation*} 
Regarding the boundary dynamics the matrices  $\ell_1^{-}$ and $\ell_1^{+}$
correspond respectively to $\mathcal{L}_{-}$ in \eqref{leftres} and 
$\mathcal{L}_{+}$ in \eqref{rightres}, with
\begin{equation}\label{dcp-ell1- and ellN+}
\ell_1^- :=  \ell^- \otimes \mathds{1}^{\otimes (N-1)}, \qquad \text{ and } \qquad \ell_N^{+} := \mathds{1}^{\otimes (N-1)} \otimes  \ell^+,
\end{equation}
where 
\begin{equation*}
\ell^{-} = \begin{pmatrix}
-\alpha &  \alpha   \\
\gamma  & - \gamma 
\end{pmatrix} , \quad
\ell^{+} = \begin{pmatrix}
-\delta &  \delta   \\
\beta  & - \beta 
\end{pmatrix}.
\end{equation*}
In total, the intensity matrix of the DCP dynamics is given by
\begin{equation}\label{dcp-intensitymatrix}
 L^{\rm DCP} =   \sum_{x=1}^{N-1} \ell_x + \widetilde \ell_1
 +\widetilde \ell_N + \ell_1^{-} + \ell_{N}^{+}.
\end{equation}
\subsection{Duality for the DCP}\label{subsec:dcp-duality}
The open DCP has a diffusive and contact bulk dynamics 
with spontaneous creation and annihilation of particles at sites $1$ and $N$ 
due to the presence of two reservoirs in ghost sites $0$ and $N+1$ 
which, as discussed above, do not appear in the process. 
Theorem \ref{Dual} shows that the dual DCP preserves the same dynamics 
in the bulk while 
the boundary mechanism is quite different. The dual process is defined 
on an extended lattice which includes sites $0$ and $N+1$ where particles are 
permanently absorbed, thus describing particle sinks. Because of this 
the dual state space is different,
\begin{equation}\label{def:dual-state-space-DCP}
   \Omega_N^{\rm dual}= \mathbb{N}_0 \times \{0,1\}^{\Lambda_N} \times  \mathbb{N}_0 ,
\end{equation}
where  $\mathbb{N}_0=\{0,1,2,\cdots\}$  is the set of non-negative integers,
 while  $\mathbb{N}=\{1,2,\cdots\}$. 
\begin{thm} \label{Dual}
The open DCP $(\e_t)_{t\geq 0}$ with generator 
\eqref{generator_opendiffusivecontact} is dual to a purely absorbing 
contact process $(\xi_t)_{t\geq 0}$
with duality function $D:\Omega_N \times \Omega_N^{\rm dual} \rightarrow \R$, 
given by
   \begin{equation} \label{globaldualityfct}
       D(\e,\xi) = \Big(\frac{\gamma}{\alpha+ \gamma} \Big)^{\xi_0}
        H (\eta, \xi) \Big( \frac{\beta}{\beta+ \delta} \Big)^{\xi_{N+1}} \;,
   \end{equation}
     where the bulk duality function is:
     \begin{equation}\label{bulkduality}
        H (\eta, \xi) = \prod_{x\in A(\xi)}(1-\e_{x}) 
        = \prod_{x \in \Lambda_N} \left( 1- \eta_{x} \right)^{\xi_{x}} \;,
     \end{equation}
     with $A(\xi) = \Big\{y\in \{1,...,N\},~\xi_y=1 \Big\}$.
The dual generator is given by
\begin{equation}\label{Ldualaceb}
        \LL^{\rm DCP, Dual}= 
        \LL_-^{\rm Dual} + \LL^{\rm CP}+ \mathcal{D}\LL^{\rm SSEP} 
        + \LL_{+}^{\rm Dual},
\end{equation}
where $\LL^{\rm CP}$ is the generator of the contact process 
defined in \eqref{contact generator}, $\LL^{\rm SSEP}$  is the 
generator of the exclusion process defined in 
\eqref{exclusion generator} and the action at the boundary 
on function $f: \Omega_N^{\rm dual} \to \mathbb{R} $ is
\begin{eqnarray}\label{Ldualab}
        \LL_-^{\rm Dual}f(\xi) 
        & = &(\alpha + \gamma) \xi_1\big[f(\xi^{1,0}) - f(\xi) \big] \\
  \label{Ldual+} 
       \LL_{+}^{\rm Dual}f(\xi)
        & = &(\delta + \beta) \xi_N\big[f(\xi^{N,N+1}) - f(\xi) \big].
 \end{eqnarray}
\end{thm}

As pointed out in the introduction, 
both the contact and the exclusion processes 
with {\it closed} boundary conditions are self-dual with duality 
function \eqref{bulkduality}, 
see in particular \cite[Chapter 3, Section 4]{IPS} and 
 \cite[Chapter 5]{Konno1994}. This leads to a self-duality relation 
for the bulk of the DCP \cite{sudbury2000dual},  
namely we have that for any $\e,\xi\in \Omega_N$,
\begin{equation}\label{Duality}
    \big( \LL^{\rm CP}+ \LL^{\rm SSEP}\big)H(\e,.)(\xi) 
    = \big( \LL^{\rm CP}+ \LL^{\rm SSEP}\big)H(.,\xi)(\e) \,.
\end{equation}

\begin{rem}\label{parameters}

(a) It was noticed in \cite{FluctuationsSchutz} 
that for the symmetric simple exclusion process with either 
closed or periodic boundary conditions,  
the function 
\begin{equation}
    \widetilde H(\eta, \xi) 
    = \prod_{x \in \Lambda_N} 
    \left[  a_1 + a_2 \eta_{x} \right]^{ a_3+a_4 \xi_{x}} , \quad
a_i\in \R, \, i\in\{1,2,3,4\} 
\label{bulkdualitySSEP}
\end{equation}
is a self-duality function for any choice of the parameters $a_i$. The case $a_1=1$, 
$a_2=-1$, $a_3=0$ and $a_4=1$ yields $\widetilde H= H$, the self-duality function 
for the contact process.

(b) The bulk duality functions \eqref{bulkduality} and \eqref{bulkdualitySSEP} 
have a locally factorized form, namely
\begin{equation*}
    \mathrm{G}(\eta,\xi) = \prod_{x \in \Lambda_N} \mathrm{g} (\eta_x,\xi_x)
    \end{equation*}
which corresponds to a duality matrix of the form $G^{\otimes N}$, 
for a local duality matrix
\begin{equation}\label{genericG}
G = \left(\ba{cc}
a & b \\
c & d
\ea\right),
\end{equation}
where for the symmetric simple exclusion process the matrix elements 
are arbitrary, while 
for the contact process we have  $a=b=c=1,\,d=0$. 
\end{rem}

\begin{figure}
        \centering
\includegraphics[scale=0.2]{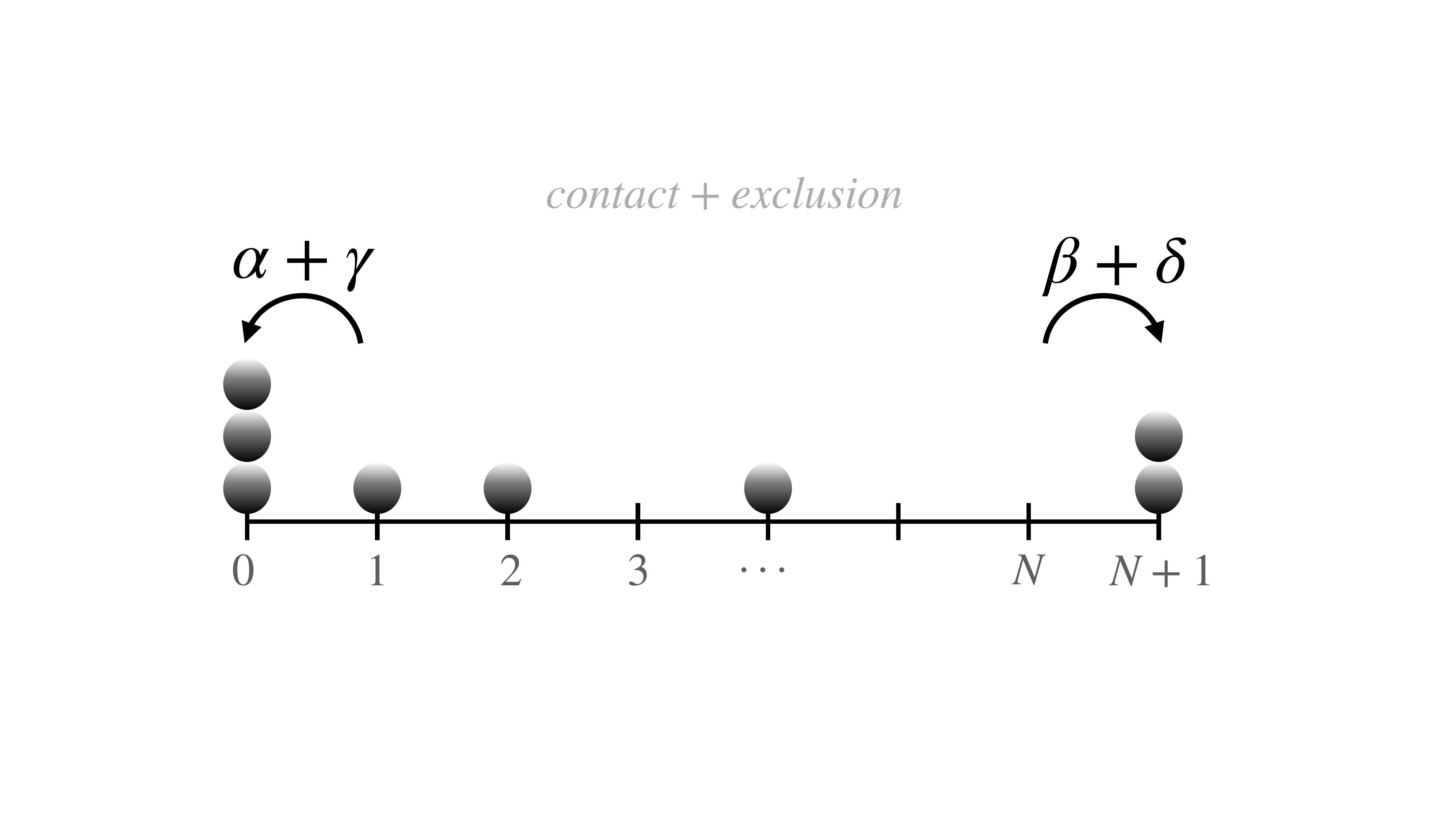}
        \caption{Dual diffusive contact process with sinks at sites $0$ and $N+1$.}
    \end{figure}

From now on the letter $\e\in\Omega_N$ will be relative to a DCP 
and $\xi\in\Omega_N^{\rm dual}$ to the dual process introduced in Theorem \ref{Dual}.

\subsection{Correlation functions via duality}\label{subsec:corfu}
In order to provide a better understanding of the invariant measure $\nu^{\rm DCP}$ 
of the DCP with absorbing boundaries, we study its correlation functions 
thanks to the duality relation established in Theorem \ref{Dual}. The dual 
process is seemingly simpler, as it almost surely becomes extinct. We are 
thus able to express the correlation functions of the original process in 
terms of absorption probabilities of the dual process. This allows to get 
an upper bound for the one-point function, see subsection 
\ref{subsubsection_one_point}. However, the explicit computation of 
correlation functions relies on solving large systems of linear equations, 
see subsection \ref{subsec:correlation-exact}, and it is not clear whether 
they are simpler to solve than the ones obtained in Appendix \ref{Konno} 
through a more standard approach. \medskip

In the following, the one-point function is the expectation 
$\rho_1^{\rm DCP}(y) 
  := \mathbb{E}_{\nu^{\rm DCP}}[\e_y]$  of the occupation number $\e_y$
w.r.t. the invariant measure of the DCP. For $2\leq \ell \leq N$, 
the $\ell$-point correlation function of 
$\nu^{\rm DCP}$ is the function $\rho_{\ell}^{\rm DCP}$ defined by 
\begin{equation}\label{ell_point_corr_def}
    \forall \   \ x_1 <... < x_{\ell} \in \Lambda_N,~ \rho_{\ell}^{\rm DCP}
    (x_1,...x_{\ell}) := \mathbb{E}_{\nu^{\rm DCP}}
    \big[\e_{x_1}...\e_{x_{\ell}} \big].
\end{equation}
Introduce the extinction time $\tau_N$ of the DCP $(\xi(t))_{t\geq 0}$ 
with purely absorbing boundaries:
\begin{equation}\label{def:tauN}
    \tau_N:= \inf \big\{t\geq 0,~ \forall x\in \{1,...,N\},~ \xi_x(t) =0\big\}. 
\end{equation}
Also, introduce the total number of particles absorbed on the left and 
right hand side boundary:
\begin{equation}\label{def:absorbedparticles}
\xi_0(\infty) := \underset{t\rightarrow \infty}{\lim}\xi_{0}(t),
~ ~ \text{and}~ ~ \xi_{N+1}(\infty) 
:=\underset{t\rightarrow \infty}{\lim}\xi_{N+1}(t) , 
\end{equation}
where the limits are almost sure. As the process almost surely becomes extinct, 
$\tau_N$, $\xi_0(\infty) $ and $\xi_{N+1}(\infty)$ are almost surely finite and 
$\xi_0(\infty) = \xi_0(\tau_N)$, $\xi_{N+1}(\infty) = \xi_{N+1}(\tau_N)$ 
almost surely.
To state and prove the results it is convenient to denote
\begin{equation}\label{def:c0-cN+1}
c_{-}:= \frac{\gamma}{\alpha + \gamma}~ ~ 
\text{and}~ ~ c_{+} := \frac{\beta}{\beta + \delta},
\end{equation}
which are related to the left (resp. right) reservoir  
densities $1-c_-$ (resp. $1-c_+$).

\subsubsection{One-point function}\label{subsubsection_one_point}

For $y\in \Lambda_N$, denote by $\delta_y$ the element of $\Omega_N^{\rm dual}$ 
where there is a particle in site $y \in \Lambda_N$ and all other sites 
$x \neq y $ are empty. 

\begin{prop}\label{Proposition_one_point_DCP} 
The one-point 
function of $\nu^{\rm DCP}$ is given by
\begin{equation}\label{one_point_corr}
  \forall \ \ y \in \Lambda_N,~ \rho_1^{\rm DCP}(y) 
  = 1 - \mathbb{E}_{\delta_y}\big[c_{-}^{\xi_0(\infty)}c_{+}^{\xi_{N+1}(\infty)} \big] 
  = 1 - \mathbb{E}_{\delta_y}\big[c_{-}^{\xi_0(\tau_N)}
  c_{+}^{\xi_{N+1}(\tau_N)} \big].
\end{equation}
\end{prop}

Although the dual process is seemingly simpler than the original one, 
 it remains a non-conservative process 
and the total number of particles absorbed by the reservoirs 
can reach arbitrary values. 
Therefore, we are not able to provide an explicit expression of the right hand side 
term in \eqref{one_point_corr}, contrary to the case of conservative particle systems 
where exact formulas are available. We refer to section 2.3 of \cite{frassek2022exact} 
and Theorem 2.2 of \cite{integrable}, where a closed formula 
for the $n$ points correlation 
is found in the non-equilibrium steady state of two models 
with open boundary and the same purely absorbing dual.

However, note that the dual process $(\xi(t))_{t\geq 0}$ either dies out 
before any particle lying in the bulk has had time to reach the boundary, or, 
it dies out and at least one particle has been absorbed by the boundary reservoirs. 
Thanks to this observation, the following upper bound holds:
\begin{lem}\label{lemme}
    For $y\in \Lambda_N$, denote by $A_y(\infty)$ the following event: 
   \[  A_y(\infty) 
   =  \Big(\xi(0)= \delta_y,~  c_{-}^{\xi_0(\infty)}c_{+}^{\xi_{N+1}(\infty)} 
   < 1 \Big),\]
    that is, starting from a single dual particle at site $y$, 
    at least one particle is absorbed by the boundary of the system. Then,
    \begin{equation} \label{onebound}
        \rho_1^{\rm DCP}(y) \leq \mathbb{P}_{\delta_y}\big[A_y(\infty) \big].
    \end{equation}
\end{lem}

\begin{rem}\label{rem:interest}
    This inequality is of interest if the upper bound is small, that is, 
   that starting from a particle at $y$, 
    it is unlikely that the process reaches the boundary before 
    it becomes extinct. Heuristically, if the time of extinction $\tau_N$ 
    of the process is small, corresponding to a small birth rate $\lambda$,
    then, the particle should not have time to reach the boundary. 
    Precisely, an estimate of $\tau_N$ is known in the case of a 
    subcritical contact process on $\Lambda_N$ (without exclusion), 
    and we refer to \cite[Part 1, Theorem 3.3]{Liggett2}.
\end{rem}

\subsubsection{Higher order correlation functions}\label{subsec:correlation-exact}

The idea of relating the $\ell$-point correlations using the dual absorption 
probabilities goes back to  \cite{giardina2007duality} for a model of 
stochastic diffusion of energy. In the same spirit as in Proposition 2 
in \cite{giardina2007duality} we can show the following results for the DCP.
Given $2\leq \ell \leq N$ and sites $1\leq x_1<...<x_{\ell}\leq N$, 
denote by  $ \delta_{x_1,x_2,...,x_{\ell}}$, the element of 
$\Omega_N^{\rm dual}$ where there is a particle at sites $x_1,...,x_{\ell}$, 
and none elsewhere.
\begin{prop}\label{any_point_label} 
For any $1\leq \ell \leq N$ 
and any $1\leq x_1<...<x_{\ell}\leq N$,
\begin{equation}\label{lcorrelations3}
\begin{split}
        \mathbb{E}_{\nu^{\rm DCP}}
  \Big[ \prod_{j=1}^{\ell} \left( 1 -\eta_{x_j} \right) \Big] 
  &= \sum_{k=0}^{\ell}(-1)^k \sum_{1\leq i_1<...<i_k\leq \ell}
  \rho_k^{\rm DCP}(x_{i_1},...,x_{i_k}) \\
        &=\sum_{m,n\geq 0} \mathbb{P}_{\delta_{x_1,x_2,...,x_{\ell}}}\Big[\xi_0(\infty)=m, \xi_N(\infty)=n \Big] c_{-}^m c_{+}^n.
\end{split}
\end{equation} 
\end{prop}

If we consider a small bulk, we can explicitly compute the above
 absorption probabilities.
 The idea is to perform a conditioning on the first step of the dual dynamics. 
We leave to Appendix \ref{appN2} the explicit computations for $N=1$ and $N=2$.

For $N$ sites notice that there are 
\[\displaystyle \sum_{k=1}^{N} \binom{N}{k} = 2^N -1\] 
nontrivial initial dual configurations. 
We define $M_N$ to be the $2^N -1$ square matrix given by 
the condition on the first steps starting from the dual initial configuration.
Then our goal is to show that $M_N$ is invertible so that all the absorption 
probabilities -- and so the corresponding correlation functions -- are determined
as solutions of a linear system. 
For $k\geq 0$, denote by 
$x_i^k = \mathbb{P}_{\xi^{i}}\big[\xi_0(\infty) = k \big]$  
the absorption probabilities of $k$ dual particles starting from the initial 
configuration $\xi^{i}$, $1 \leq i \leq 2^N -1$ that we want to find.
The order of these unknown variables 
is chosen in an arbitrary way and one option to systematically ordering 
them is the following, 
defined recursively for $N$ sites. There are $2^N -1$ total number of non trivial 
configurations: the first $2^{N-1}$ unknown variables have a particle at site $1$ 
and the rest is completed 
by a configuration from the $N-1$ th ordering. 
The $2^{N-1}-1$ coordinates left are the ones 
such that site $1$ is empty and the rest is completed 
by a configuration from the $N-1$ th ordering.

More explicitly, for $N=2$ sites we have
$\xi^1= \delta_1$, $\xi^2= \delta_{1,2} $, $\xi^3= \delta_2$ 
which is the order we used, see Fig. \ref{Ordering}. 
For $N=3$, this rule would give us the ordering
$\xi^1= \delta_1$, 
    $\xi^2= \delta_{1,2}$, $\xi^3= \delta_{1,2,3}$, $\xi^4
    = \delta_{1,3}$, $\xi^5= \delta_2$, $\xi^6
    = \delta_{2,3}$, $\xi^7= \delta_{3}$ and so on. 

 \begin{figure}
        \centering
        \includegraphics[scale=0.3]{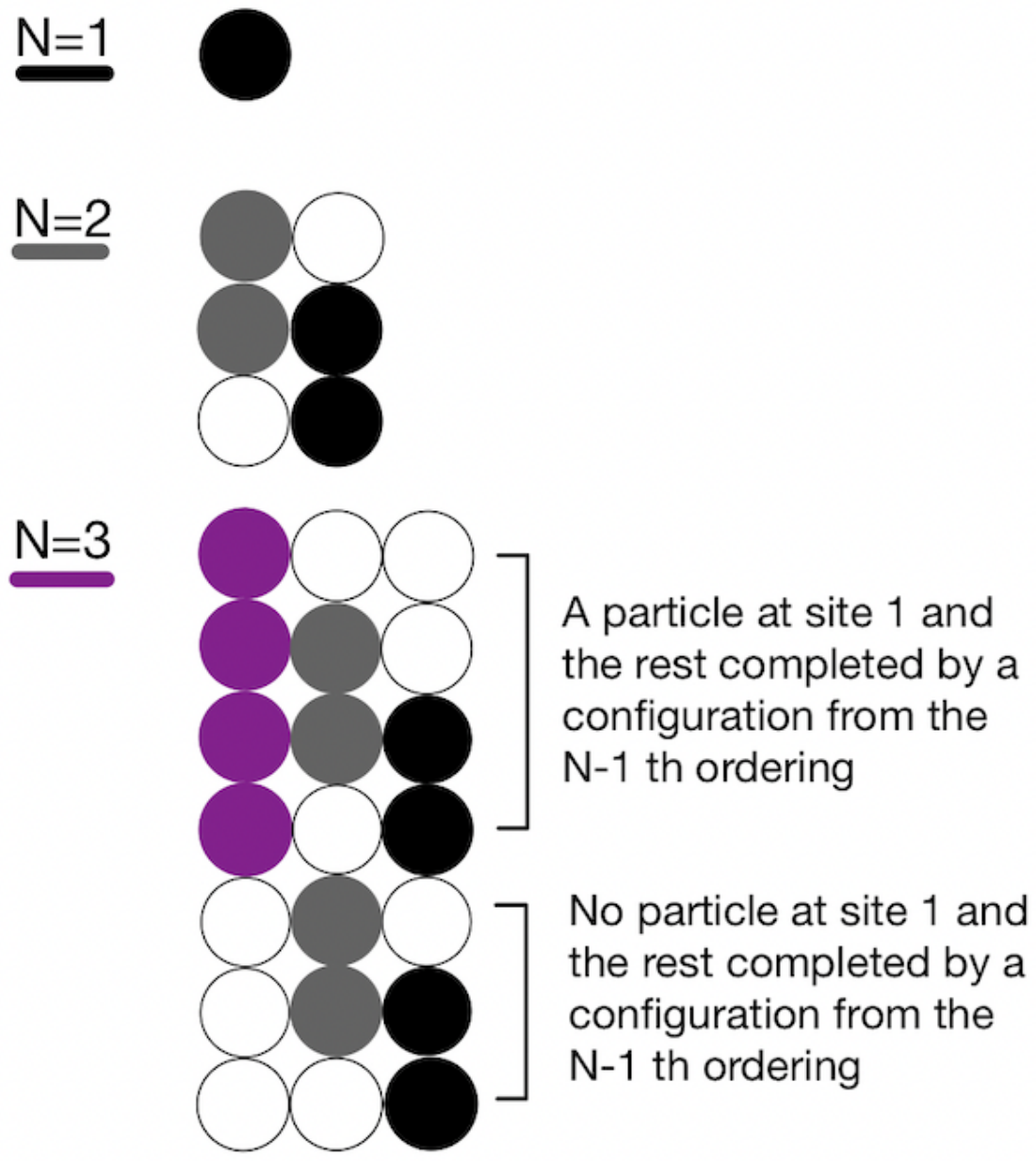}
        \caption{Ordering of configurations}\label{Ordering}
    \end{figure}

Also, for $\xi\in \Omega_N^{\rm dual}$, introduce $\mathcal{A}(\xi)$ 
the set of configurations containing at least one particle and 
resulting from $\xi$ after the jump, birth or death of a particle in $\e$ 
(but not an absorption in the left side reservoir). 
Denote by $(p(\xi^{i},\xi^{j}))_{1\leq i,j \leq n}$ 
the transition probabilities between configurations. 
Then for any $1\leq i \leq 2^{N}-1$ and for any $k\geq 2$, 
the following recurrence relation holds:
\begin{equation*}
    x_i^k = \sum_{\underset{\xi^j\in \mathcal{A}(\xi^{i}) }{j=1}}^{2^N-1} 
    p(\xi^{i},\xi^j)x_j^k
+ \sum_{\underset{\xi^j\in \Lambda_N \setminus \mathcal{A}(\xi^{i}) }{j=1}}^{2^N-1} p(\xi^{i},\xi^j)x_j^{k-1}.
\end{equation*}
For $k=0,1$ the recurrence relation remains the same modulo some changes 
on the second sum in the right hand side. Therefore, denoting $C_k$ 
the column vector $(x_i^k)_{1\leq i\leq 2^N-1}$, for any $k\geq 0$, 
there is a matrix of size $2^N-1$ of the form 
$M_{N}=I_{2^N-1}-P$ such that 
\[M_{N} C_k = R_k, \]
 where $R_k$ is a column vector involving some constant terms for $k=0,1$ 
 and terms from $C_{k-1}$ for $k\geq 2$, and where 
 \[P= \big(p(\xi^{i},\xi^{j})\mathds{1}_{\xi^j\in \mathcal{A}(\xi^{i}) }\big)_{1\leq i,j\leq 2^N-1}.\]
By the Perron-Frobenius Theorem, as $P$ is the transition matrix 
of an irreducible dynamics, one has that $M_{N}$ is invertible and
\begin{equation*}
    M^{-1}_{N} = \sum_{k\geq 0} P^k \;. 
\end{equation*}

\subsubsection{Fast stirring limit}\label{subsec:faststirring}
It is interesting to consider the fast stirring limit, that is, 
$\mathcal{D} \to \infty$ for the DCP model. Unlike in hydrodynamics, here, 
we discuss the setting where time is not rescaled when performing this limit.\\ 
When $\mathcal{D} \to \infty$, the DCP reduces to a birth-death chain 
on the lattice $\{0,1,\dots,N\}$. This can be seen by considering a configuration 
$\eta(t)$ with $n(t)$ particles, where 
$0\leq n(t) \leq N$, and where the process $n(t)$ is defined in a way 
that is analogous to the derivation
of the reaction term in the fast stirring limit of a reaction-diffusion processes,
as discussed in \cite{demasietal,durneu}. More precisely, between any attempt 
to create or annihilate a 
particle in the bulk or at the boundaries,
the process is just a SSEP with infinite rate, which has the uniform measure 
as its stationary distribution. This means that in between any 
creation/annihilation event of a particle (death, infection, and reservoir 
events), any configuration $\eta(t)$ with $n(t)$ particles, moves uniformly 
to another configuration $\eta'(t)$ with $n'(t)$ particles,
so that on the time scale of the creation and annihilation processes, 
the following transitions occur: any given 
configuration with $n(t)$ particles will turn uniformly
into another one with $n(t)\pm 1$ particles. The invariant measure of the
process in this limit thus becomes a convex combination of uniform measures
on configurations with $n$ particles, where the weights are given by the
stationary weights of the birth-death chain that describes the transitions
between configurations of $n$ and $n\pm 1$ particles.\medskip

\noindent
Specifically, the rate at which a particle dies in the closed DCP 
at a given site  (transition $n(t)\to n(t)-1$), is $n/N$ (which is the 
probability of finding an occupied 
site in the uniform distribution), and since there are $N$ sites 
where particles can die, the rate of transition
is $n$. In the open DCP, particles can also die at the boundary sites 
with rates $\gamma$, and $\beta$ respectively. Hence the transition rate 
from a configuration with $n(t)$ particles to a configuration with $n(t)-1$ 
particles is $n+(\beta+\gamma)n/N$. Likewise, the rate at which a particle 
is created in the DCP by infection at a bond, is given by the probability 
of finding an occupied site with a vacant neighbor in the uniform distribution 
(which is $n(N-n)/(N(N-1))$ times the number of 
bonds (which is $N-1$) times $\lambda$ 
(since at each bond the infection
can occur with rate $\lambda/2$ from a particle 
on the right or on the left site 
of the bond). In the open DCP particles can also be created at the boundary sites 
with rates
$\alpha$, and $\delta$ respectively. Hence the transition rate 
from a configuration with $n(t)$ particles to a configuration with $n(t)-1$ 
particles is $\lambda n (1-n/N)  + (\alpha+\delta) (1-n/N)$.
Thus, the DCP degenerates into a birth-death chain $n(t)$
with the following transition rates:
\begin{equation}\label{faststirringrates}
    n \rightarrow n-1,~\text{at rate }~n [1 + (\beta+\gamma)/N],
    ~~n \rightarrow n+1,~\text{at rate }~ [\lambda n 
     + (\alpha+\delta)] (1-n/N).
\end{equation}
Without reservoirs, $n=0$ is an absorbing state, but not when $\alpha\neq 0$ 
or $\delta\neq 0$. 

Since the arguments that lead to these rates are
analogous to those of \cite{demasietal,durneu} we refrain from
providing a formal proof. For the stationary probabilities and general properties
of this birth-death chain
we refer to \cite{Karlin1957}. 
Again, we point to Appendix \ref{appN2} where we compute explicitly 
the absorption probabilities for $N=2$ under the fast stirring limit.

\subsection{Proofs for Section \ref{dcp section}}\label{subsec:proofs-dcp} 
\subsubsection{Proofs for Section \ref{subsec:dcp-duality}}
\label{subsubsec:proofs-dcp-dual}

First, notice that in our case Definition \ref{def:duality} is equivalent 
to the following relation involving the 
corresponding Markov generators of the processes:

\begin{prop}\label{Generatorapproach}{\rm(\cite[Theorem 3.42]{liggett2010continuous})}
Given $X$ a process with generator ${\cal L}$ and $Y$, with generator ${\cal L}^{\rm dual}$, then $X$ and $Y$ are dual with local duality function $D$ if
  \begin{equation}\label{gendualfirstdef}
\left({\cal L} D(\cdot, x)\right) (y)
=\left({\cal L}^{\rm dual} D(y, \cdot) \right)(x) \;.
\end{equation}  
for all $x \in \Omega$ and $y \in \Omega^{\rm dual}$.
\end{prop}

\begin{proof}[Proof of Theorem \ref{Dual}]
Taking into account Proposition \ref{Generatorapproach}, 
and by equation \eqref{Duality}, we are left to show that 
\begin{equation*}
    \big(\LL_{-}+ \LL_{+}\big)D(.,\xi)(\e) 
    = \big(\LL_{-}^{\rm Dual} +\LL_{+}^{\rm Dual} \big)D(\e,.)(\xi),
\end{equation*}
where we recall that 
$\LL_{-}^{\rm Dual}$ is defined in \eqref{Ldualab},
$\LL_{+}^{\rm Dual}$ in \eqref{Ldual+}, $\LL_{-}$ in \eqref{leftres} 
and $\LL_{+}$ in \eqref{rightres}. Without loss of generality, we prove 
the result for the left reservoir, i.e. 
\begin{equation*}
    \LL_{-} D(.,\xi)(\e) = \LL_{-}^{\rm Dual} D(\e,.)(\xi),
\end{equation*}
and the proof that $\LL_{+} D(.,\xi)(\e) = \LL_{+}^{\rm Dual} D(\e,.)(\xi)$ 
follows the same steps. We have
\begin{equation*}
     \LL_{-}D(.,\xi)(\e)
     = \alpha\big(1-\e_1\big)\big[D(\e^1,\xi) - D(\e,\xi) \big] 
     + \gamma \e_1 \big[D(\e^1,\xi)- D(\e,\xi) \big].
\end{equation*}
Now, 
\begin{equation*}
\begin{split}
     D(\e^1,\xi) &= c_-^{\xi_0} \ c_+^{\xi_{N+1}}
     \prod_{x\in A(\xi)}(1-\e^1_{x})  \\
     &= c_-^{\xi_0} \ c_+^{\xi_{N+1}} \prod_{x\in A(\xi)\setminus\{1\}}(1-\e_{x}) \e_1\mathds{1}_{1\in A(\xi)} + D(\e,\xi)\mathds{1}_{1\notin A(\xi)}.
\end{split}
   \end{equation*}
Therefore,
\begin{equation*}
    \begin{split} 
        &D(\e^1,\xi)-D(\e,\xi)=D(\e^1,\xi) 
        - D(\e,\xi)\mathds{1}_{1\in A(\xi)}-D(\e,\xi)\mathds{1}_{1\notin A(\xi)}\\
        &=c_-^{\xi_0} \ c_+^{\xi_{N+1}} \prod_{x\in A(\xi)\setminus\{1\}}(1-\e_{x}) 
    \e_1\mathds{1}_{1\in A(\xi)} 
        - D(\e,\xi)\mathds{1}_{1\in A(\xi)}\\
        &= c_-^{\xi_0} \ c_+^{\xi_{N+1}}(2\e_1-1)\prod_{x\in A(\xi)\setminus\{1\}}
        (1-\e_{x}) \mathds{1}_{1\in A(\xi)}.
    \end{split}
\end{equation*}
Finally, we are left with
\begin{equation*}
    \begin{split}
        & \LL_{-} D(.,\xi)(\e)
        = \alpha\big(1-\e_1\big)(2\e_1-1) \ c_-^{\xi_0} \ c_+^{\xi_{N+1}} \
        prod_{x\in A(\xi)\setminus\{1\}}(1-\e_{x}) \mathds{1}_{1\in A(\xi)}\\
        & + \gamma \e_1 (2\e_1-1) \ c_-^{\xi_0} \ c_+^{\xi_{N+1}}
        \prod_{x\in A(\xi)\setminus\{1\}}(1-\e_{x}) \mathds{1}_{1\in A(\xi)}\\
        &= - \alpha D(\e,\xi)\mathds{1}_{1\in A(\xi)}
         + \gamma \ c_-^{\xi_0} \ c_+^{\xi_{N+1}} 
         \prod_{x\in A(\xi)\setminus\{1\}}(1-\e_{x})
         \e_1\mathds{1}_{1\in A(\xi)}\\
        &= - (\alpha + \gamma) D(\e,\xi)\mathds{1}_{1\in A(\xi)} + (\alpha +\gamma) \
        c_-^{\xi_0 +1} \ c_+^{\xi_{N+1}} 
        \prod_{x\in A(\xi)\setminus\{1\}}(1-\e_{x}) \mathds{1}_{1\in A(\xi)}\\
        &= (\alpha + \gamma)\xi_1\Big[D(\e,\xi^{1,0}) - D(\e,\xi) \Big],
        \end{split}
\end{equation*}
which is exactly the dual absorbing generator of the left boundary.
\end{proof}

\subsubsection{Proofs for Section \ref{subsec:corfu}}
\label{subsubsec:proofs-corfu}

\begin{proof}[Proof of Proposition \ref{Proposition_one_point_DCP}]
    By definition \eqref{globaldualityfct} of the duality function $D$
    given in Theorem \ref{Dual}, 
    if we choose as dual process $\xi(t) = \delta_y (t)$, we have 
\begin{equation*}
    D(\e,\delta_y) = 1- \eta_y \;.
\end{equation*}
The duality relation \eqref{gendualfirstdef}
 then yields that for any $t\geq 0$,

\begin{equation}\label{relation}
\int_{}\mathbb{E}_{\e}\big[D(\e(t),\delta_y) \big]\nu^{\rm DCP}(d \e) 
= \int_{}\mathbb{E}_{\delta_y}\big[D(\e,\xi(t)) \big]\nu^{\rm DCP}(d \e).
\end{equation}
By invariance of $\nu^{\rm DCP}$, 
the left hand side in \eqref{relation} does not depend on $t$ and equals
\begin{equation}\label{one_pointR1}
\int_{}\mathbb{E}_{\e}\big[D(\e,\delta_y) \big]\nu^{\rm DCP}(d\e)
 = 1- \mathbb{E}_{\nu^{\rm DCP}}[\e_y] \;,
\end{equation}
while the right hand equals 
\begin{equation}\label{one_pointR2}
\int_{}\mathbb{E}_{\delta_y}\big[D(\e,\xi(t)) \big]\nu^{\rm DCP}(d \e) 
\underset{t\rightarrow \infty}{\longrightarrow}
\mathbb{E}_{\delta_y}\big[c_{-}^{\xi_0(\infty)}c_{+}^{\xi_{N+1}(\infty)} \big],
\end{equation}
since for $t\rightarrow \infty$ all dual particles will eventually be absorbed 
at the boundary. Collecting \eqref{one_pointR1} and \eqref{one_pointR2} 
yields the result.
\end{proof}

\begin{proof}[Proof of Lemma \ref{lemme}]
    We have 
    \begin{align*}
    \mathbb{E}_{\delta_y}\big[c_{-}^{\xi_0(\infty)}c_{+}^{\xi_{N+1}(\infty)} \big] 
    &= \mathbb{P}_{\delta_y}\big[A_y(\infty)^{c} \big] 
    + \mathbb{E}_{\delta_y}\big[c_{-}^{\xi_0(\infty)}
    c_{+}^{\xi_{N+1}(\infty)} \mathds{1}_{A_y(\infty)} \big]\\
    &= 1- \mathbb{P}_{\delta_y}\big[A_y(\infty) \big] 
    + \mathbb{E}_{\delta_y}\big[c_{-}^{\xi_0(\infty)}
    c_{+}^{\xi_{N+1}(\infty)} \mathds{1}_{A_y(\infty)} \big],
\end{align*}
therefore,
\begin{equation*}
   \rho_1^{\rm DCP}(y) = 1- \mathbb{E}_{\delta_y}\big[c_{-}^{\xi_0(\infty)}
   c_{+}^{\xi_{N+1}(\infty)} \big]\leq \mathbb{P}_{\delta_y}\big[A_y(\infty) \big].  
\end{equation*}
\end{proof}

\begin{proof}[Proof of Proposition \ref{any_point_label}]
    For simplicity, we consider that only the left reservoir interacts with the bulk, 
    that is, $\delta=\beta=0$, but the reasoning remains the same if that is not 
    the case. Our goal is therefore to prove that for any 
    $1\leq x_1<...<x_{\ell}\leq N$,  
\begin{equation} \label{lcorrelations2}
    \mathbb{E}_{\nu^{\rm DCP}}
    \Big[ \prod_{j=1}^{\ell} \left( 1 -\eta_{x_j} \right) \Big] 
    = \sum_{k\geq 0} 
    \mathbb{P}_{\delta_{x_1,x_2,...,x_{\ell}}}\Big[\xi_0(\infty)=k \Big] c_{-}^k.
\end{equation}
In this case, it follows that the duality relation \eqref{processduality}
 holds with duality function
\begin{equation}\label{def:duual-fct}  
D(\e,\xi) = c_{-}^{\xi_0}\prod_{x\in \Lambda_N}(1-\e_x)^{\xi_x} .
\end{equation}
In other words, the initial process is dual to a DCP with an absorbing 
sink with rate $\alpha + \gamma$. Then, proceeding as in the proof of 
Proposition \ref{Proposition_one_point_DCP}, we consider the duality 
relation \eqref{processduality} 
under the invariant measure for the original process and, taking 
$\xi = \delta_{x_1,x_2,...,x_{\ell}}$, we have:
\begin{equation} \label{lcorrelations}
    \mathbb{E}_{\nu^{\rm DCP}}\Big[\left(1-\e_{x_1}\right)\left(1-\e_{x_2}\right)
    \ldots \left(1-\e_{x_{\ell}}\right) \Big] 
    = \mathbb{E}_{\delta_{x_1,x_2,...,x_{\ell}}}\Big[c_{-}^{\xi_0(\infty)} \Big].
\end{equation}
Expanding both sides of \eqref{lcorrelations}, we are left with 
\begin{equation*}
    \sum_{k=0}^{\ell}(-1)^k \sum_{1\leq i_1<...<i_k\leq \ell}
    \rho_k^{\rm DCP}(x_{i_1},...,x_{i_k}) 
    =\sum_{k\geq 0} 
    \mathbb{P}_{\delta_{x_1,x_2,...,x_{\ell}}}\Big[\xi_0(\infty)=k \Big] c_{-}^k \;,
\end{equation*}
where 
$\mathbb{P}_{\delta_{x_1,x_2,...,x_{\ell}}}\Big[\xi_0(\infty)=k \Big]$ 
is the probability that $k$ dual particles are absorbed 
in the sink when the dual process 
is initialized with $\ell$ particles in sites $0 < x_1 < \ldots x_\ell $.
\end{proof}

\section{Generalized diffusive contact process (GDCP)}\label{sec:GDCP}

In this section, we introduce a generalized diffusive contact (GDCP) 
on $\Lambda_N$ which also 
exhibits a factorized duality property, similar to the one for the DCP. For 
a specific choice of the parameters of the GDCP on a finite one-dimensional 
lattice with reservoirs, 
one can extract an explicit expression of the one-point function for 
its invariant measure. This is in contrast with the DCP, where we were not able 
to extract explicit expressions of the correlation functions from the duality 
result provided by Theorem \ref{Dual}. For time-dependent correlations, 
duality leads to a closed system of first-order ordinary differential equations 
with constant coefficients. This is demonstrated in explicit form for 
the one-point function.

\subsection{The model}\label{subsec:modelGDCP}
Here, as for the DCP, particles evolve on the open one-dimensional 
finite lattice $\Lambda_N$. 
The reservoir dynamics is the same as in Section \ref{dcp section}, given by 
the generators $\mathcal{L}_{-}$ and $\mathcal{L}_{+}$ defined in \eqref{leftres} 
and $\eqref{rightres}$. The exclusion dynamics is the same as for the DCP, 
given by 
the generator $\mathcal{L}^{\rm SSEP}$ defined in \eqref{exclusion generator}.

What differs is the contact dynamics which is here more general.
The birth rate is $\lambda > 0$, while the death rates, instead of being $1$,
depend on the occupation variables in the nearest neighbor sites: 
a particle dies with rate $2\mu_1$ when both neighbors are empty,
with rate $\mu_1+\mu_2$ when one of the two neighbors is empty and the other
is occupied, and with rate $2\mu_2$ when both neighbors are occupied. We refer \eqref{rates} for a description of the rates in a tabular form. The generator of the generalized diffusive contact processes is given by 
\begin{equation}\label{generator_opengeneralizeddiffusivecontact}
\mathcal{L}^{\rm GDCP}= \LL_{-} + \LL^{\rm GCP}+ \mathcal{D}\LL^{\rm SSEP} + \LL_{+}.
\end{equation} 
Above, $\LL_{-}$, resp. $\LL_{+}$, is the left, resp. right hand side 
reservoir generator, 
and is given by \eqref{leftres}, resp. \eqref{rightres}, with the boundary rates 
$\widetilde{\alpha}, \widetilde{\beta}, \widetilde{\gamma}, \widetilde{\delta}$ 
for the GDCP, 
instead of $\alpha, \beta, \gamma, \delta$ for the DCP. As before,  
$\mathcal{D}\geq 0$ is the 
diffusion parameter which tunes the exclusion dynamics described by 
$\LL^{\rm SSEP}$ in 
equation \eqref{exclusion generator}. While $\LL^{\rm GCP}$ is the generator of the 
generalized contact process, whose action on a given function
$f:\Omega_N \rightarrow \R$, is, for $\eta\in\Omega_N$
\begin{equation}\label{gen:GCP}
    \LL^{\rm GCP} f(\e) = \sum_{x=1}^{N-1}\LL_{x,x+1}^{\rm GCP}f(\e) ,
\end{equation}
with the bond generator 
\begin{equation}\label{genbond:GCP}
\begin{split}
        \LL_{x,x+1}^{\rm GCP}f(\e) &= \e_x\big[
        \mu_2\eta_{x+1}+\mu_1(1-\eta_{x+1})\big]\big[f(\e^x)-f(\e) \big]
        + \lambda (1-\e_x)\e_{x+1}\big[f(\e^x)-f(\e) \big]\\
        & + \e_{x+1}\big[
        \mu_2\eta_{x}+\mu_1(1-\eta_{x})\big]\big[f(\e^{x+1})-f(\e) \big]
        + \lambda (1-\e_{x+1})\e_{x}\big[f(\e^{x+1})-f(\e) \big].
\end{split}
\end{equation}
That is,
\begin{equation}\label{generalizedcontact generator}
\begin{split}
        &\LL^{\rm GCP} f(\e)  
        =\Big[ \eta_1 \big( \mu_1 + \left( \mu_2 - \mu_1 \right) \eta_2   \big) 
         + \lambda\big(1-\eta_1 \big)\eta_2 \big]\big[f(\e^1)-f(\e) \Big]\\
        &+\sum_{x=2}^{N-1} \Big [\eta_x \big( 2 \mu_1 
        + \left( \mu_2 - \mu_1 \right) \left( \eta_{x-1} + \eta_{x+1}  \right) \big) 
        + \lambda\big(1-\eta_x \big)\big(\eta_{x-1} 
        + \eta_{x+1} \big) \Big] \big[f(\e^x)-f(\e) \big]\\ 
        &+ \Big[ \eta_N \big( \mu_1 + \left( \mu_2 - \mu_1 \right) \eta_{N-1}   \big) 
         + \lambda\big(1-\eta_N \big)\eta_{N-1} \big]\big[f(\e^N)-f(\e) \Big] \;.
\end{split}
\end{equation}

\begin{rem}\label{RecoverDCP_GDCP}
\begin{itemize}
    \item[(1)] One can recover an open DCP with boundary parameters 
    $\alpha, \beta, \gamma, \delta$, 
    from an open GDCP with boundary parameters 
    $\widetilde{\alpha}, \widetilde{\beta}, \widetilde{\gamma}, \widetilde{\delta}$,
     and death parameters $\mu_1,\mu_2$, by taking
\begin{equation}\label{eq:RecoverDCP_GDCP}
    \mu_1=\mu_2=\frac{1}{2},~\widetilde{\alpha}=\alpha,~\widetilde{\gamma}
    =\gamma + \frac{1}{2},~ \widetilde{\delta}=\delta,~\text{and}~\widetilde{\beta}
    = \beta +\frac{1}{2} .
\end{equation}
\item[(2)] 
For $\mu_2=0$, the GDCP reduces to a biased voter model with stirring
by symmetric simple exclusion,
and when $\mu_2=0$ and $\mu_1 = \lambda$,
this is the usual voter model with stirring, studied in \cite{Belitsky2001}
for asymmetric simple exclusion with step initial condition on $\Z$, 
and \cite{Krebs2003} for totally asymmetric simple exclusion on $\Lambda_N$
and open boundaries with $\widetilde{\gamma}=\widetilde{\delta}=0$. 
\end{itemize}
\end{rem}

As for the DCP, since the state space is finite, one can encode 
the dynamics in a finite sized matrix. In this setting, the intensity matrix 
of the GDCP in the bulk writes as the following $2^N$-sized  matrix

\begin{equation}\label{L-GDCP}
L^{\rm GCP} +  \mathcal{D}L^{\rm SSEP}= \sum_{x=1}^{N-1} \ell_x,
\end{equation}
where for $x=1, \ldots N-1$
\begin{equation}\label{ell-x}
\ell_x := \mathds{1}^{\otimes (x-1)}\otimes \ell \otimes \mathds{1}^{\otimes (N-1-x)},
\end{equation}
with the two-dimensional unit matrix $\mathds{1}$, and the local intensity matrix
\begin{equation}\label{ell}
\ell = \begin{pmatrix}
0 &  0 & 0 & 0  \\
\mu_1  & -( \mu_1 + \mathcal{D} + \lambda) & \mathcal{D} & \lambda \\
\mu_1  & \mathcal{D} & -( \mu_1 + \mathcal{D} + \lambda) & \lambda  \\
0 & \mu_2 &\mu_2 & -2\mu_2 
\end{pmatrix} .
\end{equation}
The intensity matrices of the boundary dynamics are the same as 
for the DCP dynamics, see (1) in Remark \ref{RecoverDCP_GDCP}.

The bond transition rates for neighboring sites $(x,x+1)$ 
can be schematically represented in tabular form as
\begin{equation}\label{rates}
\begin{array}{llll}
Initial & & Final & Rate \\
\hline\hline \\[-2mm]
0 1 & \rightarrow & 0 0 & \mu_1 \\[2mm]
0 1 & \rightarrow & 1 0 & \mathcal{D} \\[2mm]
0 1 & \rightarrow & 1 1 & \lambda \\[2mm]
1 0 & \rightarrow & 0 0 & \mu_1 \\[2mm]
1 0 & \rightarrow & 0 1 & \mathcal{D} \\[2mm]
1 0 & \rightarrow & 1 1 & \lambda \\[2mm]
1 1 & \rightarrow & 0 1 & \mu_2 \\[2mm]
1 1 & \rightarrow & 1 0 & \mu_2 
\end{array}
\end{equation}

\subsection{Duality results for the GDCP}\label{subsec:dualGDCP}
In this section we establish in Theorem \ref{dualityGCDmodel}
a duality relation for the GDCP. 
A special case, derived in Corollary \ref{Annihilating} 
will be of interest for applications, as
it allows to find a dual process with no birth rate. This means 
that the sum in the equation \eqref{lcorrelations2} for correlations  
is finite, and can be explicitly computed for a general bulk of $N$ sites. 

\begin{thm} \label{dualityGCDmodel}
Assume $\mathcal{D} + \mu_1 - \mu_2 \geq 0$ and 
$\lambda + \mu_2 - \mu_1\geq 0$. 
The open GDCP with generator 
 \eqref{generator_opengeneralizeddiffusivecontact} 
 is dual with a purely absorbing GDCP with generator
\begin{equation}
\label{dualgenerator_opengeneralizeddiffusivecontact}
   \mathcal{L}^{\rm GDCP,Dual}= \LL_-^{\rm Dual}
    + \LL^{\rm GCP,Dual}+ \mathcal{\widehat D}\LL^{\rm SSEP} 
    + \LL_+^{\rm Dual},
\end{equation} 
w.r.t. the duality function \eqref{globaldualityfct} obtained for the open 
diffusive contact process. We have that
$\LL^{\rm GCP,Dual}$ is the generator 
\eqref{generalizedcontact generator} of the generalized contact process 
with birth rate $\widehat \lambda = \lambda + \mu_2 - \mu_1$ 
and death rates exchanged, that is, 
\begin{equation}\label{dual_generalizedcontact generator}
\begin{split}
   &  \LL^{\rm GCP, Dual} f(\xi)  
   = \\ & \sum_{x=2}^{N-1} \Big [\xi_x \big( 2 \mu_2 
   + \left( \mu_1 - \mu_2 \right) \left( \xi_{x-1}
    + \xi_{x+1}  \right) \big) + \left( \lambda + \mu_2 
    - \mu_1 \right)\big(1-\xi_x \big)\big(\xi_{x-1}
     + \xi_{x+1} \big) \Big] \Big[f(\xi^x)-f(\xi) \Big]\\ 
        &+ \Big[ \xi_1 \big(  \mu_2 + \left( \mu_1 - \mu_2 \right) 
        \xi_2   \big) 
         +  \left( \lambda + \mu_2 - \mu_1 \right) \big(1-\xi_1 \big)
         \xi_2 \big]\Big[f(\xi^1)-f(\xi) \Big]\\
        &+ \Big[ \xi_N \big( \mu_2 + \left( \mu_1 - \mu_2 \right) \xi_{N-1}   \big) 
         +  \left( \lambda + \mu_2 - \mu_1 \right) \big(1-\xi_N \big)\xi_{N-1} 
         \big]\Big[f(\xi^N)-f(\xi) \Big] \; ,
\end{split}
\end{equation}
$ \widehat{\mathcal{D}}= \mathcal{D} + \mu_1 - \mu_2 $ is 
the diffusion parameter of the exclusion dynamics, and $\LL_-^{\rm Dual}$, 
$\LL_+^{\rm Dual}$ are defined in \eqref{Ldualab}, \eqref{Ldual+}
with rates $\widetilde \alpha, \widetilde \beta, \widetilde \gamma, 
\widetilde \delta$, 
in place of $\alpha, \beta, \gamma, \delta$.
\end{thm}

The proof of Theorem \ref{dualityGCDmodel} is given in subsection
\ref{subsec:proofs-gdcp}.
A special case of the above setting is elucidated in the Corollary below.

\begin{cor}[Annihilating dual process]\label{Annihilating}
Under the same hypothesis as before, setting $\mu_1 = \lambda + \mu_2$, 
the dual GDCP has birth rate $\widehat \lambda = 0$, i.e. no particle is ever created. 
The intensity matrix associated to the local dual generator is given by
\begin{equation} \label{anndual}
\ell^{\rm Dual}
=  \left(\ba{cccc}
0 & 0 & 0 & 0 \\
{\mu}_2 & -({\mathcal{D}}+{\lambda}+
{\mu}_2) &\mathcal{D} + \lambda & 0 \\
\mu_2 & \mathcal{D} + \lambda & -({\mathcal{D}}+{\lambda}+{\mu}_2) & 0 \\
0 & \lambda + {\mu}_2 &  \lambda + {\mu}_2 & - 2 ( \lambda + {\mu}_2) 
\ea\right) \;.
\end{equation} 
\end{cor}
\noindent
We remark that no such dual process can be found for the DCP since the condition
\begin{equation*}
    \mu_1 = \lambda + \mu_2
\end{equation*}
is not satisfied.

\subsection{Application of duality: computing correlations}\label{subsec:compcor}
As for the DCP, the dynamics of the GDCP is irreducible and the state space is finite 
so there is a unique invariant measure $\nu^{\rm GDCP}$ for the GDCP, which depends 
on the birth, death, exchange and boundary rates. For a particular choice 
of these rates, 
the one-point function of $\nu^{\rm GDCP}$ can be explicitly computed, 
see Proposition \ref{explicit_one_point_GDCP}. Furthermore, some information 
on the time-dependent one-point function for arbitrary initial 
distributions can be obtained using duality. 
To this end,
consider $\xi = \delta_{x_1, \ldots, x_K }$ namely the dual 
configuration with $K$ particles respectively 
on sites $0 < x_1 < \ldots < x_K < N+1 $ and none elsewhere,
 then the duality function reads
\begin{eqnarray}\label{thisdualfct}
    D(\eta, \delta_{x_1, \ldots, x_K } ) 
    = \prod_{x \in A(\xi) } \left( 1- \eta_x \right) 
    =  \prod_{i=1  }^{K} \left( 1- \eta_{x_{i}}\right),
\end{eqnarray}
where $A(\xi)=\lbrace x_1, \ldots, x_K \rbrace $. 
Then one can write the time-dependent 
$K$-points correlations via duality as
\begin{align*}
   & \mathbb{E}_{\eta} 
   \left[ D \left( \eta(t), \delta_{x_1, \ldots, x_K }  \right) \right] 
   =     \mathbb{E}_{\eta} \left[ \prod_{x \in A(\xi) } 
   \left( 1- \eta_x(t) \right) \right]
   =
     \mathbb{E}_{\delta_{x_1, \ldots, x_K}} \left[ c_-^{\xi_0(t)} c_+^{\xi_{N+1}(t)}\prod_{x \in A(\xi (t)) } 
     \left( 1- \eta_x \right)  \right]  \\ 
     &  = \sum_{n_{-}=0}^\infty \sum_{n_{+}=0}^\infty \sum_{\zeta\in\Omega_N} 
     c_{-}^{n_{-}} c_{+}^{n_{+}} \prod_{i=1}^{M} (1-\zeta_{y_i}) p_t\big((0;x_1,\dots,x_K;0), (n_-;y_1,\dots,y_M;n_+) \big)
\end{align*}
where $p_t\big((0;x_1,\dots,x_K;0), (n_-;y_1,\dots,y_M;n_+) \big)$ is
the probability that, starting with a dual configuration $\delta_{x_1, \ldots, x_K }$, 
at time $t$ there are $n_-$ dual particles absorbed 
in the left sink in site $0$, $n_+$ 
dual particles absorbed in the right sink in site $N+1$, 
and a dual particle in sites $y_1, \ldots, y_M$.

\subsubsection{Invariant measure}
For $1\leq \ell \leq N,$ the $\ell$-point correlation function
 $\rho^{\rm GDCP}$ of $\nu^{\rm GDCP}$ is the function 
 $\rho_{\ell}^{\rm GDCP}$ defined by
\begin{equation}\label{ell_point_corr_def_GDCP}
    \forall x_1 <\ldots < x_{\ell} \in \Lambda_N,~ \rho_{\ell}^{\rm GDCP}(x_1,\ldots,x_{\ell}) 
    = \mathbb{E}_{\nu^{\rm GDCP}}\big[\e_{x_1}\ldots \e_{x_{\ell}} \big].
\end{equation}
From now on, we denote by
\begin{equation}\label{def:c0-cN+1_GCDP}
\widetilde c_{-}:= \frac{\widetilde \gamma}{\widetilde \alpha + \widetilde \gamma},~ ~ 
\text{and}~ ~ \widetilde c_{+} :
= \frac{ \widetilde \beta}{\widetilde \beta + \widetilde \delta}~. 
\end{equation}
As for the DCP, the dual $(\xi(t))_{t\geq 0}$ 
of the GDCP (we keep the same notation $\xi$ 
to refer to the dual process) becomes extinct almost surely and, recall that  
the total number of particles absorbed on the left, resp. right hand side boundary is:
\begin{equation}\label{def:absorbedparticles_GDCP}\xi_0(\infty) 
:= \underset{t\rightarrow \infty}{\lim}\xi_{0}(t),~ ~ \text{resp.}~ ~ \xi_{N+1}(\infty) 
:=\underset{t\rightarrow \infty}{\lim}\xi_{N+1}(t) , 
\end{equation}
where the limits are almost sure.\medskip

Following the same lines as for the proof of Proposition \ref{any_point_label}, 
one can prove the following general formula for the correlation functions of the GDCP:
\begin{prop}\label{any_point_label_GDCP} For any $1\leq \ell \leq N$, and any 
$1\leq x_1<...<x_{\ell}\leq N$,
\begin{equation}\label{lcorrelations3_GDCP}
\begin{split}
        \mathbb{E}_{\nu^{\rm GDCP}}\Big[ \prod_{j=1}^{\ell} \left( 1 -\eta_{x_j} \right) \Big] 
        &= \sum_{k=0}^{\ell}(-1)^k \sum_{1\leq i_1<\ldots <i_k\leq \ell}\rho_k^{\rm GDCP}(x_{i_1},\ldots ,x_{i_k}) \\
        &=\sum_{m,n\geq 0} \mathbb{P}_{\delta_{x_1,x_2,\ldots,x_{\ell}}}\Big[\xi_0(\infty)=m, \xi_N(\infty)
        =n \Big] \widetilde{c}_{-}^{ \ m} \widetilde{c}_{+}^{ \ n},
\end{split}
\end{equation} 
\end{prop}
In the particular setting of Corollary \ref{Annihilating}, the one-point  function 
can be computed explicitly. 
\begin{prop}\label{explicit_one_point_GDCP}
    Consider a GDCP with birth rate $\lambda>0$, 
    diffusion coefficient $\mathcal{D} >0$, 
    boundary rates $\widetilde{\alpha}, \widetilde{\beta}, 
    \widetilde{\gamma}, \widetilde{\delta}$ 
    and death rates $\mu_1$ and $\mu_2$, with $\mu_1=\mu_2 + \lambda$. Then, 
\begin{equation*}
    \rho_{1}^{\rm GDCP}(x) = u_x(1-\widetilde c_{-}) + v_x(1-\widetilde c_{+}),
\end{equation*}
where with the constants 
\begin{equation}\label{abcddef}
\begin{split}
        \widetilde{a} & := \frac{\widetilde \alpha 
        + \widetilde \gamma}{\widetilde \alpha 
        + \widetilde \gamma + \mathcal{D}+ \lambda+ \mu_2}, \quad 
        \widetilde{b} 
        := \frac{\mathcal{D}+ \lambda }{\widetilde \alpha 
        + \widetilde \gamma + \mathcal{D}+ \lambda+ \mu_2}, \\
\widetilde{c} & :=  \frac{\widetilde \beta 
+ \widetilde \delta}{\widetilde \beta + \widetilde \delta 
+ \mathcal{D}+ \lambda+ \mu_2}, \quad
\widetilde{d} :=  \frac{\mathcal{D}+ \lambda }{\widetilde \beta
 + \widetilde \delta + \mathcal{D}+ \lambda+ \mu_2}
\end{split}
\end{equation}
and
\begin{eqnarray}
A & := &
\frac{\mathcal{D}+ \lambda }{\mathcal{D}+ \lambda+\mu_2}, \quad
r_{\pm} \, := \, \frac{1 \pm \sqrt{1-A^2}}{A} \label{rootsdef},\\
B_N & := & r_{-} (1 - \widetilde{b} r_{-}) 
(1 - \widetilde{d} r_{+}^{-1}) + r_{-}^N r_{+}^{1-N} 
(\widetilde{b} r_{+} - 1) (1 - \widetilde{d} r_{-}^{-1}) \label{BNdef}, \\
B'_N & := & \widetilde{a} (\widetilde{c} N+ 1-\widetilde{c})
 + (1-2 \widetilde{a}) \widetilde{c},
\label{BN2def}
\end{eqnarray}
the terms $u_x$ and $v_x$ are given by, for all $x\in\{1,\dots,N\}$,
\begin{eqnarray}
u_x & = & \begin{cases}
\displaystyle \frac{\widetilde{a}}{B_N} [(1 - \widetilde{d} r_{+}^{-1})  r_{-}^x 
+  (\widetilde{d} r_{-}^{-1}-1) r_{-}^N  r_{+}^{x-N} ] & \mu_2 > 0\\[4mm]
\displaystyle \frac{\widetilde{a}}{B'_N}  [1-\widetilde{c} + \widetilde{c}(N-x)]  & \mu_2 = 0,
    \end{cases}
\label{ux} \\
v_x & = & \begin{cases}
\displaystyle \frac{\widetilde{c}}{B_N} [(\widetilde{b} r_{+} - 1)]  r_{+}^{1-N} r_{-}^x + 
r_{-} (1 - \widetilde{b} r_{-}) 
r_{+}^{x-N} & \mu_2 > 0\\[4mm]
\displaystyle \frac{\widetilde{c}}{B'_N}
[1-\widetilde{a} + \widetilde{a} (x-1)] & \mu_2 = 0.
    \end{cases}
\label{vx} 
\end{eqnarray}
\end{prop}

\begin{rem}\label{rem:5}

(a) For $\mu_2>0$ one finds that the stationary bulk density 
\begin{equation}\label{rem5:eq1}
\rho^{\rm GDCP}_{bulk}(s) := \lim_{N\to\infty}\rho_1([sN])
\end{equation}
vanishes for all $s\in(0,1)$. A non-trivial  exponentially decaying density profile appears near the boundaries.

(b) In contrast, for $\mu_2=0$ a linear bulk density profile emerges, 
\begin{equation}\label{rem5:eq2}
\rho^{\rm GDCP}_{bulk}(s) 
= \frac{\widetilde \alpha}{\widetilde \alpha + \widetilde \gamma} (1-s) + \frac{\widetilde \beta}{\widetilde \beta + \widetilde \delta} s
= \frac{\widetilde \alpha}{\widetilde \alpha + \widetilde \gamma} 
- \frac{\widetilde \alpha  \widetilde \delta - \widetilde \beta \widetilde \gamma}{(\widetilde \alpha 
+ \widetilde \gamma)(\widetilde \beta + \widetilde \delta)} s
\end{equation}
This phenomenon is reminiscent of a superposition of shock measures appearing in
the voter model with {\it totally asymmetric} stirring (rather than symmetric
stirring as in the present case) on $\Lambda_N$ and open boundaries with 
$\widetilde{\gamma}=\widetilde{\delta}=0$ \cite{Krebs2003}. The invariant
measure of that model is a convex combination of shock product measures with
marginals $\mu(\eta_x) = 1$ for $1\leq x < x_s$ and $\mu(\eta_x) = 0$ for $x_s \leq x \leq N$
and uniform distribution of the microscopic shock position $x_s\in\Lambda_N$
which leads to a linear stationary density profile
and a simple form of the two-point
correlation function computed in \cite{Jafapour2004}.
For the present more complicated case of {\it symmetric} stirring and {\it arbitrary} 
boundary parameters we leave a detailed investigation of the two-point correlation function for future work.

(c)
Moreover, computing higher order correlations is in general a complicated task. 
For these models, the difficulty lies in the fact that if there are 
two or more dual particles, the death rates depend on the configuration.
We leave this for future work.
\end{rem}

\subsubsection{Time-dependence of the one-point function}\label{subsec:timedep}
For the choice $\mu_1 = \lambda + \mu_2$ the absence of 
particle birth or insertion in the dual process guarantees that the duality 
function yields a finite inhomogeneous system of ordinary differential equations
(ODE's) for the one-point function.

\begin{prop}\label{prop:time-dep}
We assume that $\mu_1 = \lambda + \mu_2$. Denoting by 
$\overline{\eta}_x := \eta_x - \rho_{1}(x)$ centered variables, 
and defining the function 
\begin{equation}\label{def:gxt}
g(x,t) := e^{2\mu_2 t} \exval{\overline{\eta}_x},
\end{equation}
where $\exval{\cdot}$ means expectation at time $t$ 
for an arbitrary initial measure $\nu$, we have
\begin{eqnarray}\label{lineareq3}
\ddt g(x,t) & = & (\mathcal{D} + \lambda)(g(x+1,t)
+ g(x-1,t) - 2 g(x,t)) , \quad 1 < x < N, \nonumber \\
\ddt g(1,t)  & = & (\mathcal{D} + \lambda)(g(2,t)
- g(1,t)) - (\widetilde{\alpha}+\widetilde{\gamma}-\mu_2) g(1,t),  \\
\ddt g(N,t) & = & (\mathcal{D} + \lambda)(g(N-1,t)
- g(N,t)) - (\widetilde{\beta}+\widetilde{\delta}-\mu_2) g(N,t) . \nonumber
\end{eqnarray}
\end{prop} 
In the system \eqref{lineareq3}, we recognize
the ODE's for the time-dependent centered one-point function
of the open SSEP with diffusion coefficient $\widehat{D}=\mathcal{D} + \lambda$, and reservoir rates
$\widehat{\alpha},\widehat{\beta},\widehat{\gamma},\widehat{\delta}$, satisfying 
\begin{equation}\label{recognize}
    \widehat{\alpha} + \widehat{\gamma}
    =\widetilde{\alpha}+\widetilde{\gamma}-\mu_2,~\text{and } \widehat{\beta}+\widehat{\delta}
    =\widetilde{\beta}+\widetilde{\delta}-\mu_2.
\end{equation}
Following \cite{ExactSchutz}, this system can be solved in a closed form 
by a discrete Fourier transformation with a
reflected wave, in the special cases obtained from the four combinations 
of reservoir parameters given by
$\widetilde{\alpha}+\widetilde{\gamma} \in \{\mu_2,\mu_2 + \mathcal{D} + \lambda\}$ and
$\widetilde{\beta}+\widetilde{\delta} \in \{\mu_2,\mu_2 + \mathcal{D} + \lambda\}$.
We find surprising that the system of equations \eqref{lineareq3} matches the one for the space time empirical profile of an open SSEP, see equation (19) and below in \cite{gonccalves2017hydrodynamics}.
In particular, the choice 
$\widetilde{\alpha}+\widetilde{\gamma} 
= \widetilde{\beta}+\widetilde{\delta}= \mu_2$ corresponds to
$\widehat{\alpha}=\widehat{\beta}=\widehat{\gamma}=\widehat{\delta}=0$ 
for the boundary 
parameters of the SSEP (i.e. $\kappa = 0$ in 
\cite{gonccalves2017hydrodynamics}) and leads to the same system of equations, 
both for the SSEP with reflecting boundary which conserves the total number 
of particles and for our GCDP which does not conserve the total number of 
particles due to positive birth and death rates. We believe it would be 
interesting to prove the hydrodynamic limit for this non-conservative process 
with open boundary, in the same spirit as for the open finite volume SSEP.

\subsection{Proofs for Section \ref{sec:GDCP}}\label{subsec:proofs-gdcp}

\subsubsection{Proofs for Subsection \ref{subsec:dualGDCP}}
\label{subsec:pfs}

\begin{proof}[Proof of Theorem \ref{dualityGCDmodel}]  
In order to show the matrix duality relation of equation 
\eqref{matrixnotation}, as the generator acts on two sites, it is enough to show that
\begin{equation}
\label{localduality}
\ell G^{\otimes 2} = G^{\otimes 2} \left( \ell^{ \rm Dual} \right)^{T},
\end{equation} 
for $G,\ell$ defined respectively in \eqref{genericG} and \eqref{ell}; 
we have 
\begin{equation}\label{forG}
G^{\otimes 2} 
=\left(\ba{cccc}
a^2 & ab & ba & b^2 \\
ac & ad & bc & bd \\
ca & cb & da & db\\
c^2 & cd & dc & d^2
\ea\right) .
\end{equation} 
For any choice of the parameters $a,b,c,d$, this defines a 
self-duality function for the SSEP (see Remark \ref{parameters}), 
and we only have to take into consideration the action of the GCP.
The left hand side of equation \eqref{localduality} becomes
\begin{equation} \label{dua}
\ell G^{\otimes 2} 
 = \left(\ba{cccc}
0 & 0  & 0 & 0 \\
x [a \mu_1 - c \lambda] & a y \mu_1 - d x \lambda - \Delta \mathcal{D} & b x \mu_1 - c y \lambda  + \Delta \mathcal{D} & y [b \mu_1 - d \lambda] \\
x [a \mu_1 - c \lambda] & b x \mu_1 - c y \lambda  + \Delta \mathcal{D} & a y \mu_1 - d x \lambda - \Delta \mathcal{D} &   y [b \mu_1 - d \lambda]  \\
2c x \mu_2 & (d x + c y) \mu_2 & (d x + c y) \mu_2 & 2dy \mu_2
\ea\right),
\end{equation}
where $\Delta := ad-bc$, $x:= a-c$ and $y:=b-d$.
For the right hand side of \eqref{localduality}, we first consider 
a local generator $\ell^{\rm Dual}$ which describes a GDCP with possibly 
different rates: $\widetilde{\mathcal{D}}$ is the diffusion coefficient, 
$\widetilde{\lambda} $ the birth rate while $\widetilde{\mu}_1$ 
and $\widetilde{\mu}_2$ 
are the death rates. This gives the following dual local intensity matrix
\begin{equation}\label{duallocal}
\ell^{\rm Dual}
=  \left(\ba{cccc}
0 & 0 & 0 & 0 \\
\widetilde{\mu}_1 & -(\widetilde{\mathcal{D}}+\widetilde{\lambda}
+\widetilde{\mu}_1) & \widetilde{\mathcal{D}} & \widetilde{\lambda} \\
\widetilde{\mu}_1 & \widetilde{\mathcal{D}} 
& -(\widetilde{\mathcal{D}}+\widetilde{\lambda}+\widetilde{\mu}_1) 
& \widetilde{\lambda} \\
0 & \widetilde{\mu}_2 & \widetilde{\mu}_2 & - 2\widetilde{\mu}_2
\ea\right) \;.
\end{equation} 
Computing the right hand side of \eqref{localduality}, we get
\begin{align} \label{dua2}
&G^{\otimes 2} \left(\ell^{\rm Dual} \right)^T = \\
&\left(\ba{cccc}
0 & z (a \widetilde{\mu}_1 - b \widetilde{\lambda}) & z (a \widetilde{\mu}_1 - b \widetilde{\lambda}) & 2  \widetilde{\mu}_2 bz \\
0 & aw \widetilde{\mu}_1 - d z \widetilde{\lambda} - \Delta \widetilde{\mathcal{D}} & cz \widetilde{\mu}_1 - bw \widetilde{\lambda} + \Delta\widetilde{\mathcal{D}} 
& (bw+dz) \widetilde{\mu}_2 \\
0 & cz \widetilde{\mu}_1 - bw \widetilde{\lambda} + \Delta\widetilde{\mathcal{D}}  & aw \widetilde{\mu}_1 - d z \widetilde{\lambda} - \Delta \widetilde{\mathcal{D}}
& (bw +dz) \widetilde{\mu}_2 \\
0 & w [c \widetilde{\mu}_1 - d \widetilde{\lambda}] &  w [c \widetilde{\mu}_1 - d \widetilde{\lambda}] & 2  d w \widetilde{\mu}_2
\ea\right)\;, \nonumber
\end{align} 
where $z= a-b$ and $w=c-d$.
Since we are aiming to match expressions \eqref{dua} and \eqref{dua2}, 
we first notice that equality in the first column requires $x=0$, i.e. $a=c$, 
while equality in the first row requires  $z=0$, i.e. $a=b$. Thus $a=b=c$ yields  
$\Delta = - a(a-d)$, and we are left with 
\begin{equation} \label{dua3}
\ell G^{\otimes 2} 
 = (a-d) \left(\ba{cccc}
0 & 0  & 0 & 0 \\
0 & a  \mu_1 +a \mathcal{D} & -a \lambda -a \mathcal{D} & a\mu_1 - d \lambda \\
0 & -a \lambda -a \mathcal{D} & a  \mu_1 +a \mathcal{D} &   a\mu_1 - d \lambda \\
0 & a \mu_2 & a\mu_2 & 2d \mu_2
\ea\right)
\end{equation}
for the left hand side, and 
\begin{align} \label{dua23}
G^{\otimes 2} \left(\ell^{\rm Dual} \right)^T = (a-d)
\left(\ba{cccc}
0 & 0 & 0 & 0 \\
0 & a \widetilde{\mu}_1 +a \widetilde{\mathcal{D}} & -a \widetilde{\lambda} -a \widetilde{\mathcal{D}} 
& a\widetilde{\mu}_2 \\
0 & -a \widetilde{\lambda} -a \widetilde{\mathcal{D}}  & a \widetilde{\mu}_1 +a \widetilde{\mathcal{D}} 
& a \widetilde{\mu}_2 \\
0 & a \widetilde{\mu}_1 - d \widetilde{\lambda}&  a \widetilde{\mu}_1 - d \widetilde{\lambda}& 2 d \widetilde{\mu}_2
\ea\right)
\end{align} 
for the right hand side.

Consider $d=0$ in both equations \eqref{dua3}  and    \eqref{dua23}.
Then, the identity holds if we set
 $\widetilde{\mu}_2 = \mu_1$ and $\widetilde{\mu}_1 = \mu_2$, 
 $\widetilde{\mathcal{D}}= \mathcal{D} + \mu_1 - \mu_2$ and 
 $\widehat \lambda = \lambda + \mu_2 - \mu_1$. Furthermore, 
 without loss of generality, choosing $a=1$, we are left with the duality matrix 
\begin{equation} \label{dmat}
G = \left(\ba{cc}
1 & 1 \\
1 & 0
\ea\right),
\end{equation}
which corresponds to the duality function
\begin{equation}\label{dmathfct}
\mathrm{G}(\eta,\xi) = \prod_{x \in \Lambda_N} (1-\eta_x)^{\xi_x}.
\end{equation}
Note that $G$ is the same duality function as in \eqref{bulkduality}
(see Remark \ref{parameters}). 
Since the bulk duality function is the same as the one for the 
diffusive contact process, we can extend the result to the GDCP 
with the corresponding boundary parameters. Hence,
the duality function has the same form as for the diffusive contact process, 
that is: $D:\Omega_N\times \Omega_N^{\rm dual} \rightarrow \R$ given by
\begin{equation} \label{thisone}
    D(\e,\xi) = \Big(\frac{\widetilde \gamma}{ \widetilde \alpha + \widetilde \gamma} \Big)^{\xi_0}\prod_{x \in \Lambda_N} (1-\e_x)^{\xi_x}\Big(\frac{  \widetilde \beta}{\widetilde \beta + \widetilde \delta} \Big)^{\xi_{N+1}}.
\end{equation}
Notice that for $\mu_1=\mu_2=\frac{1}{2}$,  we recover Theorem \ref{Dual}.
\end{proof}
The proof of Corollary \ref{Annihilating} follows from Theorem \ref{dualityGCDmodel}.

\subsubsection{Proofs for subsection \ref{subsec:compcor}}
\label{subsec:proofs-compcor}

\begin{proof}[Proof of Proposition \ref{any_point_label_GDCP}]
In this setting, by Corollary \ref{Annihilating}, the dual process of the GDCP 
is a GDCP with no birth rate, diffusion coefficient $\mathcal{D} + \lambda$, 
boundary rates 
$0, \widetilde \alpha + \widetilde \gamma, 0, \widetilde \beta + \widetilde\delta$
 (purely absorbing boundaries), and death rates $\mu_2$ and $\lambda + \mu_2$. 
 Thus, the generator of the dual process is given by
    \begin{equation*}
    \mathcal{L}^{\rm GDCP, Dual}=\LL_-^{\rm Dual} 
    + \LL^{\rm GCP,Dual}+ \big(\mathcal{D}+ \lambda\big)\LL^{\rm SSEP} 
    + \LL_{+}^{\rm Dual},
\end{equation*}
where we recall that $\LL_-^{\rm Dual}$ and $\LL_+^{\rm Dual}$ 
are defined in \eqref{Ldualab} (with $\widetilde \alpha + \widetilde \gamma$ 
instead of $\alpha + \gamma$ and $\widetilde \beta + \widetilde \delta $ 
instead of $\beta + \delta$) and,
\begin{equation*}
\begin{split}
        \LL^{\rm GDCP,Dual} f(\xi) & = \sum_{x=2}^{N-1} \Big [\eta_x \big( 2 \mu_2 + \lambda  \left( \xi_{x-1} + \xi_{x+1}  \right) \big)  \Big] \big[f(\xi^x)-f(\xi) \big]\\ 
        &+ \Big[ \xi_1 \big( \mu_2 + \lambda \xi_2   \big)  \big]\big[f(\xi^1)-f(\xi) \Big]\\
        &+ \Big[ \xi_N \big( \mu_2 + \lambda  \xi_{N-1}   \big)   \big]\big[f(\xi^N)-f(\xi) \Big] \;.
\end{split}
\end{equation*}
By Theorem \ref{dualityGCDmodel} the duality function 
$D:\Omega_N\times \Omega_N^{\rm dual} \rightarrow \R$ 
is the same as for the DCP, that is,  given by \eqref{thisone}.
The same reasoning as in the proof of Proposition 
\ref{Proposition_one_point_DCP} yields that for $x\in \{1,\ldots,N\}$,
\begin{equation*}
    \begin{split}
        \mathbb{E}_{\nu^{\rm GDCP}}\big[1-\e_x] &= \mathbb{P}_{\delta_x}\big[\xi_0(\infty)=1\cap \xi_{N+1}(\infty)=0 \big]\widetilde c_{-}\\
        & + \mathbb{P}_{\delta_x}\big[\xi_0(\infty)=0\cap \xi_{N+1}(\infty)=1 \big]\widetilde c_{+}\\
        & +\Big(1- \mathbb{P}_{\delta_x}\big[\xi_0(\infty)=1\cap \xi_{N+1}(\infty)=0 \big]-\mathbb{P}_{\delta_x}\big[\xi_0(\infty)=0\cap \xi_{N+1}(\infty)=1 \big]\Big)\\
        &=: u_x\widetilde c_{-} + v_x\widetilde c_{+} + (1-u_x-v_x).
    \end{split}
\end{equation*}
Therefore,
\begin{equation*}
    \rho_{1}^{\rm GDCP}(x) = u_x(1-\widetilde c_{-}) + v_x(1-\widetilde c_{+}).
\end{equation*}
In this case the sum \eqref{lcorrelations3_GDCP} is finite due to the fact 
that the rate of birth is zero, so $\xi_0(\infty)$ and $\xi_{N+1}(\infty)$ 
can be at most equal to one.
\end{proof}

\begin{proof}[Proof of Proposition \ref{explicit_one_point_GDCP}]
To prove Proposition \ref{explicit_one_point_GDCP}, we are now 
left to compute $(u_x)_{1\leq x \leq N}$ and $(v_x)_{1\leq x \leq N}$. 
By conditioning on the first possible event (jump, birth, death, 
or absorption by a reservoir) of the process  starting from a particle 
at site $x$, we have the following discrete equations
\begin{itemize}
    \item For $(u_x)_{1\leq x \leq N}$:
\begin{equation*}
\begin{split}
u_1 & =  \widetilde{a} + \widetilde{b} u_2 \\
u_x & =   \frac{A}{2} (u_{x-1}+u_{x+1}) \quad \forall x\in \{2,\ldots ,N-1\}  \\
u_N & =  \widetilde{d} u_{N-1}, \nonumber
\end{split}
\end{equation*}
with the constants $\widetilde a, \widetilde b, \widetilde d,$ and $A$ 
defined in  \eqref{abcddef} and \eqref{rootsdef}.
\item For $(v_x)_{1\leq x \leq N}$:
\begin{equation*}
\begin{split}
v_1 & = \widetilde{b} v_2 \\
v_x & = \frac{A}{2} (v_{x-1}+v_{x+1}) \quad \forall x\in \{2,\ldots,N-1\}  \\
v_N & =\widetilde{c} + \widetilde{d} v_{N-1}, 
\end{split}
\end{equation*}
with the constants $\widetilde b, \widetilde c, \widetilde d,$ and $A$ 
defined in  \eqref{abcddef} and \eqref{rootsdef}.
\end{itemize}
Both recurrence relations are of the form
\begin{equation}\label{recurrence_relation}
\begin{split}
w_1 & = a + \widetilde{b}w_2 \\
w_x & = \frac{A}{2} (w_{x-1}+w_{x+1}) \quad \forall x\in \{2,\ldots,N-1\}  \\
w_N & = c + \widetilde{d}w_{N-1}
\end{split}
\end{equation}
with $a=\widetilde{a}$, 
$c=0$, 
for $(u_x)_{1\leq x \leq N}$ and $a=0$,
$c=\widetilde{c}$, 
for $(v_x)_{1\leq x \leq N}$. 
The properties of the solution of this recursion depend on 
$\mu_2$ as follows.\medskip

\noindent \underline{\textbf{Case 1: $\mu_2 > 0$}}\\
\noindent To solve the general recurrence relation, notice that the constants defined
in \eqref{rootsdef} satisfy $A<1$, $r_{-}<1$ and $r_{+}>1$.
Then, there are $p,q\in \R$ such that for $2\leq x \leq N-1$, 
    \begin{equation}\label{wx}
        w_x = pr_{-}^x + qr_{+}^x,
    \end{equation}
which follows from the bulk part of the 
recurrence relation \eqref{recurrence_relation}.
Furthermore, the recurrence relation involving sites 1 and 2 yields
$ pr_{-} + qr_{+} = a + \widetilde{b} [pr_{-}^2 + qr_{+}^2]$, therefore
$ p r_{-} (1 - \widetilde{b} r_{-})  
= a + q r_{+} (\widetilde{b} q r_{+} - 1)$. Since $0<r_{-}<1$
and $\widetilde{b}<1$, for both $u_x$ and $v_x$, one obtains
\begin{equation*}
p  = \frac{a + q r_{+} (\widetilde{b} r_{+} - 1)}{r_{-} (1 - \widetilde{b} r_{-})}.
\end{equation*}
In a similar fashion, the recurrence relation involving sites $N$ and $N-1$ yields
$ pr_{-}^N + qr_{+}^N = c + \widetilde{d} [pr_{-}^{N-1} + qr_{+}^{N-1}]$, 
therefore
$ p r_{-}^N (1 - \widetilde{d} r_{-}^{-1}) 
= c + q r_{+}^N (\widetilde{d} r_{+}^{-1} - 1)$. Since $r_{+}>1$
and $\widetilde{d}<1$, for both $u_x$ and $v_x$, one obtains
\begin{equation*}
q = \frac{c - p r_{-}^N (1 - \widetilde{d} r_{-}^{-1})}
{r_{+}^N (1 - \widetilde{d} r_{+}^{-1})}.
\end{equation*}
Finally, for $B_N$ defined in \eqref{BNdef}
one finds
\begin{eqnarray}\label{qp-q}
q & = & \frac{1}{B_N}
[c r_{-} (1 - \widetilde{b} r_{-}) + a r_{-}^N (\widetilde{d} r_{-}^{-1}-1)]
r_{+}^{-N} \\\label{qp-p}
p & = &  \frac{1}{B_N} [a (1 - \widetilde{d} r_{+}^{-1}) 
+ c r_{+}^{1-N} (\widetilde{b} r_{+} - 1)] .
\end{eqnarray}
This result yields \eqref{ux} and \eqref{vx} for $\mu_2>0$.
\medskip

\noindent\underline{\textbf{Case 2, $\mu_2 = 0$}}\\

\noindent In this case, $A=1$ and the bulk part of the recurrence relation \eqref{recurrence_relation} 
can be written $\Delta w_x = 0$, where $\Delta$ is the
discrete one-dimensional Laplacian. The general solution is the linear function
\begin{equation}\label{wx-again}
w_x = p' + q' x.
\end{equation}
The boundary condition at site 1 yields
$ p' + q' = a + \widetilde{b} (p' + 2 q')$ and therefore 
$p' (1-\widetilde{b}) = a + (2 \widetilde{b}-1) q'$.
Likewise, the boundary condition at site $N$ yields
$ p' + q' N = c + \widetilde{d} (p' + q' (N-1))$ and therefore
$ p' (1-\widetilde{d}) = c - q' [(1-\widetilde{d})N-\widetilde{d}]$. 
We then discuss the following cases:
\begin{itemize}
    \item[(i)] $\widetilde{b}=\widetilde{d}=1$. In this case, 
    we have closed boundary conditions, 
    and the empty lattice is the absorbing state, implying without further computation
    $u_x=v_x=\rho^{\rm GDCP}_1(x)=0$ for all $x\in\{1,\dots,N\}$.
    \item [(ii)] $\widetilde{b}=1$, $\widetilde{d}<1$. 
    This corresponds to a closed left boundary, 
    with $\widetilde{a}=0$, but with open right boundary, 
    leaving the invariant measure non-trivial. 
    The recurrence for $w_x$ yields $q'=0$ and one obtains
$u_x=0$, $v_x=1$ for all $x\in\{1,\dots,N\}$.
\item [(iii)]$\widetilde{d}=1$, $\widetilde{b}<1$. This corresponds 
to a closed right boundary, 
with $\widetilde{c}=0$, but with open left boundary, leaving 
the invariant measure non-trivial. 
The recurrence for $w_x$ yields $q'=0$ and one obtains
$u_x=1$, $v_x=0$ for all $x\in\{1,\dots,N\}$.
\item [(iv)]$\widetilde{d}<1$, $\widetilde{b}<1$. The recurrence yields
\begin{eqnarray}\label{q'p'-q'}
q'  & = & \frac{c (1-\widetilde{b}) - a (1-\widetilde{d})}{B'_N}  \\\label{q'p'-p'}
p' & = & \frac{a [(1-\widetilde{d})N+\widetilde{d}]  + c (2\widetilde{b}-1)}{B'_N},
\end{eqnarray}
for $B'_N$ defined in \eqref{BN2def}.
\end{itemize}
Hence, observing that for $\mu_2=0$, one has $1=\widetilde{a}+\widetilde{b}
=\widetilde{c}+\widetilde{d}$, we arrive at \eqref{ux} and \eqref{vx} for $\mu_2=0$,
which covers all four cases \textit{(i)-(iv)}.
\end{proof}

\begin{proof}[Proof of Proposition \ref{prop:time-dep}]
Assuming $\mu_1 = \lambda + \mu_2$ the duality function of the
GDCP yields the finite inhomogeneous system of ordinary differential equations
(ODE's)
\begin{eqnarray}\label{lineareq}
\ddt \exval{\eta_x} & = & (\mathcal{D} + \lambda)(\exval{\eta_{x+1}} 
+ \exval{\eta_{x-1}} - 2 \exval{\eta_{x}})
 - 2\mu_2 \exval{\eta_x}, \quad 1 < x < N \nonumber \\
\ddt \exval{\eta_1} & = & (\mathcal{D} + \lambda)(\exval{\eta_{2}} 
- \exval{\eta_{1}}) - (\mu_2+\widetilde{\alpha}
+\widetilde{\gamma}) \exval{\eta_1} + \widetilde{\alpha}, \\
\ddt \exval{\eta_N} & = & (\mathcal{D} + \lambda)(\exval{\eta_{N-1}} 
- \exval{\eta_{N}}) - (\mu_2+\widetilde{\beta}
+\widetilde{\delta}) \exval{\eta_N} + \widetilde{\delta} \;. \nonumber
\end{eqnarray}
The inhomogeneity arising from the constants $\widetilde{\alpha}$ and 
$\widetilde{\delta}$ can be removed by considering the centered variables
$\overline{\eta}_x := \eta_x - \rho_{1}(x)$, and by noting that invariance of the
measure yields
$\widetilde{\alpha} 
= (\mathcal{D} + \lambda + \mu_2+\widetilde{\alpha}+\widetilde{\gamma}) \rho_{1}(1) 
- (\mathcal{D} + \lambda)\rho_{1}(2) $ and 
$\widetilde{\delta} 
= (\mathcal{D} + \lambda + \mu_2+\widetilde{\beta}
+\widetilde{\delta}) \rho_{1}(N) - (\mathcal{D} + \lambda)\rho_{1}(N-1) $.
With the recursion \eqref{recurrence_relation}, which by linearity is also valid
for $\rho_{1}(x)$, we conclude that
\begin{eqnarray}\label{lineareq2}
\ddt \exval{\overline{\eta}_x} & = & (\mathcal{D} 
+ \lambda)(\exval{\overline{\eta}_{x+1}} 
+ \exval{\overline{\eta}_{x-1}} 
- 2 \exval{\overline{\eta}_{x}}) 
- 2\mu_2 \exval{\overline{\eta}_x}, \quad 1 < x < N ,\nonumber \\
\ddt \exval{\overline{\eta}_1} 
& = & (\mathcal{D} + \lambda)(\exval{\overline{\eta}_{2}} 
- \exval{\overline{\eta}_{1}})
 - (\mu_2+\widetilde{\alpha}+\widetilde{\gamma}) \exval{\overline{\eta}_1},  \\
\ddt \exval{\overline{\eta}_N}
 & = & (\mathcal{D} + \lambda)(\exval{\overline{\eta}_{N-1}} 
- \exval{\overline{\eta}_{N}}) 
- (\mu_2+\widetilde{\beta}+\widetilde{\delta}) \exval{\overline{\eta}_N},  \nonumber
\end{eqnarray}
that is \eqref{lineareq3}, which is a {\it homogeneous} system 
of first order ordinary differential equations.
\end{proof}  

\section{The susceptible-infectious-recovered (SIR) model}\label{sec:SIR}

We now consider another non conservative model, originally introduced 
in \cite{Kermack&al1927}, in terms of a nonlinear system of differential equations 
for the sizes of populations of three species of individuals which are 
subject to an infection/recovery mechanism.
The susceptible-infectious-recovered (SIR) model describes 
propagation of infections in the following sense:  a susceptible individual 
(characterized by its state $S$) can become an infectious individual 
(characterized by its state $I$), according to an infection rate $\beta$, 
if it is in contact with an infectious individual. An infectious individual recovers
 (then it is characterized by its state $R$) with recovery rate $\gamma$; 
 once an individual has recovered
 it stays immune, that is, it remains in state $R$ forever. 

Despite the simplicity of the original mean-field type model, it is known to capture 
important features of the temporal dynamics of an infection. However, 
only limited results 
are available if fluctuations (which inevitably occur in a real system) are taken
into account. This question was addressed in the SIR model of \cite{Schutzetal2008}
where particles evolve on a one dimensional space of sites, 
as in the diffusive contact process. 
However, contrary to the diffusive contact process 
studied in Section \ref{dcp section}), the model is defined on the 
infinite translation invariant lattice $\Z$, 
so that in particular, there are no reservoirs. 
Notice that the parameters $\beta$ and $\gamma$ 
for the SIR model have nothing to do with the reservoir parameters 
in the previous sections. 
The reason for which we keep this notation is that it is canonically 
used in the literature on the SIR model.

The stochastic evolution of the collection of particles in the system 
is governed by a Markov process  denoted by 
$\lbrace \eta_x(t), x\in \Z, t\geq 0 \rbrace$, 
with state space
\begin{equation}
    \mcb{S}:=\{S,I,R \}^{\mathbb{Z}},
\end{equation}
so that for $x\in \Z$, $a \in \{S,I,R\}$, 
$\eta_x = a $ means that  $x$ is in state $a$.
Using the notation in \cite{Schutzetal2008},
the (translation invariant)  transition rates  between nearest neighboring 
sites are given by
\begin{equation}\label{ratesSIR}
\begin{array}{llll}
Initial & & Final & Rate \\
\hline\hline \\[-2mm]
IS & \rightarrow & II & \beta \\[2mm]
SI & \rightarrow & II & \beta \\[2mm]
I & \rightarrow & R & \gamma 
\end{array}
\end{equation}
Note that this dynamics is not attractive (the necessary and sufficient
conditions required for attractiveness in \cite{Borrello} are not satisfied), 
so we cannot rely on monotonicity for this model.
\medskip
In view of the duality result for the SIR model, see Theorem \ref{SIRduality}, 
we give an alternative definition of the dynamics in terms of its generator 
$\mathcal{L}^{{\rm  SIR}}$. For $f: \mcb{S} \rightarrow \mathbb{R}$,
\begin{align}\nonumber
 \mathcal{L}^{{\rm  SIR}} f  (\eta) 
 = & \sum_{x\in \Z} \left\{
  \beta \  \eta_{x}^{I}\ \eta_{x+1}^{S} 
  \left[f(T_{x+1}^{I}\e) - f(\eta) \right] \right. 
  + \beta \  \eta_{x}^{I} \ \eta_{x-1}^{S} \left[f(T_{x-1}^{I}\e) - f(\eta) \right]\\
 &\qquad +  \left.\gamma \ \eta_{x}^{I} \ 
 \left[ f (T_x^{R}\e) - f(\eta) \right]\right\} \;, \label{sir gen}
\end{align}
where for $x,y\in \Z$ and $a\in \{S,I,R\}$
\begin{equation}\label{corresp}
    \e^{a}_x = \mathds{1}_{ \lbrace \e_x=a \rbrace},
\end{equation}
and $T_y^a$ is the operator acting on elements of $\mcb{S}$ 
which flips the state of $y$ into $a$, that is, for $z\in \Z$,
\begin{equation}\label{def:Ty} 
   (T_y^a\e)_z= \left\{
    \begin{array}{ll}
       \e_z ~ ~\text{if}~ z\neq y\\
       a ~ ~\text{if}~ z=y.
        \end{array}
\right.
\end{equation}
Our first goal is to find a duality relation for the SIR model 
with generator \eqref{sir gen}.

\subsection{Clusters as duality function} 
\label{sectioncluster}
This section is devoted to showing duality for the SIR model, 
for the purpose of studying the expected population size, 
and correlations as a function of time
for arbitrary initial distributions, rather than only for 
translation invariant ones, as studied in \cite{Schutzetal2008}. 
This is achieved using a 
duality relation between \textit{cluster functions} (defined below) with a bi-layered 
two-dimensional random walk on two copies of the semi-infinite lattice
$\N\times \Z$ that we shall label by $G$ and $J$ respectively. The random
walk is asymmetric on each lattice and lattice $G$ is absorbing, i.e., once
 the random walker has left lattice $H$ it cannot return to it. In addition, there is
a further (single) absorbing cemetery state that the random walker can reach from
lattice $G$.\\

To make this qualitative picture of the dual process precise and obtain
information about the original process, we use the notation 
 $ \langle \cdot \rangle_{\nu} : =  \mathbb{E}_\nu [ \ \cdot \ ] $ 
 to denote the expectation with respect to some 
initial distribution $\nu$, 
e.g. $\langle \eta_x^{a} (t) \rangle_{\nu} $ is the expected state 
at site $x$ and time $t$ with respect to $\nu$.
As already noticed in \cite{Schutzetal2008} quantities of interest are written 
in terms of high order correlation functions, called $n$-point cluster functions, 
for $n\in \N$, the set of strictly positive integers:
\begin{eqnarray}\label{def:G}
    G^{\nu}(r,n,t) & 
     :=& \left\langle \e^{I}_{r-1}(t)\left(\prod_{j=0}^{n-1}\e^S_{r+j}(t)\right)\e^{I}_{r+n}(t) \right\rangle_{\nu} \\\label{def:H}
        H^{\nu}(r,n,t) & 
       : =&\left\langle \left(\prod_{j=0}^{n-1}\e^S_{r+j}(t)\right)\e^{I}_{r+n}(t) \right\rangle_{\nu} \;,
\end{eqnarray}
where in what follows, to lighten the notation we do not write 
the dependence in time explicitly. 

Instead of the cluster $H$, it will be more convenient to define a different cluster 
which is of the same size as $G$. Namely, 
\begin{align}\label{def:J}
    J^{\nu}(r,n,t) & := 
    \left\langle \e^{R}_{r-1}(t)\left(\prod_{j=0}^{n-1}\e^S_{r+j}(t)\right)\e^{I}_{r+n}(t) \right\rangle_{\nu} \;. 
    \end{align}
    Using that for any $x \in \mathbb{Z}$, $\e_x^S + \e_x^R + \e_x^I =1 $, 
    the relation between these clusters is
\begin{equation}
J^{\nu}(r,n,t)  = H^{\nu}(r,n,t) - H^{\nu}(r-1,n+1,t) - G^{\nu}(r,n,t),
\end{equation}
so that given $J^{\nu}(r,n,t)$ and $G^{\nu}(r,n,t)$, the cluster function
$H^{\nu}(r,n,t)$ can be computed recursively.

In particular, one can gather information about the correlations of the SIR model 
by choosing short clusters, see also Remark \ref{clusterlenght} below.  
We are now ready to give the duality result for the SIR model. To this end, 
denote by \begin{equation}\label{def:Sdual}
\mcb{S}^{\rm dual}:=\big(\Z\times\N \times \{G,J\} \big) \cup \{\partial\},
\end{equation}
where $\partial$ will be a trap for the dual evolution.

\begin{thm} [Duality relation for SIR] \label{SIRduality}
The SIR model $(\e_t)_{t\geq 0}$ with generator \eqref{sir gen} is dual 
to a two dimensional biased random walk on two layers (see Figure \ref{DualSIR})
with duality function $d:\mcb{S} \times \mcb{S}^{\rm dual} \rightarrow \R$ given by: 
for $\e\in \mcb{S}$ and $\xi \in \mcb{S}^{\rm dual} $, 
if  $\xi=(r,n,i)\in \Z\times\N \times \{G,J\}$, 
  \begin{equation}\label{def:dualfct}
  \begin{split} 
      d(\e,(r,n,i))
      &:=\e^{I}_{r-1}\e^S_r\cdots \e^S_{r+n-1}\e^{I}_{r+n}\mathds{1}_{i=G}
      +\e^R_{r-1} \e^S_r\cdots \e^S_{r+n-1}\e^{I}_{r+n}\mathds{1}_{i=J}\\
        &=\left[\e^{I}_{r-1}\mathds{1}_{i=G}+\e^R_{r-1}\mathds{1}_{i=J}\right]
        \left(\prod_{j=0}^{n-1}\e^S_{r+j}\right)\e^{I}_{r+n},
  \end{split}
\end{equation}
and \begin{equation}\label{eq:dtrap}d(\e,\partial):=0.\end{equation}
The dual generator acts on local functions
$f:\mcb{S}^{\rm dual}\rightarrow \R$ as follows: 
for $\xi\in \mcb{S}^{\rm dual}$, if $\xi=\partial$,
\begin{equation}\label{eq:ftrap}\mathcal{L}^{\rm dual}f(\partial)=0,\end{equation}
while if $\xi=(r,n,i)\in \Z\times\N \times \{G,J\}$,
\begin{equation}\label{sirdualgen}
\begin{split}
        \mathcal{L}^{\rm dual}f(r,n,i) &= \beta\mathds{1}_{i=G}[ f(r-1,n+1,i) - f(r,n,i)\big]
         + \beta [ f(r,n+1,i) - f(r, n,i)\big]\\
        &+\gamma \mathds{1}_{i=J}[ f(r,n,\phi(i)) - f(r, n,i)\big]
        + 2\gamma \mathds{1}_{i=G}[ f(\partial) - f(r,n,i)\big],
\end{split}
\end{equation}
where we define the flip operator $\phi$ as
\begin{equation}\label{def:phiflip}
    \phi:\{G,J\} \rightarrow \{G,J\},
    ~ \text{such that }~ \phi(G)=J~ ~ \text{and}~ ~\phi(J)=G \;.
\end{equation}
\end{thm} 
\begin{rem}\label{rk:SIRdual}

By \eqref{eq:ftrap}, $\partial$ 
is indeed a trap: once the process reaches $\partial$, it remains there forever. 
Furthermore, the generator of the dual dynamics defines a bilayer random walk with the following transitions 
on the dual state space, illustrated in Fig. \ref{DualSIR}:
\begin{eqnarray}\label{SIRdual-1}
&&(r,n,i)\rightarrow (r,n+1,i) ~ \text{at rate}~ \beta 
\text{ for both layers: } i\in\{G,J\} \\\label{SIRdual-2}
&&(r,n,G)\rightarrow (r-1,n+1,G)~ \text{at rate}~ \beta\text{ only for layer } G. 
\end{eqnarray} 
 It is only possible to go from layer $J$ to layer  $G$, 
 (but not the other way around): 
\begin{eqnarray} \label{SIRdual-3}
&&
(r,n,J)\rightarrow (r,n,G)~ \text{at rate}~ \gamma, 
\end{eqnarray}
and absorption in the trap is only possible if the walker is in layer $G$: 
\begin{eqnarray}\label{SIRdual-4}
&&(r,n,G) \rightarrow \partial~ \text{at rate}~ 2 \gamma. 
\end{eqnarray}
Notice the non-translation invariance nature of the dynamics 
in the second transition \eqref{SIRdual-2}.

%
%
%
%
\end{rem} 

\begin{figure}
        \centering
        \includegraphics[scale=0.4]{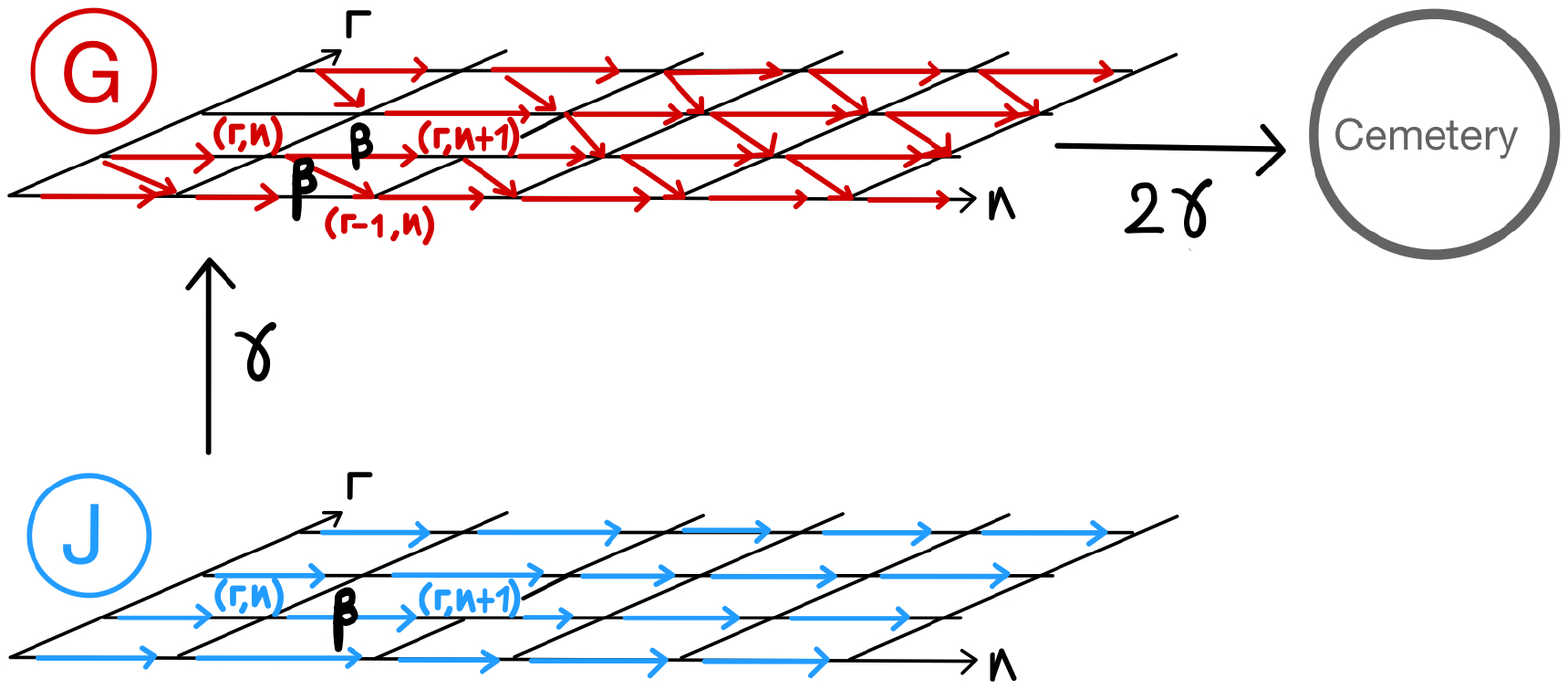}
        \caption{Dual dynamics for the SIR model}\label{DualSIR}
    \end{figure}

\subsection{Applications of the duality relation}\label{subsec:appldual}
Consider a measure $\nu$ on $\{S,I,R\}^{\Z}$ and 
recall the $n$-point cluster functions 
$G^{\nu}(r,n,t)$ and $J^{\nu}(r,n,t)$ defined respectively in \eqref{def:G}
and \eqref{def:J}. Note that $J^{\nu}(r,n,t)$ is not the same as  
$H(n)$ in \cite{Schutzetal2008}, because of the term $ \mathds{1}_{\e_{r-1}(t)=R}$ 
but we still obtain closed equations for $G$ and $J$ type chains.

\begin{rem}[SIR correlations for small clusters]\label{clusterlenght}
Choosing $n = 1$ the duality functions can give information on the three point 
correlation functions given in terms of the clusters $G$ or $J$, 
which can then be used 
to get two point correlation functions via the cluster $H$.
\end{rem}

\subsubsection{Coupled equations for the cluster functions}\label{subsub:coupledforcluster}

In \cite{Schutzetal2008} (see equation (20)), two first order differential equations 
are provided and explicitly solved for $G^{\nu}(r,n,t)$ and $H^{\nu}(r,n,t)$, 
provided that the initial measure $\nu$ is translation invariant. Therefore, 
the initial measure and the dynamics (we are on $\Z$), being both translation invariant, 
one can drop the dependence on $r$ and write $G^{\nu}(r,n,t)=G^{\nu}(n,t) $ 
and $H^{\nu}(r,n,t)=H^{\nu}(n,t) $. 
We first show that one can recover these equations for $G^{\nu}(n,t)$ 
and $J^{\nu}(n,t)$, by only using the duality relation of Theorem \ref{SIRduality}.

\begin{prop}\label{prop:cfeq20}
Assume that the initial distribution $\nu$ is translation invariant.
Fix $n\in \mathbb{N}$, then we have
\begin{align}\label{analogous}
       \begin{cases}
       \dfrac{ d G^{\nu}(n,t)}{ dt} = - 2(\gamma + \beta)G^{\nu}(n,t) + 2\beta G^{\nu}(n+1,t)\\
        \dfrac{ d J^{\nu}(n,t)}{ dt} = - (\gamma + \beta) J^{\nu}(n,t)+\beta J^{\nu}(n+1,t) + \gamma  G^{\nu}(n,t) \;. 
       \end{cases}
\end{align}
The solutions are given by
\begin{eqnarray}\label{solG}
    G^{\nu}(n,t) &=& e^{-2(\gamma+\beta)t}\sum_{\ell\geq 0}\frac{(2\beta t)^{\ell}}{\ell !}G^{\nu}(\ell +n,0)\\ \nonumber
  J^{\nu}(n,t) &=& e^{-(\gamma+\beta)t}\sum_{\ell \geq 0}
  \frac{(\beta t)^{\ell}}{\ell !}J^{\nu}(\ell + n,0)\\
  && \quad+ \gamma \int_0^t e^{-(\gamma+\beta)(t-s)}\sum_{\ell \geq 0}\frac{(\beta (t-s))^{\ell}}{\ell !}G^{\nu}(\ell + n,s)ds.  \label{solJ}
\end{eqnarray}
Moreover
\begin{align} \label{newcluster}
    J^{\nu}(n,t) &=H^{\nu}(n,t) - H^{\nu}(n+1,t) -G^{\nu}(n,t) \;.
\end{align}
\end{prop}

\subsubsection{Non-translation invariant case}\label{subsub:nontranslinv}

In \cite{Schutzetal2008}, the cluster correlation functions 
have not been solved in the case 
where $\nu$ is not translation invariant, because the equations 
obtained in $G$ and $J$ or $H$ 
are not closed. Our goal now is to get an explicit expression 
of such cluster functions 
in the non translation invariant case, by using the dual process.

Denote by $\xi$ an element of $\mcb{S}^{\rm dual}$. For any SIR configuration $\e$, 
by abuse of notation, we write
\[G^{\eta}(r,n,t)=:G^{\delta_\eta}(r,n,t)
~ \text{and}~ J^{\eta}(r,n,t)=:J^{\delta_\eta}(r,n,t),  \]
the cluster correlation functions when the system is initially in state $\e$. 
Note that the measure 
$\delta_{\e}$ is no longer a translation invariant measure for the SIR process 
(unless it is 
a constant configuration, in which case the system remains unchanged).
\begin{thm}\label{nontranslation1}
The $G^{\eta}(r,n,t)$ cluster is given by:
\begin{equation} \label{eqG}
    G^{\eta}(r,n,t)=e^{-2(\gamma+\beta)t}\sum_{\underset{a\leq b}{(a,b)\in (\mathbb{N})^2}}
    \frac{\binom{b}{a}}{2^b}\frac{(2\beta t)^b}{b!}G^{\eta}(r-a,n+b,0).
\end{equation}
\end{thm}

\begin{thm}\label{nontranslation2}
The  $J^{\eta}(r,n,t)$ cluster is given by:
\begin{equation} \label{eqJ}
\begin{split}
        &J^{\eta}(r,n,t) 
        = e^{-(\gamma+\beta) t}\sum_{k\geq 0}\frac{(\beta t)^k}{k!} 
        J^{\eta(r,n+k,0)}\\
        & + \gamma e^{-(\gamma + \beta) t}\int_0^t e^{-(\gamma+ \beta)(t-s)}
        \sum_{k\geq 0}\sum_{\underset{(a,b)\in \mathbb{N}^2}{a\leq b}}
        \frac{(\beta s)^k}{k!} 
        \frac{\binom{b}{a}}{2^b}\frac{(2\beta (t-s))^b}{b!} G^{\nu}(r-a,n+k+b,0)ds.
\end{split}
\end{equation}
\end{thm}

Theorems \ref{nontranslation1} and \ref{nontranslation2} hold 
for any initial measure $\nu$ 
which is not necessarily translation invariant. 
We conclude this section showing that under the extra assumption 
of translation invariance 
we can recover the expressions (21) and (22) in \cite{Schutzetal2008}.

\subsubsection*{Special case: translation invariance}
Let us recover the formulas obtained in \cite{Schutzetal2008} 
in the translation invariant case using ODE's (here we have not used such ODE's), 
from \eqref{eqG} and \eqref{eqJ}. If $\nu$ is translation invariant, 
$G^{\nu}(r,n,0)$ and $J^{\nu}(r,n,0)$ are independent of $r$, so
\begin{equation}\label{Ginv}
    \begin{split}
        G^{\nu}(n,t)&:=G^{\nu}(r,n,t)=e^{-2(\gamma+\beta)t}
        \sum_{\underset{a\leq b}{(a,b)\in \mathbb{N}^2}}\frac{\binom{b}{a}}{2^b}\frac{(2\beta t)^b}{b!}G^{\nu}(r-a,n+b,0)\\
        &= e^{-2(\gamma+\beta)t}\sum_{b\geq 0}\sum_{a=0}^b\frac{\binom{b}{a}}{2^b}\frac{(2\beta t)^b}{b!}G^{\nu}(n+b,0).\\
        \text{Therefore,~ }G^{\nu}(n,t)&= e^{-2(\gamma+\beta)t}\sum_{b\geq 0}\frac{(2\beta t)^b}{b!}G^{\nu}(n+b,0),
    \end{split}
\end{equation}
which matches (21) in \cite{Schutzetal2008}. Now,
\begin{equation}
    \begin{split}
        J^{\nu}(n,t)&:=J^{\nu}(r,n,t)= e^{-(\gamma+\beta) t}\sum_{k\geq 0}\frac{(\beta t)^k}{k!} J^{\nu}(n+k,0)\\
        & + \gamma e^{-(\gamma + \beta) t}\int_0^t e^{-(\gamma+ \beta)(t-s)}\sum_{k\geq 0}
        \sum_{\underset{(a,b)\in \mathbb{N}^2}{a\leq b}}\frac{(\beta s)^k}{k!} \frac{\binom{b}{a}}{2^b}\frac{(2\beta (t-s))^b}{b!} G^{\nu}(n+k+b,0)ds\\
        &= e^{-(\gamma+\beta) t}\sum_{k\geq 0}\frac{(\beta t)^k}{k!} J^{\nu}(n+k,0)\\
        & + \gamma e^{-(\gamma + \beta) t}\int_0^t e^{-(\gamma+ \beta)(t-s)}\sum_{k\geq 0}\sum_{b\geq 0}\frac{(\beta s)^k}{k!} \frac{(2\beta (t-s))^b}{b!} G^{\nu}(n+k+b,0)ds.
    \end{split}
\end{equation}
Performing the change of variable $u=t-s$, the second integral becomes
\begin{equation}
\begin{split}
        & \int_0^t\sum_{k\geq 0} \frac{(\beta(t-u))^k}{k!} 
        e^{-(\gamma + \beta)u} \sum_{b\geq 0}\frac{(2\beta u)^b}{b!}
        G^{\nu}(n+k+b,0)ds\\
        &= \int_0^t\sum_{k\geq 0} \frac{(\beta(t-u))^k}{k!} G^{\nu}(n+k,s)ds,
\end{split}
\end{equation}
where we used the expression for $G^{\nu}(n,t)$ given by \eqref{Ginv}. Finally, we get
\begin{equation} 
  J^{\nu}(n,t) = e^{-(\gamma+\beta)t}\sum_{\ell \geq 0}
  \frac{(\beta t)^{\ell}}{\ell !}J^{\nu}(\ell + n,0) 
  + \gamma \int_0^t e^{-(\gamma+\beta)(t-s)}
  \sum_{\ell \geq 0}\frac{(\beta (t-s))^{\ell}}{\ell !}G^{\nu}(\ell + n,s)ds,
\end{equation}
which is exactly the solution given in \eqref{solJ}, 
found by solving the ODE for $J$, in the same spirit as in \cite{Schutzetal2008}.

\subsection{Proofs for Section \ref{sec:SIR}}\label{subsec:proofs-sir}
\subsubsection{Proofs for Subsection \ref{sectioncluster}}\label{subsec:pfs-cluster}

\begin{proof}[Proof of Theorem \ref{SIRduality}]
First, note that due to the definitions and conventions \eqref{def:Sdual},
\eqref{eq:dtrap}, \eqref{eq:ftrap}, \eqref{sirdualgen}, 
we have, for any configuration $\e$, 
\begin{equation}\label{rel:dualSIR-trap}
\mathcal{L}^{{\rm SIR}}d(\cdot, \partial)(\e) 
= \mathcal{L}^{\rm dual}d(\e,\cdot )(\partial).
 \end{equation}
Then,  in order to show the following duality relation for all configuration 
$\e$ and $(r,n,i)\in\Z\times\N \times \{G,J\}$
\begin{equation}\label{rel:dualSIR-main}\mathcal{L}^{{\rm SIR}}d(\cdot, (r,n,i))(\e) = \mathcal{L}^{\rm dual}d(\e,\cdot )(r,n,i), \end{equation}
it is convenient to split the SIR generator into three different actions, namely
\begin{equation}\label{3terms}
\mathcal{L}^{{\rm SIR}}d(\cdot,(r,n,i))(\e) = \sum_{y\in \Z} \big[A_y + B_y + C_y\big] d(\cdot,(r,n,i))(\e),
\end{equation}
where 
\begin{equation}\label{def:ABC}
    \begin{split}
        A_y \ d(\cdot,(r,n,i))(\e) &:= \beta\e^{I}_y\e^{S}_{y+1}\Big[d\Big(
        T_{y+1}^I\e, (r,n,i)\Big) - d\left(\e,(r,n,i)\right)  \Big]\\
       B_y \ d(\cdot,(r,n,i))(\e) &:= \beta\e^{I}_y\e^{S}_{y-1}\Big[d\Big(
        T_{y-1}^I\e, (r,n,i)\Big) - d\left(\e,(r,n,i)\right)  \Big]\\
       C_y \ d(\cdot,(r,n,i))(\e) & := \gamma\e^{I}_y\Big[d\Big(
     T_y^R \e, (r,n,i)\Big) - d\left(\e,(r,n,i)\right)  \Big],
    \end{split}
\end{equation}
for the operators $T^a_y$ defined in \eqref{def:Ty}. 
To conclude the proof we proceed by direct computations of each 
of the three items above:

\begin{lem}\label{abc lemma}[Actions of $A_y$, $B_y$, $C_y$]
    The actions of the three operators $A_y$, $B_y$ and $C_y$ 
    whose sum over $y \in \mathbb{Z}$ 
    defines the SIR generator \eqref{sir gen} are 
 \begin{eqnarray}\label{sumAy}
    \sum_{y\in\Z}A_y \ d(\e,(r,n,i)) 
    &=& \beta \mathds{1}_{i=G}\Big[d(\e,(r-1,n+1,i))-d(\e,(r,n,i)) \Big]\\
 \label{sumBy}
    \sum_{y\in\Z}B_y \ d(\e,(r,n,i)) 
    &=& \beta\big[ d(\e,(r,n+1,i))-d(\e,(r,n,i))\big]\\
\label{sumCy}
    \sum_{y\in\Z}C_y \ d(\e,(r,n,i)) &=& \gamma \mathds{1}_{i=J}\big[d(\e,(r,n,\phi(i)))-d(\e,(n,i)) ]-2\gamma \mathds{1}_{i=G}d(\e,(r,n,i)) \;.
\end{eqnarray}
\end{lem} 

Collecting \eqref{sumAy}, \eqref{sumBy}, \eqref{sumCy} and using that $d(\e,\partial) = 0$ we get
\begin{equation}
    \begin{split}
        \mathcal{L}^{{\rm SIR}}d(\cdot,(r,n,i))(\e)&= \beta \mathds{1}_{i=G}\Big[ d(\e,(r-1,n+1,i)) - d(\e,(r,n,i))\Big]\\
        & + \beta\Big[ d(\e,(r,n+1,i)) - d(\e,(r,n,i))\Big]\\
        &+ \gamma \mathds{1}_{i=J}\Big[d(\e,(r,n,\phi(i))) - d(\e,(r,n,i)) \Big]\\
        &+2\gamma \mathds{1}_{i=G}\Big[d(\e,\partial)  -d(\e,(r,n,i))\Big]=\mathcal{L}^{\rm dual}d(\e,\cdot)(r,n,i),
    \end{split}
\end{equation}
namely, we have recovered the generator $\mathcal{L}^{\rm dual}$ of equation \eqref{sirdualgen}.
\end{proof}

\begin{proof}[Proof of Lemma \ref{abc lemma}]
    We proceed by direct computations. Recall the definition of the duality function in \eqref{def:dualfct}.
    \begin{enumerate}[label=(\roman*)]
    \item \underline{Action of $A_y$} \quad
    If $y+1 \notin \lbrace r-1, \ldots, r+n \rbrace $ there is no action. Moreover, \newline
\begin{eqnarray*}
\eta_y^I\eta_{y+1}^S d(T^I_{y+1} \eta, (r,n,i))&=&
       \begin{cases}
     0 & \text{ if } y+1 \in \lbrace r, \ldots, r+n \rbrace\\
     d(\eta, (r-1,n+1,G))  & \text{ if } y+1 = r-1
       \end{cases} \\ \\
\eta_y^I\eta_{y+1}^S d( \eta, (r,n,i))&=&
       \begin{cases}
     0 & \text{ if } y \in \lbrace r, \ldots, r+n \rbrace\\
     d(\eta, (r,n,G))  & \text{ if } y = r-1
       \end{cases}
\end{eqnarray*}
and so 
\begin{equation*}
     \sum_{y\in\Z}A_y \ d(\e,(r,n,i)) = \Big( A_{r-2} + A_{r-1} \Big) d(\e,(r,n,i))
      = \beta \Big[ d(\e,(r-1,n+1,G)) - d(\e,(r,n,G)) \Big]
\end{equation*}
     \item \underline{Action of $B_y$} \quad
    If $y-1 \notin \lbrace r-1, \ldots, r+n \rbrace $ 
    there is no action. Moreover, \newline
\begin{eqnarray*}
\eta_{y-1}^S\eta_{y}^I d(T^I_{y-1} \eta, (r,n,i))&=&
       \begin{cases}
     0 & \text{ if } y-1 \in \lbrace r-1, \ldots, r+n-1 \rbrace\\
     d(\eta, (r,n+1,i))  & \text{ if } y-1 = r+n
       \end{cases}\\  \\
\eta_{y-1}^S\eta_{y}^I  d( \eta, (r,n,i))&=&
       \begin{cases}
     0 & \text{ if } y \in \lbrace r, \ldots, r+n-1 \rbrace\\
     d(\eta, (r,n,i))  & \text{ if } y = r+n
       \end{cases}
\end{eqnarray*}
and so 
\begin{equation*}
     \sum_{y\in\Z}B_y \ d(\e,(r,n,i)) = \Big( B_{r+n+1} + B_{r+n} \Big) d(\e,(r,n,i)) 
     = \beta \Big[ d(\e,(r,n+1,i)) - d(\e,(r,n,i)) \Big]
\end{equation*} 
          \item \underline{Action of $C_y$} \quad
    If $y \notin \lbrace r-1, \ldots, r+n \rbrace $ there is no action. Moreover, \newline
\begin{eqnarray*}
\eta_{y}^I d(T^R_{y} \eta, (r,n,i))&=&
       \begin{cases}
     0 & \text{ if } y \in \lbrace r, \ldots, r+n \rbrace\\
     d(\eta, (r,n,G)) \mathds{1}_{i=J}  & \text{ if } y = r-1
       \end{cases}\\ \\
\eta_{y}^I  d( \eta, (r,n,i))&=&
       \begin{cases}
     0 & \text{ if } y \in \lbrace r, \ldots, r+n-1 \rbrace\\
    d(\eta, (r,n,G)) & \text{ if } y = r-1
     \\
     d(\eta, (r,n,i))  & \text{ if } y = r+n
       \end{cases}
\end{eqnarray*}
and so 
\begin{eqnarray*}
     \sum_{y\in\Z}C_y \ d(\e,(r,n,i)) 
     & = &\Big( C_{r-1} + C_{r+n} \Big) d(\e,(r,n,i)) \\ 
     & =&\gamma \Big[ d(\e,(r,n,G)) \mathds{1}_{i=J} - d(\e,(r,n,G)) -   d(\e,(r,n,i))\Big]
\end{eqnarray*}  
    \end{enumerate}
\end{proof}

\subsubsection{Proofs for Subsection \ref{subsec:appldual}}\label{subsec:pfs-appldual}

\begin{proof}[Proof of Proposition \ref{prop:cfeq20}]
We have
\begin{align*}
    \frac{d}{dt}G^{\nu}(n,t) &=\frac{d}{dt}\mathbb{E}_{\nu}\Big[d(.,(r,n,G))(\e(t)) \Big]  
    =\mathbb{E}_{\nu}\Big[ \mathcal{L}^{{\rm SIR}}d(.,(r,n,G))(\e(t)) \Big]\\
    &= \mathbb{E}_{\nu}\Big[ \mathcal{L}^{\rm Dual}d(\e(t),.)(r,n,G)\Big]\\
    &= \beta \Big[\mathbb{E}_{\nu}\Big[d(.,(r-1,n+1,G))(\e(t)) \Big] - \mathbb{E}_{\nu}\Big[d(.,(r,n,G))(\e(t)) \Big]   \Big]\\
    & + \beta \Big[\mathbb{E}_{\nu}\Big[d(.,(r,n+1,G))(\e(t)) \Big] - \mathbb{E}_{\nu}\Big[d(.,(r,n,G))(\e(t)) \Big]   \Big]\\
    &- 2\gamma \mathbb{E}_{\nu}\Big[d(.,(r,n,G))(\e(t)) \Big] \\
    &= 2\beta\big[ G^{\nu}(n+1,t) - G^{\nu}(n,t)] - 2\gamma G^{\nu}(n,t),
\end{align*}
where we used the duality relation in the first line and the fact that $G^{\nu}$ 
does not depend on $r$, by translation invariance, in the last line. In the same way, we get:
\begin{align*}
    \frac{d}{dt}J^{\nu}(n,t) &=\frac{d}{dt}\mathbb{E}_{\nu}\Big[d(.,(r,n,H))(\e(t)) \Big]  
    =\mathbb{E}_{\nu}\Big[ \mathcal{L}^{{\rm SIR}}d(.,(r,n,H))(\e(t)) \Big] \\
    &= \mathbb{E}_{\nu}\Big[ \mathcal{L}^{\rm Dual}d(\e(t),.)(r,n,H)\Big]\\
    &= \beta \Big[\mathbb{E}_{\nu}\Big[d(.,(r,n+1,H))(\e(t)) \Big] - \mathbb{E}_{\nu}\Big[d(.,(r,n,H))(\e(t)) \Big]   \Big] \\
    &+ \gamma \Big[\mathbb{E}_{\nu}\Big[d(.,(r,n,G))(\e(t)) \Big] - \mathbb{E}_{\nu}\Big[d(.,(r,n,H))(\e(t)) \Big]   \Big]\\
    &= \beta\big[ J^{\nu}(n+1,t) - J^{\nu}(n,t)] +\gamma \big[ G^{\nu}(n,t) - J^{\nu}(n,t) \big] \;.
\end{align*}
We have thus derived \eqref{analogous}.
Note that indeed these are the analogous of equation (20) in \cite{Schutzetal2008}.
The first equation is exactly the same one as for $G$ in \cite{Schutzetal2008} 
and the solution is given by \eqref{solG}.
The second equation can be solved similarly in a recursive way,  by treating the terms $J^{\nu}(n+1,t) $ and $ G^{\nu}(n,t)$ as inhomogeneities.
The solution is given by \eqref{solJ}.
Then, using that $\e_{r-1}^S+\e_{r-1}^I+ \e_{r-1}^R = 1$, and noticing that
\begin{equation*}
    H^{\nu}(n,t) = H^{\nu}(n+1,t) + J^{\nu}(n,t) + G^{\nu}(n,t),
\end{equation*}
the equation for $H^{\nu}$, which is the same as the one in \cite{Schutzetal2008}, is
\begin{equation*}
    \frac{dH^{\nu}(n,t)}{dt} =-(\gamma + \beta)H^{\nu}(n,t) + \beta\big(H^{\nu}(n+1,t) - G^{\nu}(n,t)\big),
\end{equation*}
with solution 
\begin{equation*}
   H^{\nu}(n,t) = e^{-(\gamma+\beta)t}\sum_{\ell \geq 0}
   \frac{(\beta t)^{\ell}}{\ell !}H(\ell + n,0) - \beta\int_0^t e^{-(\gamma+\beta)(t-s)}
   \sum_{\ell \geq 0}\frac{(\beta (t-s))^{\ell}}{\ell !}G(\ell + n,s)ds \;.
\end{equation*}
Therefore, we can write the solution  $J^{\nu}$ in terms of the $H^{\nu}$ cluster as
\eqref{newcluster}.
\end{proof}

For the proofs of Theorems \ref{nontranslation1} and \ref{nontranslation2}, 
we introduce two auxiliary dynamics. The  first dynamics is
defined on $\Z\times \mathbb{N}\times \{G\}$  by the following transition rates
\begin{equation}\label{DualG}
    (r,n,G) \rightarrow (r-1,n+1,G)
    :~ \text{at rate}~ \beta,~ ~ \text{and}~~(r,n,G) \rightarrow (r,n+1,G)
    :~ \text{at rate}~ \beta.
\end{equation}
Denote by $\mathbb{P}_{(r,n,J)}^{G}$, resp. $\mathbb{E}_{(r,n,J)}^{\rm Dual,G}$ 
the probability measure, resp. expectation, under this dynamics, 
when starting from $(r,n,G)$.

The second dynamics is defined on $\Z\times \mathbb{N} \times \{J\}$
 by the following transition rates
\begin{equation}\label{DualJ}
    (r,n,J) \rightarrow (r,n+1):~ \text{at rate}~ \beta. 
\end{equation}
Denote by $ \mathbb{P}_{(r,n,J)}^{J}$, resp. $\mathbb{E}_{(r,n,J)}^{J}$ 
the probability measure, resp. expectation, under this dynamics when 
starting from $(r,n,J)$.

\begin{proof}[Proof of Theorem \ref{nontranslation1}]
Fix $(r,n)\in \Z^2$, thanks to duality we can write 
\begin{equation} \label{pippo}
   G^{\eta}(r,n,t) =  e^{-2\gamma t} \mathbb{E}_{(r,n,G)}^{G}\big[ d(\e,(r(t),n(t),G)) \big]   \;.
\end{equation}
Indeed, for the left hand side we have
\begin{align*}
 \mathbb{E}_{\e}\big[d(\e(t),(r,n,G) )\big]
 & =  \mathbb{E}_{(r,n,G)}^{\rm Dual}\big[ d(\e,\xi(t))\big] \\
    &= \mathbb{E}_{(r,n,G)}^{\rm Dual}\big[ d(\e,\xi(t)) \cap \big(\xi(t)\neq \partial \big)  \big]\\
    &= \mathbb{E}_{(r,n,G)}^{\rm Dual}\big[ d(\e,\xi(t)) \cap \big(\forall 0\leq s\leq t,~ \xi_s\neq \partial \big)  \big] \\
    &= \mathbb{E}_{(r,n,G)}^{\rm Dual}\big[ \mathds{1}_{\{\forall 0\leq s\leq t,~ \xi_s\neq \partial \}} d(\e,\xi(t))\big]\\
    &= \mathbb{P}_{(r,n,G)}\big[\forall 0\leq s\leq t,~ \xi_s\neq \partial \big] 
    \mathbb{E}_{(r,n,G)}^{\rm Dual}\big[  d(\e,\xi(t))|\forall 0\leq s\leq t,~ \xi_s\neq \partial\big] \\
    &= \mathbb{P}_{(r,n,G)}\big[\forall 0\leq s\leq t,~ \xi_s\neq \partial \big] \times \mathbb{E}_{(r,n,G)}^{G}\big[ d(\e,(r(t),n(t),G)) \big] \;.
\end{align*}
Above we used the strong Markov property on the fifth line and the sixth one 
is due to the fact 
that when conditioning on not having reached the trap $\partial$ up to time $t$, 
the dual process 
starting from $(r,n,G)$ has the same dynamics as the one defined in \eqref{DualG}.
Finally, to get the identity in \eqref{pippo}
we used the fact that the time of reaching the trap $\partial$ 
from any $(r,n,\partial)$, 
is given by an exponential clock of parameter $2\gamma$. Expanding the expectation 
in the right hand side of \eqref{pippo}, we are left with
\begin{align*}
     G^{\eta}(r,n,t) 
     &= e^{-2\gamma t} \sum_{(\widetilde{r}, \widetilde{n})\in \Z\times \mathbb{N}}
     d\big(\e,(\widetilde{r}, \widetilde{n})\big)\mathbb{P}_{(r,n,G)}^{G}\big[(r(t),n(t),G) =(\widetilde{r}, \widetilde{n},G) \big]\\
     &=e^{-2\gamma t }\sum_{\underset{a\leq b}{(a,b)\in \mathbb{N}^2}}
     d\big(\e,(r-a,n+b)\big)\mathbb{P}_{(r,n,G)}^{G}\big[(r(t),n(t),G) =(r-a, n+b,G) \big]\\
     &=e^{-2\gamma t }\sum_{\underset{a\leq b}{(a,b)\in \mathbb{N}^2}}
     d\big(\e,(r-a,n+b)\big)e^{-2\beta t }\frac{\binom{b}{a}}{2^b}\frac{(2\beta t)^b}{b!}\\
     &= e^{-2(\gamma+\beta)t}\sum_{\underset{a\leq b}{(a,b)\in \mathbb{N}^2}}\frac{\binom{b}{a}}{2^b}\frac{(2\beta t)^b}{b!}G^{\eta}(r-a,n+b,0).
\end{align*}
The third line comes from the fact that in order to reach $(r-a,n+b)$ starting from $(r,n)$ 
and with the dynamics \eqref{DualG}, one has to perform $\binom{b}{a}$ steps, 
where $r$ decreases by $1$ and $n$ increases by $1$ and, the rest of the steps 
where it is only $n$ that increases by $1$. The time between two such jumps 
is an exponential clock with parameter $2\beta$.
\end{proof}
The other cluster can be found in a similar fashion.

\begin{proof}[Proof of Theorem \ref{nontranslation2}]
Similarly to the previous computation we have
\begin{equation}\label{splitexp}
\begin{split}
    &J^{\eta}(r,n,t) = \mathbb{E}_{\eta}\big[d(\e_t,(r,n,J)) \big] 
    = \mathbb{E}^{\rm Dual}_{(r,n,J)}\big[d(\e,\xi(t)) \big]
    =\mathbb{E}^{\rm Dual}_{(r,n,J)}\big[d(\e,\xi(t))\cap(\xi(t)\neq \partial) \big]\\
    &= \mathbb{E}^{\rm Dual}_{(r,n,J)}\big[d(\e,(r(t),n(t),i(t)))\cap(\forall 0 \leq s \leq t,~ \xi(s)\neq \partial) \big]\\
    &=\mathbb{E}^{\rm Dual}_{(r,n,J)}\big[d(\e,(r(t),n(t),G))\cap(\forall 0 \leq s \leq t,~ \xi(s)\neq \partial)\big]\\
    &+\mathbb{E}^{\rm Dual}_{(r,n,J)}\big[d(\e,(r(t),n(t),J))\cap(\forall 0 \leq s \leq t,~ \xi(s)\neq \partial)\big],
\end{split}
\end{equation}
where in the last line, we partitioned according to having or not a jump from $J$ to $G$ before time $t$. Defining:
\begin{equation*}
    \mathcal{A}_s^t := \big(\forall 0 \leq u \leq s^{-},~ \xi_3(u)
   = J\big)\cap (\forall s^+ \leq u \leq t,~ \xi_3(u)= G\big)\cap\big(\forall 0 \leq u \leq t,~ \xi(u)\neq \partial\big),
\end{equation*}
the event that before time $t$, the walk does not reach the trap $\partial$ 
and at time $0\leq s\leq t$, the walk jumps from lane $J$ to $G$, we have:
\begin{equation}\label{integralAst}
\begin{split}
    &\mathbb{E}^{\rm Dual}_{(r,n,J)}\big[d(\e,(r(t),n(t),G))\cap(\forall 0 \leq s \leq t,~ \xi(s)\neq \partial)\big]
    = \int_0^t \mathbb{E}^{\rm Dual}_{(r,n,J)}\big[\mathcal{A}_s^t \big]ds\\
&= \int_0^t d\mathbb{P}^{\rm Dual}_{(r,n,J)}\big[\mathcal{A}_s^t \big] \mathbb{E}^{\rm Dual}_{(r,n,J)}\Big[d(\e,(r(t),n(t),i(t))\Big| \mathcal{A}_s^t\Big].
\end{split}
\end{equation}
 Using that the law of the jump from state $J$ to $G$ is an exponential law of parameter $\gamma$, 
 and that the rate of  jump from state $G$ to the trap $\partial$ is $2\gamma$, we have
\begin{align*}
    d\mathbb{P}^{\rm Dual}_{(r,n,J)}\big[\mathcal{A}_s^t \big]
    &= \gamma e^{-\gamma s} \big(1-\int_0^{t-s}2\gamma e^{-2\gamma u}du\big)ds 
    = \gamma e^{-\gamma s}e^{-2\gamma (t-s)}ds = \gamma e^{-2\gamma t}e^{\gamma s}ds.
\end{align*}
Now, we expand \eqref{integralAst} by partitioning on the number of increases in $n$ 
up to time $s$ for the dual process conditioned on being of type $J$ between times $0$ and $s$, 
for a certain $s\in [0,t]$:
\begin{align*}
    &\mathbb{E}^{\rm Dual}_{(r,n,J)}\Big[d(\e,(r(t),n(t),i(t))\Big|\mathcal{A}_s^t \Big]\\
    &= \sum_{k\geq 0}\mathbb{E}^{\rm Dual}_{(r,n,J)}\Big[d(\e,(r(t),n(t),i(t))\cap (r(s),n(s),i(s)) = (r,n+k,J) \Big| \mathcal{A}_s^t \Big] \\
    &=\sum_{k\geq 0} \mathbb{P}_{(r,n,J)}^{J}\big[(r(s),n(s),J) =(r,n+k,J)\big] \mathbb{E}_{(r,n+k,G)}^{G}\big[d(\e,(r(t-s),n(t-s),G)]\\
    &= \sum_{k\geq 0} e^{-\beta s}\frac{(\beta s)^k}{k!}\sum_{\underset{(a,b)\in \mathbb{N}^2}{a\leq b}}d(\e,(r-a,n+k+b,G)) \\
    & \cdot \mathbb{P}^G_{(r,n+k,G)}\big[(r(t-s),n(t-s),G) = (r-a,r+k+b,G) ]\\
    &= \sum_{k\geq 0}\sum_{\underset{(a,b)\in \mathbb{N}^2}{a\leq b}}e^{-\beta s}
    \frac{(\beta s)^k}{k!} e^{-2\beta(t-s)} \frac{\binom{b}{a}}{2^b}\frac{(2\beta (t-s))^b}{b!} G^{\eta}(r-a,n+k+b,0),
\end{align*}
where we used the strong Markov property in the second line, and expanded the expectation 
under the process \eqref{DualG} in the third line. Therefore,
\begin{equation}\label{Fst}
\begin{split}
     &\mathbb{E}^{\rm Dual}_{(r,n,J)}\big[d(\e,(r(t),n(t),G))\cap(\forall 0 \leq s \leq t,~ \xi(s)\neq \partial)\big]=\\
     &\gamma e^{-2\gamma t}\int_0^t e^{\gamma s}\sum_{k\geq 0}\sum_{\underset{(a,b)\in \mathbb{N}^2}{a\leq b}}e^{-\beta s}
     \frac{(\beta s)^k}{k!} e^{-2\beta(t-s)} \frac{\binom{b}{a}}{2^b}\frac{(2\beta (t-s))^b}{b!} G^{\eta}(r-a,n+k+b,0)ds.
\end{split}
\end{equation}
Finally, the second expectation in the last line in \eqref{splitexp} is given by:
\begin{equation}\label{Kt}
       \mathbb{E}^{\rm Dual}_{(r,n,J)}\big[d(\e,(r(t),n(t),J))\cap(\forall 0 \leq s \leq t,~ \xi(s)\neq \partial)\big] 
       = e^{-\gamma t}\sum_{k\geq 0}\frac{ e^{-\beta t}(\beta t)^k}{k!} J^{\eta}(r,n+k,0).
\end{equation}
Collecting \eqref{Fst} and \eqref{Kt} yields:
\begin{equation*}
\begin{split}
        &J^{\eta}(r,n,t) = e^{-(\gamma+\beta) t}\sum_{k\geq 0}\frac{(\beta t)^k}{k!} J^{\eta}(r,n+k,0)\\
        &+ \gamma e^{-2\gamma t}\int_0^t e^{\gamma s}\sum_{k\geq 0}\sum_{\underset{(a,b)\in \Z^2}{a\leq b}}
        e^{-\beta s}\frac{(\beta s)^k}{k!} e^{-2\beta(t-s)} \frac{\binom{b}{a}}{2^b}\frac{(2\beta (t-s))^b}{b!} G^{\eta}(r-a,n+k+b,0)ds\\
        &= e^{-(\gamma+\beta) t}\sum_{k\geq 0}\frac{(\beta t)^k}{k!} J^{\eta}(r,n+k,0)\\
        & + \gamma e^{-(\gamma + \beta) t}\int_0^t e^{-(\gamma+ \beta)(t-s)}
        \sum_{k\geq 0}\sum_{\underset{(a,b)\in \Z^2}{a\leq b}}\frac{(\beta s)^k}{k!} 
        \frac{\binom{b}{a}}{2^b}\frac{(2\beta (t-s))^b}{b!} G^{\eta}(r-a,n+k+b,0)ds.
\end{split}
\end{equation*}
\end{proof}

\appendix

\section{The diffusive contact process on small finite sets}\label{DCP-small}
\subsection{Invariant measure of the DCP on small finite sets}\label{Konno}
Here, in the same spirit as \cite{Konno1994}, we compute the stationary 
distribution of the diffusive contact process for small $N$, where we 
recall that the dynamics is given by the generator $\mathcal{L}^{\rm DCP}$, 
defined in \eqref{generator_opendiffusivecontact}, here we stress the dependence 
on the lattice size with the subscript $N$. This approach is based on the definition 
of the invariant measure as the probability distribution $\nu^{\rm DCP}$ satisfying 
$\nu^{\rm DCP}\mathcal{L}^{\rm DCP}=0$ 
cf. \cite[Chapter 1, Proposition 1.8]{IPS}). 
Here, configurations are represented in their vectorial form, 
for instance, for $N=1$, $(0)$ 
is the empty configuration and $(1)$ the full one.

We notice that in our case the presence of boundary reservoirs allows
 to see the effect of the diffusive coefficient $\mathcal{D}$ already 
 when the bulk size is $N=2$. This is in contrast to the scenario 
 studied in Section 6.2 of \cite{Konno1994} with boundary conditions 
 $\eta_0 = \eta_{N+1}=1$ when the results are equivalent to the case 
 of $\mathcal{D}=0$. 
\subsubsection*{$N=1$ site}
Let us solve the stationary condition
$ \nu^{\rm DCP}\mathcal{L}^{\rm DCP} = 0 $. 
For $N=1$, the corresponding intensity matrix is
\begin{equation*}
L^{\rm DCP} = \begin{pmatrix}
-(\alpha + \delta) & (\alpha + \delta)  \\
(\gamma + \beta +1) & -(\gamma + \beta +1),  
\end{pmatrix} 
\end{equation*}
and the stationary condition yields
\begin{equation*}
    \begin{cases}
\nu^{\rm DCP} ((0)) &= \dfrac{\gamma + \beta + 1}
{1 + \beta + \gamma + \delta + \alpha} \\
\nu^{\rm DCP} ((1)) &= \dfrac{\alpha + \delta}{1 + \beta + \gamma + \delta + \alpha}.
\end{cases}
\end{equation*}
Notice that for just one bulk site there is no diffusive effect.
\subsubsection*{$N=2$ sites} 
This is the first non trivial case where the diffusion parameter 
$\mathcal{D}$ plays a role. 
In this case the intensity matrix is 
\begin{equation*}
  L^{\rm DCP} = \begin{pmatrix}
-(\alpha + \delta) &  \delta & \alpha & 0  \\
1 + \beta  & -(1 + \beta + \mathcal{D} + \lambda + \alpha) 
& \mathcal{D} & \lambda + \alpha \\
1 + \gamma & \mathcal{D} & -(1 + \gamma + \mathcal{D} + \lambda + \delta) 
& \lambda + \delta \\
0 & 1 + \gamma & 1 + \beta & -(2 + \beta + \gamma) 
\end{pmatrix} 
\end{equation*}
and the stationary condition yields

\begin{equation}\label{sol-like-konno}
    \begin{cases}
c(\mathcal{D})\nu^{\rm DCP}((0,0))&= \mathcal{D} (\beta + \gamma + 2)^2 
+ (\alpha + \beta + \gamma+ \delta + 2+ 2 \lambda)(1+\beta + \gamma+\beta \gamma)
 \\
c (\mathcal{D})\nu^{\rm DCP}((1,0)) &= \mathcal{D}(2+\beta +\gamma)(\alpha  + \delta)
+(\gamma+1)\left[\lambda(\alpha + \delta)
+ \delta(\alpha + \beta  +\gamma  + \delta+2)\right]
 \\
c(\mathcal{D}) \nu^{\rm DCP} ((0,1)) &= \mathcal{D}(2+\beta +\gamma)(\alpha  + \delta)
+(\beta+1)\left[\lambda(\alpha + \delta)
+ \alpha(\alpha + \beta  +\gamma  + \delta+2)\right]
\\ 
c(\mathcal{D})\nu^{\rm DCP}((1,1)) &= \mathcal{D} (\alpha + \delta)(\alpha+ \delta+2 \lambda)
+\alpha\delta(2+\alpha +\beta + \gamma+\delta)\\&\qquad+
\lambda(\alpha + \delta)(\alpha + \delta+\lambda+1)
+\lambda(\alpha \beta + \gamma\delta),
\end{cases}
\end{equation}
where 
\begin{align}\label{cforsolkonno}
c(\mathcal{D}) &=   \mathcal{D}\left[(\alpha + \beta+\gamma +\delta+2)^2
+2\lambda(\alpha +\delta)\right]\\
 &+\left[\alpha + \beta+\gamma +\delta+1+(\alpha +\gamma)(\beta +\delta)\right](\alpha + \beta+\gamma +\delta+2)\nonumber\\
&
+\lambda(\alpha+\delta)(\alpha+\delta+\lambda) 
+\lambda\left[(\alpha+\delta+1)(\beta+\gamma+2)+(\alpha +\gamma)(\beta +1)
 +(\beta+\delta)(\gamma+1)\right] \nonumber
\end{align}

In both cases, setting $\alpha = \delta = \lambda$ and $\gamma = \beta = 0$, 
our results correspond to the ones in \cite{Konno1994}, as expected. 
This strategy explicitly characterizes the stationary measure 
but only works when the size of the system is small. 

From \eqref{sol-like-konno},
recalling the notation 
    $\rho_1^{\rm DCP}(1) 
    = \mathbb{E}_{\nu^{\rm DCP}}\big[ \e_1\big]=:x$, 
    $\rho_1^{\rm DCP}(2) = \mathbb{E}_{\nu^{\rm DCP}}\big[ \e_2\big]=:y$ 
    and $\rho_2^{\rm DCP}(1,2) = \mathbb{E}_{\nu^{\rm DCP}}\big[ \e_1\e_2\big]=:z$, 
    we deduce the following:
\begin{equation}\label{explicit_three_points}
    \begin{cases}
         c(\mathcal{D}) x &= c (\mathcal{D})\nu^{\rm DCP}((1,0)) +
         c(\mathcal{D})\nu^{\rm DCP}((1,1))\\
         &=\mathcal{D}(\alpha+\delta)(2+\beta+\gamma+\alpha+\delta + 2\lambda)
          + (\alpha+\delta)(\lambda+ 2+ \alpha+\delta+\gamma)\lambda \\
        &\qquad  + (\alpha\beta + \delta\gamma)\lambda
         + (\beta+\gamma+\alpha+\delta +2)\delta(\gamma+\alpha+1)\\
         c(\mathcal{D}) y & = c(\mathcal{D}) \nu^{\rm DCP} ((0,1))+
       c(\mathcal{D})\nu^{\rm DCP}((1,1))\\  
         & = \mathcal{D}(\alpha+\delta)(2+\beta+\gamma+\alpha+\delta + 2\lambda)
         + (\alpha+\delta)(\lambda+ 2+\alpha+\delta +\beta)\lambda \\
        &\qquad    + (\alpha\beta + \delta\gamma)\lambda
+ (\beta+1+\delta)\alpha(2+\delta+\gamma+\beta+\alpha)\\
           c(\mathcal{D}) z & = c(\mathcal{D})\nu^{\rm DCP}((1,1))\\
         &=\mathcal{D} (\alpha + \delta)(\alpha+ \delta+2 \lambda)
+\alpha\delta(2+\alpha +\beta + \gamma+\delta)\\&\qquad+
\lambda(\alpha + \delta)(\alpha + \delta+\lambda+1)
+\lambda(\alpha \beta + \gamma\delta).
   \end{cases}
\end{equation}

\begin{rem}[Fokker-Planck approach]\label{fokkerplanckapproach}
We could have derived directly $x,y,z$ by solving the equation
$\int\mathcal{L}^{\rm DCP}fd\nu^{\rm DCP}=0$ for 
$f(\eta)=\eta_1,\,f(\eta)=\eta_2,\,f(\eta)=\eta_1\eta_2$.
 \end{rem}
 
Then, taking the limit $\mathcal{D}\to\infty$, 
the solutions in \eqref{sol-like-konno} become
\begin{equation}\label{sol-like-konno_infty}
    \begin{cases}
c(\infty)\nu^{\rm DCP}((0,0))(\infty)&=  (\beta + \gamma + 2)^2 
 \\
c(\infty)\nu^{\rm DCP}((1,0))(\infty) &= (2+\beta +\gamma)(\alpha  + \delta)
 \\
c(\infty)\nu^{\rm DCP} ((0,1)) (\infty)&= (2+\beta +\gamma)(\alpha  + \delta)
\\
c(\infty)\nu^{\rm DCP}((1,1))(\infty) &=  (\alpha + \delta)(\alpha+ \delta+2 \lambda)
\end{cases}
\end{equation} 
where 
\begin{equation}\label{cinftyforsolkonno}
c(\infty)=(\alpha + \beta+\gamma +\delta+2)^2+2\lambda(\alpha +\delta)
\end{equation} 
Due to the infinite diffusion,
 $\nu^{\rm DCP}((1,0))(\infty)=\nu^{\rm DCP} ((0,1)) (\infty)$, 
 so this quantity only depends on the total number of particles, 
 regardless of their position.

Finally, taking the limit $\mathcal{D}\to\infty$, the solutions in \eqref{explicit_three_points} simplifies to
\begin{equation}\label{explicit_three_points-infty}
    \begin{cases}
        & c(\infty)x (\infty)
        = (\alpha+\delta)(2+\beta+\gamma+\alpha+\delta + 2\lambda)\\
        & c(\infty)y (\infty)
        = (\alpha+\delta)(2+\beta+\gamma+\alpha+\delta + 2\lambda)\\
        & c(\infty)z (\infty)= (\alpha + \delta)(\alpha+ \delta+2 \lambda) \;.
   \end{cases}
\end{equation} 
Notice that, as remarked at the end of Appendix \ref{appN2} 
for the absorption probabilities of one dual particle starting 
in site $1$ or in site $2$, also the corresponding correlations 
$\rho_1^{\rm DCP}(1)$ and $\rho_1^{\rm DCP}(2)$ are the same. 

\subsection{Absorption probabilities} \label{appN2}
In this subsection, related to Section \ref{subsec:corfu},
we explicitly compute the absorption probabilities
of the dual process
given in Proposition \ref{any_point_label}  for $N=1$, 
which is simpler because there is neither diffusion nor birth effects, 
then for $N=2$.
Recall that for the sake of simplicity, 
 we took $\beta=\delta=0$ in Section \ref{subsec:corfu}, 
 and we will also do this in what follows.
\subsubsection*{$N=1$ site} 
We perform a conditioning on the first step of the dual dynamics. 
The behavior of the bulk is conservative as the contact dynamics has no effect 
here (no possible birth). This provides a finite sum in the expression 
of the correlation function \eqref{lcorrelations2} obtained through duality, 
allowing us to compute it conveniently. By \eqref{one_point_corr}, we have
\begin{equation*}
    \mathbb{E}_{\nu^{\rm DCP}}[1-\e_1]
    =1-\rho_1^{\rm DCP}(1)= \mathbb{P}_{\delta_1}\Big(\xi_0 (\infty ) 
    = 0 \Big) c_{-}^0 + \mathbb{P}_{\delta_1} \Big(\xi_0 (\infty ) = 1 \Big)c_{-}.
\end{equation*}
To compute the absorption probabilities 
$\mathbb{P}_{\delta_1}\Big(\xi_0 (\infty ) = 0 \Big)$ and 
$\mathbb{P}_{\delta_1} \Big(\xi_0 (\infty ) = 1 \Big)$,
we perform a conditioning on the first possible outcome of the process, 
starting from a particle at site $1$:
\begin{align*}
    \mathbb{P}_{\delta_1} \Big(\xi_0 (\infty ) = 0 \Big) 
    & =  \mathbb{P}_{\delta_1} \Big(\xi_0 (\infty ) = 0 \mid 
    \text{left jump} \Big)  \mathbb{P}_{\delta_1}(\text{left jump}) \\ 
    &+ \mathbb{P}_{\delta_1} \Big( \xi_0 (\infty ) = 0 \mid 
    \text{it dies} \Big)  \mathbb{P}_{\delta_1}( \text{it dies} ) \\ 
    & = 
    0 \cdot \frac{\alpha + \gamma }{\alpha + \gamma + 1 } 
    + 1 \cdot \frac{1}{\alpha + \gamma + 1 } \\ & = \frac{1}{\alpha + \gamma + 1 }
\end{align*}
while if the particle is absorbed, 
\[\mathbb{P}_{\delta_1} \Big(\xi_0 (\infty )
 = 1 \Big) = 1- \mathbb{P}_{\delta_1} \Big(\xi_0 (\infty ) = 0 \Big) 
 = \frac{\alpha + \gamma }{\alpha + \gamma + 1 } \;.\]
Therefore,
\begin{equation*}
    \rho_1^{\rm DCP}(1)= \frac{\alpha}{\alpha + \gamma +1},
\end{equation*}
\subsubsection*{$N=2$ sites} 
We now
 explicitly compute the absorption probabilities 
 of the dual process when $N=2$, as well as their fast stirring limits.
 Notice that, in computing correlations,  this is the first non trivial case. 
 In addition to the dynamics for one bulk site, 
the effect of the exclusion process (with rate $\mathcal{D}$) as well as 
the effect of the contact process (with rate $\lambda$) are present, 
therefore we have an infinite sum in equation \eqref{lcorrelations2}, 
as expressed below in \eqref{corr_for_three}. Recall the expressions 
of the correlation functions \eqref{ell_point_corr_def} provided by 
the duality relations \eqref{one_point_corr}, and \eqref{lcorrelations2}:
\begin{equation}\label{corr_for_three}
    \begin{cases}
\rho_1^{\rm DCP}(1) &= 1-\sum_{k\geq 0} \mathbb{P}_{\delta_1}\big(\xi_0(\infty)=k\big)c_{-}^k\\
\rho_2^{\rm DCP}(1,2) &= 1-\sum_{k\geq 0} \mathbb{P}_{\delta_{1,2}}\big(\xi_0(\infty)=k\big)c_{-}^k\\
\rho_1^{\rm DCP}(2) &= 1-\sum_{k\geq 0} \mathbb{P}_{\delta_2}\big(\xi_0(\infty)=k\big)c_{-}^k \;.
    \end{cases}
\end{equation}
To compute the absorption probabilities appearing in the infinite sums, 
as in the case $N=1$ , we condition on the first possible outcome 
of the process, starting from a particle at site $1$, $2$ or at both sites. 
For $k\geq 0$, we define
\begin{equation*}
    x_1^k:=\mathbb{P}_{\delta_1}\Big[\xi_0(\infty)=k \Big],
    ~~~x_2^k:= \mathbb{P}_{\delta_{1,2}}\Big[\xi_0(\infty)=k \Big],
    ~~\text{and}~~x_3^k := \mathbb{P}_{\delta_2}\Big[\xi_0(\infty)=k \Big],
\end{equation*}
the probabilities that $k \geq 0$ dual particles are absorbed in the left 
reservoir starting from a particle at site $1$ ($x_1^k$), 
at site $2$ ($x_3^k$) or at both sites ($x_2^k$). 
To alleviate notation we do not write the dependence on the model parameters 
$\alpha, \gamma, \lambda$ and $\mathcal{D}$ on these absorption probabilities. 
For $k\geq 0$, define the column vector
\begin{equation}\label{Ck}
    C_k=\begin{pmatrix} 
       x_1^k\\ 
        x_2^k\\ 
        x_3^k  
\end{pmatrix} .
\end{equation}
\begin{itemize}
    \item For $k=0$, the conditioning on the first step leads to the following 
    system:
    \begin{equation}\label{syst:k0} 
   \left\{
    \begin{array}{ll}
        x_1^0 -\frac{\lambda}{\alpha + \gamma + \mathcal{D}+ \lambda +1}~x_2^0 
        - \frac{\mathcal{D}}{\alpha + \gamma + \mathcal{D}+ \lambda +1}
         ~x_3^0=\frac{1}{\alpha + \gamma + \mathcal{D}+ \lambda +1}\\\\
        \frac{-1}{\alpha + \gamma+2}~x_1^0 
        + x_2^0 - \frac{1}{\alpha + \gamma +2}~x_3^0=0\\\\
        \frac{-\mathcal{D}}{\mathcal{D}+ \lambda +1}
        ~x_1^0-\frac{\lambda}{\mathcal{D}+ \lambda +1}
        ~ x_2^0 + x_3^0 = \frac{1}{\mathcal{D}+ \lambda +1},
        \end{array}
\right.
\end{equation}
that is,
 \begin{equation}\label{syst:k0-bis}
  M_{2}C_0
  = \begin{pmatrix} 
       \frac{1}{\alpha + \gamma + \mathcal{D}+ \lambda +1} \\
        0\\ 
        \frac{1}{\mathcal{D}+ \lambda +1}
      \end{pmatrix} ,
\end{equation} 
where
 \begin{equation}\label{M2gamma}
  M_{2} :=  \begin{pmatrix} 
         1& \frac{- \lambda}{\alpha + \gamma + \mathcal{D} + \lambda + 1} 
         & \frac{-\mathcal{D}}{\alpha + \gamma+ \mathcal{D} + \lambda + 1}  \\\\ 
         \frac{-1}{\alpha + \gamma +2}& 1& \frac{-1}{\alpha + \gamma +2}\\\\ 
         \frac{-\mathcal{D}}{\mathcal{D}+\lambda +1} 
         &\frac{-\lambda}{\mathcal{D}+\lambda +1} & 1
   \end{pmatrix} \;.
 \end{equation} 
We have
\begin{equation}\label{dgamma} 
d:=  \text{det} \ M_2 = \dfrac{\widetilde d}
{(\mathcal{D} + \lambda + 1)(\alpha + \gamma + 2)(\mathcal{D} 
+ \alpha + \gamma + \lambda + 1)},
   \end{equation}
where 
\begin{equation}\label{dgammatilde} 
\begin{split}
        \widetilde{d} & :=
    \mathcal{D}A
   +(\alpha+\gamma)\left[(\alpha+\gamma+\lambda+2)(\lambda+1)+1\right]
   +2(\lambda+1),
\end{split}
\end{equation}  
with 
\begin{equation}\label{Afortilded}
A:=  (\alpha+\gamma+\lambda+2)^2-\lambda(\lambda+4)
= (\alpha+\gamma+2)^2 +2\lambda(\alpha+\gamma)
\end{equation}  
So $\det M_2\neq 0$, and solving \eqref{syst:k0-bis} by inverting $M_2$ yields
\begin{equation} \label{system0}
   \left\{
    \begin{array}{ll}
      x_1^0 = \widetilde{d}^{-1}(2\mathcal{D} + \lambda + 1)(\alpha + \gamma + 2)
        \\
     x_2^0= \widetilde{d}^{-1} \big[4\mathcal{D}
      + \alpha + \gamma + 2\lambda + 2 \big]
        \\
     x_3^0 =\widetilde{d}^{-1} (2\mathcal{D} + \alpha + \gamma + \lambda + 1)
     (\alpha + \gamma + 2).
        \end{array}
\right.
\end{equation} 
\item For $k=1$, the same conditioning leads to
\begin{equation}\label{syst:k1} 
   \left\{
    \begin{array}{ll}
        x_1^1 - \frac{\lambda}{\alpha + \gamma + \mathcal{D}+ \lambda +1} 
        ~ x_2^1 - \frac{\mathcal{D}}{\alpha + \gamma + \mathcal{D}+ \lambda +1}~ x_3^1=\frac{\alpha+ \gamma}{\alpha + \gamma + \mathcal{D}+ \lambda +1} \\\\
        \frac{-1}{\alpha + \gamma +2}~x_1^1 + x_2^1 - \frac{1}{\alpha + \gamma +2}~x_3^1 =\frac{\alpha+ \gamma}{\alpha+ \gamma +2} ~ x_3^0\\\\
        \frac{-\mathcal{D}}{\mathcal{D}+ \lambda +1}~ x_1^1 -\frac{\lambda}{\mathcal{D}+ \lambda +1} ~ x_2^1 +x_3^1 = 0,
        \end{array}
\right.
\end{equation}
that is,
 \begin{equation}\label{syst:k1-bis}
  M_2 C_1
  = \begin{pmatrix} 
        \frac{\alpha+ \gamma}{\alpha + \gamma + \mathcal{D}+ \lambda +1}\\ \\
        \frac{\alpha+ \gamma}{\alpha+ \gamma +2} ~ x_3^0\\ \\
        0
      \end{pmatrix} 
\end{equation}  
Again, as $\det M_2\neq 0$, we can invert this and, to lighten 
the formulas of the solution, we introduce the notation   
\begin{eqnarray}
        \varphi&:=
        \mathcal{D}B
        +  \lambda^2+\lambda(\alpha + \gamma + 2) 
        +\alpha + \gamma +  1, \label{def-phi-for-n=2}
\end{eqnarray} 
where
\begin{equation}\label{Bforphi}
B:= (\alpha + \gamma + 2\lambda  + 2).
\end{equation} 
We get:
\begin{equation}\label{system1} 
   \left\{
    \begin{array}{ll}
        x_1^1 = \widetilde{d}^{-1}  (\alpha + \gamma) (\alpha + \gamma +2)\left[ 
\mathcal{D} +1+
 \dfrac{\lambda( \alpha + \gamma + 1)}{\alpha + \gamma + 2} 
         \right.
        \\ \left. \qquad \qquad 
        + \lambda \widetilde{d}^{-1} (2\mathcal{D} + \lambda + 1)(2\mathcal{D} + \alpha + \gamma + \lambda + 1) \right]
        \\
        x_2^1= \widetilde{d}^{-1} (\alpha + \gamma) (\alpha + \gamma +2)
        \left[ 
        \dfrac{2\mathcal{D} + \lambda + 1}{\alpha + \gamma + 2}    
           + \widetilde{d}^{-1} \varphi(2\mathcal{D} + \alpha + \gamma +1+\lambda)\right]
       \\\\
        x_3^1 = \widetilde{d}^{-1}  (\alpha + \gamma) (\alpha + \gamma +2)\left[ 
\mathcal{D} + \dfrac{\lambda}{\alpha + \gamma + 2}        
        +\widetilde{d}^{-1} \lambda(2\mathcal{D}+ \lambda + \alpha + \gamma  + 1)^2\right] .
        \end{array}
\right.
\end{equation}
\item For $k=2$, the same conditioning leads to
\begin{equation}\label{syst:k2}  
   \left\{
    \begin{array}{ll}
        x_1^k - \frac{\lambda}{\alpha+ \gamma + \mathcal{D}+ \lambda +1} 
        ~ x_2^k - \frac{\mathcal{D}}{\alpha+\gamma +\mathcal{D}+ \lambda +1}
         ~x_3^k=0\\\\
        \frac{-1}{\alpha+ \gamma +2}~x_1^k + x_2^k 
        - \frac{1}{\alpha+{\gamma}+2}~x_3^k 
        =\frac{\alpha+ \gamma}{\alpha+ \gamma +2} ~ x_3^{k-1}\\\\
       \frac{-\mathcal{D}}{\mathcal{D}+ \lambda +1}~ x_1^k 
       -\frac{\lambda}{\mathcal{D}+ \lambda +1} ~ x_2^k +x_3^k = 0
        \end{array}
\right.
\end{equation}
that is,
 \begin{equation}\label{syst:k2-bis}
  M_2 C_k
  = \begin{pmatrix} 
        0\\ 
        \frac{\alpha+ \gamma}{\alpha+ \gamma +2} ~ x_3^{k-1}\\
        0
      \end{pmatrix} 
\end{equation} 
Again, inverting $M_2$, using the definition of $\varphi$ given in 
\eqref{def-phi-for-n=2}, and introducing 
\begin{eqnarray}
\psi&:=& \mathcal{D}(\alpha +  \gamma+2) + \lambda +  \widetilde{d}^{-1} \lambda(\alpha + \gamma + 2)(2\mathcal{D} + \alpha + \gamma + \lambda + 1)^2
\label{def-psi-for-n=2}
\end{eqnarray} 
we get:
\begin{equation}\label{system2}  
   \left\{
    \begin{array}{ll}
        x_1^2 = \widetilde{d}^{-2} \lambda(\alpha + \gamma)^2
        (2\mathcal{D}+\lambda+1)\psi
        \\
        x_2^2= \widetilde{d}^{-2} (\alpha + \gamma)^2\varphi\psi 
\\
        x_3^2 =  \widetilde{d}^{-2}\lambda(\alpha + \gamma)^2
        (2\mathcal{D} + \alpha + \gamma + \lambda + 1)\psi.
        \end{array}
\right.
\end{equation}
\item For $k=3$, the same reasoning as for the cases above yields
\begin{equation}\label{system3}  
   \left\{
    \begin{array}{ll}
        x_1^3 = \widetilde{d}^{-3}\lambda^2(\alpha + \gamma)^3
       (2\mathcal{D}+\lambda+1)(2\mathcal{D}+\alpha + \gamma + \lambda +1) \psi
        \\
        x_2^3= \widetilde{d}^{-3}\lambda(\alpha + \gamma)^3
        (2\mathcal{D}+\alpha+\gamma+\lambda+1)
        \varphi\psi 
        \\
       x_3^3 = \widetilde{d}^{-3}\lambda^2(\alpha + \gamma)^3
        (2\mathcal{D} + \alpha + \gamma + \lambda + 1)^2\psi.
        \end{array}
\right.
\end{equation}
\item For $k > 3$, using the same conditioning and reasoning recursively, we find that
\begin{equation}\label{systemk}  
   \left\{
    \begin{array}{ll}
        x_1^k = \widetilde{d}^{-k}\lambda^{(k-2)}(\alpha + \gamma)^k 
      (2\mathcal{D}+\lambda+1)(2\mathcal{D}+\alpha + \gamma + \lambda +1)^{(k-3)}\varphi\psi
        \\
        x_2^k= \widetilde{d}^{-k}\lambda^{(k-3)}(\alpha + \gamma)^k (2\mathcal{D}+\alpha+\gamma+\lambda+1)^{(k-3)}\varphi^2\psi
\\
        x_3^k = \widetilde{d}^{-k}\lambda^{(k-2)}(\alpha + \gamma)^k 
      (2\mathcal{D} + \alpha + \gamma + \lambda + 1)^{(k-2)}\varphi\psi.
        \end{array}
\right.
\end{equation}
\end{itemize}

\noindent 
 We now consider the fast stirring limit of the absorption probabilities just found.
For all $k \geq 0 $ and $i=1,2,3$ we denote 
\begin{equation*}
    x_i^k (\infty) := \lim_{\mathcal{D}\to\infty} x_i^k \;.
\end{equation*}
The solutions obtained in \eqref{system0}, \eqref{system1}, 
\eqref{system2}, \eqref{system3}, \eqref{systemk} simplify 
(recall the definitions \eqref{Afortilded} of $A$
and \eqref{Bforphi} of $B$) to
\begin{itemize}
\item For $k=0$,
\begin{equation} \label{system0-infty}
   \left\{
    \begin{array}{ll}
     x_1^0 (\infty) = x_3^0(\infty) = 2 A^{-1}(\alpha + \gamma + 2)
        \\
     x_2^0(\infty)= 4A^{-1}
        \end{array}
\right.
\end{equation} 
\item For $k=1$, we get:
\begin{equation}\label{system1-infty} 
   \left\{
    \begin{array}{ll}
        x_1^1(\infty) =  x_3^1 (\infty)= A^{-1}(\alpha + \gamma) (\alpha + \gamma +2)
        \left[1 + 4A^{-1}\lambda\right]
        \\
        x_2^1(\infty)= A^{-1} 2(\alpha + \gamma) 
        \left[ 
        1 + A^{-1}B(\alpha + \gamma +2) \right]
        \end{array}
\right.
\end{equation}
\item For $k=2$, 
we get:
\begin{equation}\label{system2-infty}  
   \left\{
    \begin{array}{ll}
        x_1^2 (\infty)= x_3^2(\infty) =
        2A^{-2}\lambda(\alpha + \gamma)^2(\alpha + \gamma +2) \left[1 + 4A^{-1}\lambda \right]
        \\
        x_2^2(\infty)= B(2\lambda)^{-1}x_1^2 (\infty)
        \end{array}
\right.
\end{equation}
\item For $k=3$, the same reasoning as for the cases above yields
\begin{equation}\label{system3-infty}  
   \left\{
    \begin{array}{ll}
        x_1^3 (\infty)= x_3^3 (\infty)= 4A^{-3}\lambda^2(\alpha + \gamma)^3
        (\alpha + \gamma +2) \left[1 + 4A^{-1}\lambda \right]
        \\
        x_2^3(\infty)= B(2\lambda)^{-1} x_1^3 (\infty)
        \end{array}
\right.
\end{equation}
\item For $k > 3$, we find that
\begin{equation}\label{systemk-infty}  
   \left\{
    \begin{array}{ll}
        x_1^k(\infty) = x_3^k (\infty)= A^{-k}(2\lambda)^{k-2}B(\alpha + \gamma)^k(\alpha + \gamma +2) \left[1 + 4A^{-1}\lambda \right]
        \\
        x_2^k(\infty)= B(2\lambda)^{-1} x_1^k(\infty)
        \end{array}
\right.
\end{equation}
\end{itemize}
These formulas induce simple expressions for the values \eqref{corr_for_three} 
when $\mathcal{D}\to\infty$, where in particular for all 
$k \geq 0$, $x_1^k(\infty) = x_3^k(\infty)$, 
that is, the absorption probability to the left reservoirs of one particle at site $1$ 
is the same as the one of a particle at site $2$.
\\
\\
\noindent{\bf Acknowledgements.} 
We would like to thank the referee for providing helpful comments.
This work was possible thanks to the conferences we attended together: 
Rencontres de Probabilit\'es \`a Rouen, PSPDE XI, Inhomogeneous Random Systems.
C.F. thanks MAP5 lab., and  E.S. thanks U. Lisboa for hospitality and support. 
C.F. is member of Gruppo Nazionale per la Fisica Matematica of Istituto 
Nazionale di Alta Matematica (INdAM) and acknowledges the European
 Union - Next Generation EU - Grant PRIN 2022B5LF52.
G.M.S. acknowledges financial support by FCT (Portugal) through project UIDB/04459/2020, 
doi 10-54499/UIDP/04459/2020
and by the FCT Grant 2020.03953.CEECIND.

\bibliographystyle{plain}
\bibliography{biblio.bib}

\end{document}